\documentclass[10pt]{amsart}
\usepackage[pagewise, mathlines]{lineno}
\usepackage[british]{babel}
\usepackage{hhline}
\usepackage[utf8]{inputenc}
\usepackage{lmodern}
\usepackage{pifont,subfigure,graphicx}
\usepackage[tableposition=top]{caption}
\usepackage{longtable}
\usepackage[colorlinks,citecolor=blue,urlcolor=blue]{hyperref}
\usepackage{url}
\usepackage{amsthm,amsmath,amssymb,amsfonts,mathrsfs,amsfonts,amstext,amsopn}
\usepackage{mathtools}
\usepackage{tikz,tikz-cd}
\usepackage{adjustbox}
\usepackage{texdraw}
\usepackage{extarrows}
\usepackage{bbm}
\usepackage{enumitem}
\usepackage{xcolor}
\usepackage[all]{xy}
\usepackage{pgfplots}
\newcounter{tmp}

\pgfplotsset{compat = newest}
\usetikzlibrary{decorations.fractals}
\usetikzlibrary{cd}
\usetikzlibrary{shapes.geometric}
\numberwithin{equation}{section}
\theoremstyle{plain}
\newtheorem{theorem}{Theorem}[section]
\newtheorem{lemma}[theorem]{Lemma}
\newtheorem{corollary}[theorem]{Corollary}
\newtheorem{proposition}[theorem]{Proposition}

\theoremstyle{definition}
\newtheorem{definition}[theorem]{Definition}
\newtheorem{example}[theorem]{Example}
\theoremstyle{remark}
\newtheorem{remark}[theorem]{Remark}

\newcommand{\lt}{\left}
\newcommand{\rt}{\right}

\newcommand{\e}{\textnormal{e}}

\newcommand{\R}{\mathbb{R}}

\newcommand{\bT}{\mathbb{T}}

\newcommand{\N}{\mathbb{N}}

\newcommand{\cB}{\mathcal{B}}
\newcommand{\cC}{\mathcal{C}}

\newcommand{\cF}{\mathcal{F}}
\newcommand{\cI}{\mathcal{I}}

\newcommand{\cM}{\mathcal{M}}
\newcommand{\cN}{\mathcal{N}}
\newcommand{\cP}{\mathcal{P}}
\newcommand{\cQ}{\mathcal{Q}}
\newcommand{\cS}{\mathcal{S}}

\newcommand{\T}{\mathbb{T}}

\newcommand{\Diam}{\textrm{Diam\,}}

\newcommand{\sT}{\mathbb{T}}

\newcommand{\Lip}{{\sf Lip}}
\newcommand{\BL}{{\sf BL}}
\newcommand{\TV}{{\sf TV}}
\newcommand{\rd}{\mathrm{d}}

\newcommand\sbullet[1][.5]{\mathbin{\vcenter{\hbox{\scalebox{#1}{$\bullet$}}}}}
\allowdisplaybreaks

\begin{document}

\title[MFL of co-evolutionary {signed} heterogeneous networks]
  {Mean field limits of co-evolutionary\\ {signed}  heterogeneous networks}
\author{Marios Antonios Gkogkas $^1$}
\author{Christian Kuehn $^{1,2}$}
\author{Chuang Xu $^{1,3}$}

\email{marios.gkogkas@yahoo.de}
\email{ckuehn@ma.tum.de}
\email{chuangxu@hawaii.edu (Corresponding author)}

\address{$^1$
Department of Mathematics\\
Technical University of Munich, Munich\\
Garching bei M\"{u}nchen\\
85748, Germany.}
\address{$^2$
Munich Data Science Institute (MDSI)\\
Technical University of Munich, Munich\\
Garching bei M\"{u}nchen\\
85748, Germany.}
\address{$3$
Department of Mathematics\\
University of Hawai'i at M\={a}noa, Honolulu, Hawai'i\\
96822, USA
}


\noindent

\begin{abstract}
 Many science phenomena are modelled as interacting particle systems (IPS) coupled on static networks. In reality, network connections are far more dynamic. Connections among individuals receive feedback from nearby individuals and make changes to better adapt to the world. Hence, it is reasonable to model myriad real-world phenomena as \emph{co-evolutionary (or adaptive) networks}. These networks are used in different areas including telecommunication, neuroscience, computer science, biochemistry, social science, as well as physics, where \emph{Kuramoto-type networks} have been widely used to model interaction among a set of oscillators. In this paper, we propose a rigorous formulation for \emph{limits} of a sequence of co-evolutionary Kuramoto oscillators coupled on heterogeneous co-evolutionary networks, which receive {both positive and negative}  feedback from the dynamics of the oscillators on the networks. We show under mild conditions, the mean field limit (MFL) of the co-evolutionary network exists and the sequence of co-evolutionary Kuramoto networks converges to this MFL. Such MFL is described by solutions of a \emph{generalized} Vlasov equation. We treat the graph limits as {signed} graph measures, motivated by the recent work in [Kuehn, Xu. Vlasov equations on digraph measures, JDE, 339 (2022), 261--349]. 
In comparison to the recently emerging works on MFLs of IPS coupled on \emph{non-co-evolutionary} networks (i.e., static networks or time-dependent networks independent of the dynamics of the IPS), our work seems the first to rigorously address the MFL of a \emph{co-evolutionary} network model. The approach is based on our formulation of a generalization of the co-evolutionary network as a hybrid system of ODEs and \emph{measure differential equations} parametrized by a vertex variable, together with an analogue of the \emph{variation of parameters formula}, as well as the generalized Neunzert's in-cell-particle method developed in [Kuehn, Xu. Vlasov equations on digraph measures, JDE, 339 (2022), 261--349].
\end{abstract}


\maketitle
\tableofcontents
\section{Introduction}

This paper studies the mean field limit (MFL) of the following general \emph{co-evolutionary} Kuramoto network:
\begin{alignat}{2}\label{coevolution-a}
\dot{\phi}_i =&\ \omega_i(t)+\frac{1}{N}\sum_{j=1}^NW_{ij}(t)g(\phi_j-\phi_i),\\\label{coevolution-b}
\dot{W}_{ij} =& -\varepsilon(W_{ij}+h(\phi_j-\phi_i)),
\end{alignat}
where $\phi_i$ is the phase of the $i$-th oscillator, $\omega_i(t)$ is the time-dependent natural frequency of the $i$-th oscillator, $W_{ij}$ is the {signed} coupling weight of the edge between node $i$ and node $j$, $g$ is the coupling function, and $h$ the adaptation rule. We also include a parameter $\varepsilon>0$, which is often assumed to be small in applications so that the particle dynamics is much faster than the time-scale of the network/graph adaptation. In addition, the feedback of the phase on the underlying graph is assumed to be local: The weight function of a given edge depends only on the edge itself  as well as the phases of the two nodes associated with the edge. In particular, when $g$ is a trigonometric function (see Section~\ref{sect-example}), the above model was proposed to describe the dynamics of a co-evolutionary network of FitzHugh-Nagumo neurons coupled through chemical excitatory synapses equipped with plasticity \cite{S01,N03,HNP16,BVSY21} (see also the references therein).

\subsection{Macroscopic limit of interacting particle systems}

Before presenting the main result of this paper, let us first review briefly the literature on macroscopic limits of Kuramoto type networks. One type of macroscopic limit is the so-called \emph{mean field limit} (MFL), i.e., a (weak) limit of the empirical distributions composed of Dirac measures of equal probability mass at the solutions of each node of the network, as the number of nodes of the network tends to infinity.
Heuristically, the MFL captures the statistical dynamics of large networks. The pioneering works date back to as early as the 1970s, by Braun and Hepp \cite{BH77}, by Dobrushin \cite{D79}, and by Neunzert \cite{N84}, where the underlying coupling graph is complete, the interaction kernel is Lipschitz, and the underlying metric induces the weak topology. The techniques have led to a quite complete derivation of the MFL for the Kuramoto model for all-to-all coupling by Lancellotti \cite{L05}. Later, MFL of Kuramoto oscillators on a sequence of dense \emph{heterogeneous} (deterministic or random) graphs with and without Lipschitz continuity were studied, e.g., in \cite{KM18,CM19a,CM19b,K20,GK20,KX21}. Recently, results were extended to \emph{sparse graphs} using different approaches \cite{ORS20,LRW23,K20,GK20,KX21,JPS24}. So far, all the above network models are given on a \emph{static} network. It is worth mentioning that the approaches in \cite{ORS20,LRW23,JPS24} dealing with sparsity are more graph-theoretic, the ones in \cite{KM18} {are}  based on analysis of $\mathcal{L}^p$-functions while restricted to \emph{graphon} type graph limits, the one in \cite{GK20} is more operator-theoretic combined with harmonic analysis techniques, and \cite{KX21} more measure-theoretic. It is noteworthy that results for graph limits of a sequence of intermediate/low density are also covered by the approaches in~\cite{GK20,KX21}. The fast development of mean field theory of IPS coupled on heterogeneous networks owes much to the development of graph limits \cite{BS01,L12,KLS19,BS20}. The MFL of large networks coupled on \emph{static} graphs, when they are viewed as a probability measure, is generally absolutely continuous with respect to {a} certain reference measure (provided the initial distribution is so) \cite{N84,KM18}. The density of the MFL is captured by the solution of a transport type PDE, the so-called \emph{Vlasov equation}, cf.~\cite{N84,KM18,K20,GK20,KX21}.

In contrast to the many aforementioned works on MFLs of IPS on statics networks, few works consider particle systems on a \emph{dynamic} network/graph. In \cite{AP21}, the MFL of a network model characterizing the collective dynamics of moving particles with time-dependent couplings among nodes is investigated. However, these time-dependent graphs satisfy that each node has the same edge weight to all the other nodes. Hence the graph is $N$-dimensional rather than $O(N^2)$-dimensional, and hence this makes it possible to treat the time-dependent weight function as a second component of an lifted particle system parametrized by the node variable. Hence, it reduces to the MFL of an IPS on a static graph, where a classical approach suffices (cf.~\cite{KM18}). However, when the weights are not ``uniform'', rigorous works are lacking. MFLs of IPS were discussed also in \cite{B21}, where kinetic equations for large finite population size were obtained, but the MFL was not rigorously characterized. Letting the number $N$ of population size tend to infinity and the small time scale $\varepsilon$ tend to zero simultaneously for sparse co-evolutionary networks using fast-slow arguments, a non-MFL result was investigated in \cite{BDOZ21}, where the underlying graph is assumed to be independent of the dynamics of the particle system and hence the dynamics of the particle system plus the weights of the underlying graph is decoupled. Moreover, evolution of graphon-valued stochastic processes motivated from genetics was lately investigated in \cite{ADR21}. To our best knowledge, there seems to have been rather rare rigorous work on the MFL of \emph{co-evolutionary} networks, where the dynamics \emph{of} the network depends on the dynamics \emph{on} the network and vice versa.
This paper is a first step to investigate rigorously the MFL of such networks by studying the co-evolutionary model \eqref{coevolution-a}-\eqref{coevolution-b}.

Another type of macroscopic limit frequently encountered in the literature is the so-called \emph{continuum limit}, {defined as a ``pointwise limit'' of the network model.  Essentially, when taking each vertex of the underlying graph of the network as a location in a metric space, as the number of vertices increases to infinity, the ODE modelling the network tends to a (possibly nonlocal due to some potentially long-range interactions) parabolic PDE, a model with a continuum spatial structure. One uses the $\ell_{\infty}$-norm to quantify the distance  between the solution of the ODE and that of the PDE so that the error is measured pointwise by taking an essential supremum over errors of trajectories at all different locations. In contrast, the MFL is a limit in a statistical sense (from the perspective of sampling), which reflects the distribution of solutions of each node (e.g., phases of the oscillators) on the graph/network of a sufficiently large size. For a more precise and clearer description of the continuum limit, we refer the reader to, e.g., a companion work \cite{GKX21}.} Continuum limits of the Kuramoto model were investigated in \cite{M19} on random sparse \emph{static} networks, and more recently investigated in \cite{GKX21} on deterministic \emph{co-evolutionary} networks. We mention that continuum limits of collective dynamics models with ``uniform'' time-varying weights (e.g., the Cucker-Smale model) were also recently studied in \cite{AP21}, which is restricted to the case where the approach for particle systems coupled on static networks remains valid. In contrast, such a restriction was removed in \cite{GKX21}. We refer the reader to \cite{AP24} for a more comprehensive review on the topic of MFL of non-exchangable systems.

\subsection{Highlights of this paper}
Our approach of analytically obtaining the MFL of the Kuramoto-type model \eqref{coevolution-a} rests on the idea that the co-evolutionary system can be decoupled (the weights of the time-dependent graph can be represented in terms of those of the initial graph as well as the added weights from the feedback of the phase state of the IPS). Therefore, we equivalently represent the Kuramoto model as an integral equation coupled on a static initial \emph{{signed}} graph that allows both positive and negative coupling manners. Then one may be able to utilize similar techniques in establishing the MFL of an IPS on a graph limit (e.g., \cite{KM18,GK20,KX21}), except that the graph limit here is allowed to be a graph with signed weights. {Some special care needs to be taken regarding the approximation of the signed graph measures. We prove the well-posedness and approximation of the MFL as solutions to a so-called generalized Vlasov equation.} 

 {Nevertheless, it is not obvious if the absolute continuity of the solution of the MFL still follows from that of the initial distribution, as in the classical case where the underlying graph is positive and static. Indeed, it seems challenging to obtain the absolute continuity of the MFL. A high-level heuristic conjecture is that the underlying mean field process becomes no longer \emph{Markovian} due to the co-evolving nature. For this reason, even if the MFL is absolutely continuous (with respect to a reference measure, e.g., the Lebesgue measure of a Euclidean space), the {corresponding} PDE accounting for the density of the MFL, may not be a transport type PDE on a finite dimensional state space, but rather on an infinite dimensional state space of functions. Technically,  the standard technique by Gronwall inequalities fails  owing to the added effects coming from the co-evolutionary terms. In this context,  as there are second order nonlocal terms in the equation of characteristics (\eqref{integralequation}), a reverse second order Gronwall inequality would be desirable to tackle this issue, which unfortunately fails in general, as shown by a simple example (see Appendix~\ref{appendix-Gronwall}). A special case where the underlying graph evolves but not coevolves with the oscillators (i.e., the adaptation function $h$ is constant) was studied in an earlier version of the manuscript \cite{GKX22}, and it is shown that the absolute continuity of the MFL is preserved provided the initial distribution was absolutely continuous. We leave this general question on absolute continuity of the MFL as well as the PDE accounting for the density for our future study.}

\subsection{Main results and sketch of the proof}

\begingroup
\setcounter{tmp}{\value{theorem}}
\setcounter{theorem}{0} 
\renewcommand\thetheorem{\Alph{theorem}}

Now we present an informal statement of the main result of this paper.

\begin{theorem}
 Under certain conditions, there exists a unique mean field limit of the Kuramoto-type model \eqref{coevolution-a}-\eqref{coevolution-b}, provided the {signed} graph sequence $\{W_{i,j}(0)\}_{N\in\N}$ as well as the sequence of initial empirical measures $\{\frac{1}{N}\sum_{i=1}^N\delta_{\phi_i(0)}\}_{N\in\N}$ converge in a suitable weak sense. More precisely, the mean field limit is a weak limit of the sequence of time-dependent empirical measures $\{\frac{1}{N}\sum_{i=1}^N\delta_{\phi_i(t)}\}_{N\in\N}$ for $t$ over a finite time interval.
\end{theorem}

For a detailed and precise statement, see Theorem~\ref{thm-existence-VE} (well-posedness) and Theorem~\ref{th-approx} (approximation).

We are going to apply the result to investigate the MFL of {an example of a binary tree} Kuramoto-type networks, where the sequence of initial underlying graphs are sparse (see Section~\ref{sect-example}).

\endgroup

There are basically five steps to achieve the well-posedness as well as the approximation of the MFL of the co-evolutionary Kuramoto model ({see Figure~\ref{f1:outline} for a flow chart on how these steps  and the relevant results are linked}).
\smallskip

\noindent \underline{Step I.} Formulation of a generalized co-evolutionary Kuramoto network {(\eqref{characteristic-Eq-1}-\eqref{characteristic-Eq-3})}. We treat {signed} graph limits as measure-valued functions (so-called {signed} digraph measures, see Definition~\ref{definition-DGM}), and hence \eqref{coevolution-a}-\eqref{coevolution-b} can be regarded as special cases of a hybrid system  {(\eqref{characteristic-Eq-1}-\eqref{characteristic-Eq-3})} of ODEs and measure differential equations (MDEs). To do this, we introduce the derivative of a family of parameterized (by `time') measures in the Banach space of all finite signed measures equipped with the total variation norm. Well-posedness of the hybrid system is then obtained (Theorem~\ref{thm-characteristic-1}), by applying the Banach contraction principle to a suitable complete metric space, in the spirit of the standard Picard-Lindel\"{o}f iteration for ODEs.
\smallskip

\noindent \underline{Step II.} We establish an analogue of the \emph{variation of parameters formula} for the MDE,  in analogy to the Duhamel's principle for PDEs in finite dimensional state spaces \cite{J82} {(Proposition~\ref{prop-equivalent})}. Thanks to this formula, we can decouple the dynamics of the oscillators from the dynamic graph measures, and successfully reduce the hybrid system {(\eqref{characteristic-Eq-1}-\eqref{characteristic-Eq-3})}  to a one-dimensional integral equation (IE) {(\eqref{integralequation})} indexed by the vertex variable coupled on the prescribed initial graph measures as well as prescribed time-dependent measure valued functions (i.e., the MFL to be determined). Such IE can be viewed as the equation of characteristics \cite{KX21}.
\smallskip

\noindent \underline{Step III.} We establish continuous dependence of the solutions to the IE  {(Proposition~\ref{prop-semiflow})}, and show the existence and the Lipschitz regularity of the semiflow forward in time generated by the IE  {(Proposition~\ref{prop-inverse-map})}, using a Gronwall type inequality. 
\smallskip

\noindent \underline{Step IV.} We construct a \emph{generalized Vlasov equation} (VE) {(\eqref{Fixed})}\---a fixed point equation induced by the pushforward of the semiflow generated by the IE, in the sense of Neunzert \cite{N84}. Then we are able to apply the Banach contraction principle again to show the unique existence of solutions to the generalized VE {(Theorem~\ref{thm-existence-VE})}. 
\smallskip

\noindent \underline{Step V.} We establish the approximation by the MFL (solution to the generalized VE) of the empirical distributions generated by a sequence of ODEs like \eqref{coevolution-a}-\eqref{coevolution-b} (Theorem~\ref{th-approx}). To do this, we rely on the continuous dependence of the semiflow (Propositions~\ref{prop-continuousdependence} and ~\ref{prop-sol-fixedpoint}).  Such an approximation result is also based on the discretization of a given initial signed digraph measure as well as the initial distribution (i.e., the initial condition of the generalized VE), {which further is a consequence with calibration (due to that the digraph measure may not be positive)}  of the recent results of probability measures by finitely supported discrete measures with equal mass on each point of the support (so-called \emph{uniform approximation} \cite{XB19}, or \emph{deterministic empirical approximation} \cite{C18,BJ22}) \cite{XB19,C18,BJ22}.

\bigskip
\begin{figure}[h]
\adjustbox{scale=.7,center}{
    \begin{tikzcd}[color=black,scale=.5,cells={nodes={ellipse, draw, thick, inner xsep=0pt},row sep=1cm, column sep=-3cm},
                    ]
\text{A sequence of digraphs $G^N(t)$} \arrow{rr}{\text{$N\to\infty$}} \arrow[d, "\text{coupling}" left, ->,shift right=2cm, start anchor={-40}, end anchor={40}] \arrow[d, "\text{adaptation}", <-,shift right=1cm,start anchor={-40}, end anchor={40}]
    &   &   \text{Digraph measure $\eta_t$} \arrow{d}  \\
    \text{Kuramoto network \eqref{coevolution-a}-\eqref{coevolution-b}} \arrow{rr}{\text{$N\to\infty$}}  \arrow{dr}{\text{MFL}} & & \begin{tabular}{c}
A hybrid system \eqref{characteristic-Eq-1}-\eqref{characteristic-Eq-3}\\ which is unified into \eqref{integralequation},\\ the equation of characteristics
\end{tabular}  \arrow{dl} \\
 & \text{Generalized Vlasov equation \eqref{Fixed}}    &               
    \end{tikzcd}
}
\caption{Schematic diagram of the approach of deriving MFL.}\label{f1:outline}
\end{figure}

\subsection{Organization of this paper}

We first provide necessary preliminaries on measure theory to introduce MDEs in Section~\ref{sect-preliminaries}. Next, we propose a generalized co-evolutionary Kuramoto model and investigate its well-posedness in Section~\ref{sect-characteristic}. We construct a generalized Vlasov equation and study the well-posedness of this equation in Section~\ref{sect-Vlasov}, and address the approximation of its solutions in Section~\ref{sect-th-approx}. 
 To demonstrate the applicability of the main results, we provide {a simple example} in Section~\ref{sect-example}. Proofs of the main results are given in Section~\ref{sect-proof}.  {Further in-depth discussions on limitations and extensions of the results and the approach of this paper together with some outlooks including other types of evolutionary network models are provided in Section~\ref{sect-discussion}.} 

\section{Preliminaries}\label{sect-preliminaries}

A table of notation is provided in Appendix~\ref{appendix-Notation} for the reader to check back and forth whenever necessary through reading this paper. {Let $(Y,d_Y)$ be a complete metric space.} For $\upsilon\in\mathcal{M}(Y)$, define the \emph{total variation norm}\footnote{This definition we give is twice as large as the standard one. We use it for the ease of exposition, particularly for its induced metric in comparison with the bounded Lipschitz metric, {in the light of their supremum representation}.}
 \begin{equation*}\label{eq:total-variation}\|\upsilon\|_{\TV}\coloneqq\sup_{f\in \mathcal{B}_1(Y)}\int f\rd\upsilon=\upsilon^+(Y)+\upsilon^-(Y),
 \end{equation*}
where $\upsilon^+$ and $\upsilon^-$ are the positive and negative part of $\upsilon$, respectively, owing to the Hahn decomposition \cite{B07}.  Recall from  \cite[Chapter 8]{B07} that $\cM(Y)$ with the total variation norm $\|\cdot\|_{\TV}$ is a Banach space. {Let $d_{\TV}$ be the metric induced by $\|\cdot\|_{\TV}$: For $\upsilon_1,\upsilon_2\in\cM(Y)$,
\[d_{\TV}(\upsilon_1,\upsilon_2)=\|\upsilon_1-\upsilon_2\|_{\TV}=\sup_{f\in \mathcal{B}_1(Y)}\int fd(\nu_1-\nu_2)\]} 
 {For every $\upsilon\in\cM(Y)$}, let
\[\|\upsilon\|_{\BL}\coloneqq \sup_{f\in\mathcal{BL}_1(Y)}\int_Y f\rd\upsilon,\]
the \emph{bounded Lipschitz norm} of $\upsilon$. 
{Let $d_{\BL}$ be the \emph{bounded Lipschitz metric} induced by $\|\cdot\|_{\BL}$, which admits the following supremum representation: For $\upsilon_1,\upsilon_2\in\cM(Y)$,
\[d_{\BL}(\upsilon_1,\upsilon_2)=\sup_{f\in \mathcal{BL}_1(Y)}\int fd(\upsilon_1-\upsilon_2)\]} 
Note that both $\cM_+(Y)$ and $\cP(Y)$ with the bounded Lipschitz metric are complete metric spaces \cite{B07}. Moreover, if the cardinality of $Y$ is infinite, then the topology induced by the bounded Lipschitz norm is strictly weaker than that induced by the total variation norm, and hence by Banach's theorem, $\cM$ equipped with the bounded Lipschitz norm is not complete since the two norms are \emph{not} equivalent \cite{B07}. In addition, if the complete metric space $Y$ is compact, then the bounded Lipschitz metric metrizes the weak topology, and convergence in $d_{\BL}$ also ensures the convergence in all finite moments. 

The following properties of measure-valued functions from \cite{GKX21,KX21} will be used to define the evolution of weights of a generalized {co-evolutionary} Kuramoto network in the next section (see \eqref{characteristic-Eq-1}). 

\begin{definition}
Let $(\eta_t)_{t\in\R}\subseteq\cM(Y)$. {Equip $\cM(Y)$ with the strong topology induced by the total variation norm.} If $$\lim_{\varepsilon\to0}\frac{\eta_{t+\varepsilon}-\eta_{t}}{\varepsilon}\in\cM(Y)$$ exists, then
\[\frac{\rd\eta_t}{\rd t}=\lim_{\varepsilon\to0}\frac{\eta_{t+\varepsilon}-\eta_{t}}{\varepsilon}\] is called the \emph{derivative of $\eta_t$ at $t$}.
\end{definition}
\begin{remark}
If $f\in \mathcal{C}^1(\R,(0,\infty))$ and $\xi\in\cM(Y)$, then $\eta_t=f(t)\xi\in\cM(Y)$ satisfies \begin{align*}
  \frac{\rd\eta_t}{\rd t}=&\ \frac{f'(t)}{f(t)}\eta_t,
  \end{align*}
c.f. \cite{GKX21}. Moreover, not all families of parameterized measures are differentiable (e.g., {$\{\delta_t\}_{t\in\R}$} \cite{GKX21}).
\end{remark}
Recall the following fundamental theorem of calculus for measure-valued functions \cite{GKX21}.
\begin{proposition}\label{prop-derivative}
{Let $\mathcal{N}$ be a compact interval of $\R$ containing a point $t_0\in\R$. Assume $\eta_{\cdot}\in \mathcal{C}(\mathcal{N},\cM(Y))$. Then $\xi_t=\int_{t_0}^t\eta_{\tau}\rd\tau\in\cM(Y)$, understood in the weak sense,
  \[\int_Yg\rd\xi_t=\int_{t_0}^t\left(\int_Yg\rd\eta_{\tau}\right)\rd\tau,\quad \forall g\in \mathcal{C}_{\sf b}(Y),\]
  is differentiable at $t$ for all $t\in\mathcal{N}$, where the derivative is understood as one-sided at the two endpoints of $\cN$.}
\end{proposition}
{Let $X$ be a compact set of a Euclidean space $\R^r$ for some $r\in\mathbb{N}$. Equipped with the metric induced by the $\ell_1$ of $\R^r$, $X$ becomes a complete metric space. Throughout this paper, $Y=X$ or $Y=\T$.} 
\begin{definition}\label{definition-DGM}
  A signed measure-valued function $\eta\in\mathcal{B}(X,\cM(X))$ is called a {signed} digraph measure ({SDGM}).
  \end{definition}
  {We remark that this definition extends the one (where the SDGM is positive and called DGM) introduced in \cite{KX21}.} 

Next, we introduce duality of sets and measures from \cite{KX21}, which will be used to state properties on the symmetry of \emph{digraph measures}\footnote{Here ``digraph measure'' refers {to} a measure-valued function, regarded as a generalization of graphon as a graph limit \cite{KX21}.} (see Proposition~\ref{prop-symmetry} below).
\begin{definition}
Given a set $A\subseteq X^2$, {the} set $A^*=\{(x,y)\in X^2\colon (y,x)\in A\}$ is called the \emph{dual of $A$}.
\end{definition}
\begin{definition}
Given a measure $\eta\in\cM(X^2)$, {the} measure $\eta^*$ defined by
$$\eta^*(A)=\eta(A^*),\quad \forall A\in\mathfrak{B}(X^2),$$ is called the \emph{dual of $\eta$}.
\end{definition}
\begin{definition}
  A measure $\eta\in\cM(X^2)$ is \emph{symmetric} if $\eta^*=\eta$.
\end{definition}

\begin{definition}\label{definition-symmetric-DGM}
{An SDGM} is \emph{symmetric} if, viewed as a {signed} measure\footnote{$\eta$ is understood as $\eta(A\times B)=\int_A\eta^x(B)\rd x$ for $A,\ B\in\mathfrak{B}(X)$ \cite{KX21}.} on $\cM(X^2)$, it is a symmetric measure.
\end{definition}

We remark that {SDGM} is a natural way to embed a graph into a dynamical network \cite{BCW23,KX21,KX22,LS23}. This is because the dynamical network is indexed by the node variable. Since the graph is directed, the weights of adjacent outward edges from a node $i$ is associated with a vector with entries being weights $(a_{ij})_{1\le j\le N}$, and such vectors can be naturally treated {as} a fiber measure \cite{BS20}. For this reason, we introduce {SDGM} (see also \cite{KX21}). Here $X$ stands for the space of vertices, and we choose $X$ to be abstract rather than the commonly used unit interval $[0,1]$ since in some cases due to the geometric feature of the underlying {signed} graphs, the space of vertices should be chosen different from $[0,1]$. For instance, if the {signed graphs} are rings, then the space $X$ is naturally chosen to be $\mathbb{T}$ to encode the cyclic nature. For more diverse interesting graph limits with different density captured by different $X$, the reader is referred to \cite[Section~2 and Section~4]{KX21}.

{For any $\eta,\xi\in\mathcal{B}(X,\cM(Y))$, let \[\|\eta\|=\sup_{x\in X}\|\eta^x\|_{\TV},\] 
\[d_{\infty,\TV}(\eta,\xi)\coloneqq\sup_{x\in X}d_{\TV}(\eta^x,\xi^x)\]
The other metric $d_{\infty,\BL}$ for space $\mathcal{B}(X,\cM(Y))$ is defined analogously.}

{Let $\cN\subseteq\R$ be a non-empty subinterval. For any $\eta_{\cdot},\xi_{\cdot}\in\mathcal{B}(\mathcal{N},\cB(X,\cM(Y)))$, let $$\|\eta_{\cdot}\|_{\cN}=\sup_{t\in\cN}\|\eta_t\|$$ be the uniform total variation norm of $\eta_{\cdot}$ and let 
\[d^{\cN}_{\infty,\TV}(\eta,\xi) =\sup_{t\in\cN} d_{\TV,\infty}(\eta_t,\xi_t),\quad d^{\cN}_{\infty,\BL}(\eta,\xi) =\sup_{t\in\cN} d_{\BL,\infty}(\eta_t,\xi_t)\]
the uniform total variation metric and uniform bounded Lipschitz metric, respectively.}

We will use the convergence in \emph{uniform total variation distance} to prove the existence of solutions to the generalized {co-evolutionary} Kuramoto network \eqref{characteristic-Eq-1}-\eqref{characteristic-Eq-3}. In comparison, we will use the \emph{uniform bounded Lipschitz distance} inducing the uniform weak topology for the approximation result in Section~\ref{sect-th-approx}.

Next we introduce the notion of uniform weak continuity, which will be used to define solutions of the VE.  
\begin{definition}
Let $Y=X$ or $Y=\T$.
Given $$\mathcal{B}(X,\cM_+(Y))\ni\eta\colon\begin{cases}
  X\to \cM_+(Y),\\ x\mapsto\eta^x,
\end{cases}$$ we say that $\eta$ is  {\emph{weakly continuous in $x$}} if for every $f\in\mathcal{C}(Y)$, we have
\[\mathcal{C}(X)\ni\eta(f)\colon\begin{cases}
  X\to\R,\\ x\mapsto\eta^x(f)\coloneqq\int_{Y}f\rd\eta^x.
\end{cases}\]
\end{definition}

\begin{definition} 
Let $\cI\subseteq\R$ be a non-empty compact interval, and $Y=X$ or $Y=\T$.
 Given $$\eta_{\cdot}\colon\begin{cases}
  \mathcal{I}\to \mathcal{B}(X,\cM_+(Y)),\\ t\mapsto\eta_t,
\end{cases}$$ we say that  $\eta_{\cdot}$ is  {\emph{uniformly weakly continuous in $t$}} if for every $f\in \mathcal{C}(Y)$, $t\mapsto\eta_t^x(f)$ is continuous in $t$ uniformly in $x\in X$.
\end{definition}

\begin{proposition}\label{prop-nu}
Let  $\cN\subseteq\R$ be a non-empty compact interval.
\begin{enumerate}
\item[\textnormal{(i)}] Let $\eta_{\cdot}\colon \mathcal{N}\to{\mathcal{B}(X,\cM_+(Y))}$. 
   Then $\eta_{\cdot}$ is uniformly weakly continuous {in $t$} if and only if $\eta_{\cdot}\in \mathcal{C}(\cN,{\mathcal{B}(X,\cM_+(Y))})$.
\item[\textnormal{(ii)}] Assume $\eta_{\cdot},\ \xi_{\cdot}\in \mathcal{C}(\cN,{\mathcal{B}(X,\cM_+(Y))})$, then
   $\|\eta_{\cdot}\|<\infty$ and $t\mapsto d_{\infty,\BL}(\eta_t,\xi_t)$ is continuous.
\item[\textnormal{(iii)}] Assume $\eta\in{\mathcal{C}(X,\cM(Y))}$. 
Then $\eta$ is weakly continuous {in $x$}.
\end{enumerate}
\end{proposition}
\begin{proof}
 { The proofs of (i)-(ii) are analogous to that of \cite[{Proposition~2.9}]{KX21} (see also \cite[ Proposition~2.3]{KX22}) and hence are omitted. For (iii), similarly, one can first show that if $\eta\in{\mathcal{C}(X,\cM_+(Y))}$, then $\eta$ is weakly continuous in $x$. Note that $\eta\in{\mathcal{C}(X,\cM(Y))}$ implies $\eta^+,\ \eta^-\in{\mathcal{C}(X,\cM_+(Y))}$, where $\eta^x=(\eta^+)^x-(\eta^-)^x$ is the Hahn decomposition of $\eta^x$ for each $x\in X$. Hence $\eta$  is weakly continuous in $x$ since so are both $\eta^+$ and $\eta^-$. It is completely analogous as for the case of real-valued functions to show that $\eta^+$ and $\eta^-$ are continuous and we leave it to the interested reader (indeed the positive part and the negative part as maps between two metric spaces are Lipschitz-1 and hence as a composition, $\eta^+,\ \eta^-$ are also continuous).} 
\end{proof}

\section{Generalized {co-evolutionary} Kuramoto network}\label{sect-characteristic}

We first reformulate several assumptions of this paper, which we are going to use in the paper. {Let $$\cB_{\infty}\coloneqq\{\xi\in\mathcal{B}(X,\cM_+(\T))\colon \int_{\T}\xi^x(\T)\rd x=1\}$$ and $$\cC_{\infty}\coloneqq\{\xi\in\mathcal{C}(X,\cM_+(\T))\colon \int_{\T}\xi^x(\T)\rd x=1\}$$ As we will set, the solutions of the generalized Vlasov equations (see Section~\ref{sect-Vlasov}) in will take values in these two sets, both of which  are closed subsets of $\mathcal{B}(X,\cM_+(\T))$, and hence are complete under the uniform metric $d_{\infty,\BL}$.} 
\begin{itemize}
\item[$\mathbf{(A1)}$] {$X\subseteq\R^r$ is a compact subset of unit Lebesgue measure for some $r\in\N$.} 
	\item[$\mathbf{(A2)}$] $g\colon \mathbb{T}\to\R$ is  Lipschitz continuous\footnote{Equivalently, $g$ can be extended to be a period-1 (coordinate-wise) Lipschitz continuous function on $\R$. Similarly for $h$.}: For $\phi,\varphi\in\mathbb{T}$,
	\begin{equation*}
	|g(\phi) - g(\varphi)| \le \Lip(g)d_{\mathbb{T}}(\phi,\varphi).
	\end{equation*}
	\item[$\mathbf{(A3)}$] $h\colon \mathbb{T}\to\R$ is Lipschitz continuous: For $\phi,\varphi\in\mathbb{T}$,
	\begin{equation*}
	|h(\phi) - h(\varphi)| \le \Lip(h)d_{\mathbb{T}}(\phi,\varphi).
	\end{equation*}
	\item[$\mathbf{(A4)}$] $\omega\colon \R\times X\to\R$ is continuous in {the first variable} $t$ and for every compact interval $\cN\subseteq\R$, $$\|\omega\|_{\cN,\infty}\coloneqq\sup\limits_{s\in \mathcal{N}}\|\omega(s,\cdot)\|_{\infty}\coloneqq\sup\limits_{s\in \mathcal{N}}\sup_{x\in X}|\omega(s,x)|<\infty.$$
	\item[{$\mathbf{(A5)}$}] $\eta_0\in \mathcal{C}(X,\cM(X))$.
\item[{$\mathbf{(A6)}$}] $\nu_{\cdot}\in\mathcal{C}_{\sf b}(\R,\cB_{\infty})$.
\item[{$\mathbf{(A4)'}$}] $\omega\colon \R\times X\to\R$ is continuous  in {the second variable} $x$.
\item[{$\mathbf{(A6)'}$}] $\nu_{\cdot}\in\mathcal{C}_{\sf b}(\R,\cC_{\infty})$.
   \end{itemize}

We make some comments on the above assumptions. $\mathbf{(A1)}$ ensures the vertex space is compact. Such compactness is important in establishing estimates in general for the main results \cite{KX21}. {For the ease of exposition, here we choose the Lebesgue measure $\lambda$ of $\R^r$ as the reference probability measure. One can simply relax the assumption $\mathbf{(A1)}$ so that $(X,\mathfrak{B}(X),\lambda)$ is a compact probability space  equipped with an arbitrary reference Borel probability measure $\mu_X$ on $X$. For instance, one can choose $(X,\mu_X)$ to be the unit sphere $\mathbb{S}^{r-1}$ with the normalized Haar measure on $\mathbb{S}^{r-1}$.} 
{Assumptions $\mathbf{(A2)}$-$\mathbf{(A6)}$ as well as $\mathbf{(A4)'}$ and $\mathbf{(A6)'}$ are the regularity conditions of the model. Among them $\mathbf{(A5)}$ is used for the approximation result, which can be relaxed to $\eta_0\in\cB(X,\cM(X))$ particularly for the results on the well-posedness \cite{KX21}. Nevertheless, for the ease of exposition, we will prefer not to distinguish the subtle difference.}

The following well-posedness of the {co-evolutionary} Kuramoto network \eqref{coevolution-a}-\eqref{coevolution-b} is an easy consequence of the Picard-Lindel\"{o}f iteration.
\begin{proposition}\label{prop-discrete-characteristic}
  Assume $\mathbf{(A2)}$-{$\mathbf{(A4)}$}. Then there exists a global solution to the {initial value problem}  of \eqref{coevolution-a}-\eqref{coevolution-b}.
\end{proposition}
{The reader may refer to Appendix~\ref{appendix-prop-discrete-characteristic} for a proof of Proposition~\ref{prop-discrete-characteristic}.}

Next, we investigate the properties of a generalized {co-evolutionary} Kuramoto model, which naturally extends \eqref{coevolution-a}-\eqref{coevolution-b}.  For the ease of exposition, we choose the initial time 
to be $0$ and let $(\phi_0,\eta_0)\in \mathcal{B}(X,\mathbb{T}\times\cM(X))$. Consider the following generalized {co-evolutionary} Kuramoto network:
\begin{alignat}{2}\label{characteristic-Eq-1}
\frac{\partial\phi(t,x)}{\partial t} = &\ \omega(t,x)+\int_X\int_{\mathbb{T}}g(\psi-\phi(t,x))
\rd\nu_{t}^y(\psi)\rd\eta_{t}^x(y),\quad t\in\cI,\ x\in X,\\
\label{characteristic-Eq-2}
\frac{\partial\eta_t^x}{\partial t}(\sbullet)=& -\varepsilon\eta_{t}^x(\sbullet)
-\Bigl(\varepsilon\int_{\mathbb{T}}h(\psi-\phi(t,x))\rd\nu_{t}^{\sbullet}(\psi)\Bigr)\lambda(\sbullet),\quad t\in\cI,\ x\in X,\\
\label{characteristic-Eq-3}
\phi(0,x)=&\ \phi_0(x),\quad {\eta_{t}^x|_{t=0}}=\ \eta_0^x,
\end{alignat}
which is interpretated in a (weaker) sense as an integral equation: For $t\in\mathcal{I}$, $x\in X$,
\begin{alignat}{2}\label{characteristic-Eq-int-1}
\phi(t,x) = &\ \phi_0(x)+\int_{0}^t\left[\omega(\tau,x)+\int_X\int_{\mathbb{T}}g(\psi-\phi(\tau,x))
\rd\nu_{\tau}^y(\psi)\rd\eta_{\tau}^x(y)
\right]\rd\tau \mod 1,\\
\label{characteristic-Eq-int-2}
\eta_t^x(\sbullet) = &\ \eta_0^x(\sbullet)-\left(\varepsilon\int_{0}^t\eta_{\tau}^x\rd\tau\right)(\sbullet)
-\Bigl(\varepsilon\int_{0}^t\left(\int_{\mathbb{T}}
h(\psi-\phi(\tau,x))\rd\nu_{\tau}^{\sbullet}(\psi)\right)\rd\tau\Bigr)\lambda(\sbullet),
\end{alignat}
where by Proposition~\ref{prop-derivative}, the equation \eqref{characteristic-Eq-int-2} is understood in the weak sense
\begin{multline}\label{characteristic-Int}
  \int_Xf(y)\rd\eta_t^x(y)=\int_Xf(y)\rd\eta_0^x(y)-\varepsilon\int_{0}^t
  \left(\int_Xf(y)\rd\eta_{\tau}^x(y)\right)\rd\tau\\
\ -\varepsilon\int_{0}^t\left(\int_Xf(y)\int_{\mathbb{T}}h(\psi-\phi(\tau,x))\rd\nu_{\tau}^y(\psi)\right)
\rd y\rd\tau,\quad \forall f\in \mathcal{C}(X).
\end{multline}

The above equation \eqref{characteristic-Int} for measures $\eta_t^{x}$ is well-defined. Indeed, since $h$ is continuous {and $\phi(t,x)$ is continuous in $t$ (to be shown in Theorem~\ref{thm-characteristic-1} below)}, by Proposition~\ref{prop-nu}(i) and {$\mathbf{(A6)'}$}, we have $h(\cdot-\phi(t,x))$ is integrable w.r.t. $\nu_{t}^y$, $\varepsilon\int_{\mathbb{T}}h(\psi-\phi(t,x))\rd\nu_{t}^y(\psi)$ is continuous $t$ and is integrable w.r.t. $\lambda$ so that $\Bigl(\varepsilon\int_{\mathbb{T}}h(\psi-\phi(t,x))\rd\nu_{t}^{\sbullet}(\psi)\Bigr)\lambda(\sbullet)$ defines a measure in $\cM(X)$ absolutely continuous w.r.t. $\lambda$.

It is easy to verify that one can recover the {co-evolutionary} network \eqref{coevolution-a}-\eqref{coevolution-b} by substituting specific $X$, $\lambda$, $\phi$, $\eta_{\cdot}$, and $\nu_{\cdot}$ into the characteristic equation \eqref{characteristic-Eq-1}-\eqref{characteristic-Eq-3}  (see Appendix~\ref{appendix-prop-discrete-characteristic} for details). {Indeed, \eqref{characteristic-Eq-1}-\eqref{characteristic-Eq-3} is an analogue of the McKean stochastic differential equation (SDE) when noise is added, where $\nu_t$ (i.e., the collection of fiber measures $\{\nu_t^x\}_{x\in X}$) that induces the charateristic equation corresponds to the distribution of the system: In the SDE case, it precisely corresponds to  the distribution of the McKean SDE; in the ODE case in our paper, it would correspond to the empirical distribution (see \eqref{eq:empirical-distribution} in Section~\eqref{sect-example}).} 

The rationale of using \eqref{characteristic-Eq-2} to describe the evolution of weights naturally is motivated by \cite{BS20}, where fiber measures are generalizations of vectors of weights of a vertex of a graph. Such a family  of measure differential equations are parameterized by vertices naturally appear in the context of dynamical networks \cite{BCW23,KX21,LS23}.

\begin{definition}
  A pair $(\phi,\eta_{\cdot})\in \mathcal{C}(\R,\mathcal{B}(X,\mathbb{T}\times\cM(X)))$  is called a \emph{global solution} to the {initial value problem (IVP)} of {\eqref{characteristic-Eq-1}-\eqref{characteristic-Eq-3}} if it satisfies \eqref{characteristic-Eq-int-1}-\eqref{characteristic-Eq-int-2} for all $x\in X$ and $t\in\R$.
\end{definition}
\begin{definition}
  Let $T>0$ and $\cI=[0,T]$. A pair $(\phi,\eta)\in \mathcal{C}(\cI,\mathcal{B}(X,\mathbb{T}\times\cM(X)))$  is called a \emph{local solution} to the IVP of {\eqref{characteristic-Eq-1}-\eqref{characteristic-Eq-3}} if it satisfies \eqref{characteristic-Eq-int-1}-\eqref{characteristic-Eq-int-2} for all $x\in X$ and $t\in\cI$.
\end{definition}

The following result provides well-posedness of the generalized {co-evolutionary} Kuramoto network \eqref{characteristic-Eq-1}-\eqref{characteristic-Eq-3}.

\begin{theorem}\label{thm-characteristic-1}
  Assume $(\mathbf{A1})$-$(\mathbf{A6})$. Let $(\phi_0,\eta_0)\in \mathcal{B}(X,\mathbb{T}\times\cM(X))$. Then there exists a unique global solution {$\mathcal{T}_{t}(\phi_0,\eta_0)=(\mathcal{T}^1_{t}(\phi_0,\eta_0),\mathcal{T}^2_{t}(\phi_0,\eta_0))$} in $\mathcal{B}(X,\mathbb{T}\times\cM(X))$ to \eqref{characteristic-Eq-1}-\eqref{characteristic-Eq-3}. 
   In particular,
     if $(\mathbf{A4})'$ holds and $(\phi_0,\eta_0)\in \mathcal{C}(X,\mathbb{T}\times\cM(X))$, then $\mathcal{T}_{t}(\phi_0,\eta_0)\in \mathcal{C}(X,\mathbb{T}\times\cM(X))$ for all $t\in\R$.
\end{theorem}
The proof is provided in Subsection~\ref{subsect-proof-of-characteristic}.
\begin{remark}
Note that the digraph measures may not be regular enough, so that the convolution of the test function with the fiber measure may not be smooth but rather just measurable. Hence, {the bounded Lipschitz metric will not be a good choice while the total variation norm seems to be a natural choice as measurability of test functions suffices}. Nevertheless, such norm is only well-suited for the existence of solutions to the characteristic equation. {As will be seen below, {in contrast}, bounded Lipschitz metric is used to establish the well-posedness and approximation of solutions to the generalized Vlasov equation {(see Section~\ref{sect-Vlasov})}.} 
\end{remark}

The following property demonstrates that the symmetry of the evolving graph measure is preserved over time {under certain symmetry condition on $h$ and $\nu_{\cdot}$}. 

{\begin{proposition}\label{prop-symmetry}
 Assume $(\mathbf{A1})$-{$(\mathbf{A6})$}. Additionally assume $h$ and $\nu_{\cdot}$ satisfy the following symmetry condition:
 \begin{equation}\label{eq:condition-symmetry}
 \int_{\T} h(\psi-\phi(t,x))\rd\nu^y_t(\psi) = \int_{\T} h(\psi-\phi(t,y)\rd\nu^x_t(\psi),\quad \forall t\in\cI,\ \lambda\otimes\lambda\text{-a.e.}\ (x,y)\in X^2,
 \end{equation}
 where $\lambda\otimes\lambda$ is the Lebesgue measure on $\R^{2r}$.
 Let $(\phi,\eta_{\cdot})$ be the solution to the IVP  {\eqref{characteristic-Eq-1}-\eqref{characteristic-Eq-3}} 
  on $t\in \cI$.  Then $\eta_t$ is symmetric for all $t\in\cI$ provided $\eta_0$ is symmetric.  
\end{proposition}}
\begin{proof}
{Since $\eta_0$ is symmetric, in the light of \eqref{Eq-variation-2} below and Fubini Theorem, it suffices to show for each $A\in\mathfrak{B}(X^2)$, \[\int_A\int_{\T} h (\psi-\phi(\tau,x)) \rd\nu^y_{\tau}(\psi)\rd x\rd y=\int_A\int_{\T} h (\psi-\phi(\tau,y)) \rd\nu^x_{\tau}(\psi)\rd x\rd y,\] which holds provided \eqref{eq:condition-symmetry} holds.}  
\end{proof}
{\begin{example}
  Let $X=[0,1]$ and $h$ with its natural extension (still denoted $h$) being an even 1-periodic function on $\R$, and $\nu_t^x=\delta_{\phi(t,x)}$ for $t\in\cI$ and $\lambda$-a.e. $x\in[0,1]$. Then $h(u)=h(-u)$ for all $u\in\R$ and hence for all $t\in\cI$ and $x,y\in[0,1]$, \begin{multline*}\int_{\T}h(\psi-\phi(t,x))\rd\nu^y_t(\psi)=\int_{\T}h(\psi)\rd\nu_t^y(\phi(t,x)+\psi) = h(\phi(t,y)-\phi(t,x))\\
  =h(\phi(t,x)-\phi(t,y))=\int_{\T} h(\psi)\rd\nu_t^y(\phi(t,x)+\cdot)=\int_{\T}h(\psi-\phi(t,y))\rd\nu^x_t(\psi),
  \end{multline*} i.e., \eqref{eq:condition-symmetry} holds. Let $N\in\mathbb{N}$ and $W=(W_{i,j})_{1\le i,j\le N}$ be a symmetric signed matrix ($W_{ij}=W_{ji}\in\R$) and $I_i=[\tfrac{i-1}{N},\tfrac{i}{N}[$ for $i=1,\ldots,N$. Let $\phi(t,x)=\phi_i(t)$ for $x\in$, $i=1,\ldots,N$. Define $$\eta_0\colon x\mapsto \sum_{i=1}^N\mathbbm{1}_{I_i}(x)\sum_{j=1}^NW_{ij}\lambda|_{I_j},$$ where $\lambda|_{I_j}$ is the Lebesgue measure restricted to $I_j$. Then for any $A\in\mathfrak{B}(X^2)$,
  \begin{align*}
  \eta_0(A)=& \sum_{i=1}^N\sum_{j=1}^N\int_{A\cap I_i\times I_j}\rd\eta_0^x(y)\rd x=\sum_{i=1}^N\sum_{j=1}^NW_{ij}(\lambda\otimes\lambda)(A\cap I_i\times I_j)\\
  =& \sum_{i=1}^N\sum_{j=1}^NW_{ji}(\lambda\otimes\lambda)(A\cap I_j\times I_i)=\sum_{i=1}^N\sum_{j=1}^NW_{ij}(\lambda\otimes\lambda)(A\cap I_j\times I_i)\\
  =& \sum_{i=1}^N\sum_{j=1}^NW_{ij}(\lambda\otimes\lambda)(A^*\cap I_j\times I_i)=\eta_0(A^*),
  \end{align*}
  where the fourth equality uses the symmetry of $(W_{ij})_{1\le i,j\le N}$. By Definition~\ref{definition-symmetric-DGM}, $\eta_0$ is a symmetric SDGM.
\end{example}}
We have an equivalent characterization of the solutions to the characteristic equation \eqref{characteristic-Eq-1}-\eqref{characteristic-Eq-3} by an integral equation.
\begin{proposition}[Variation of constants formula]\label{prop-equivalent}
 Assume $(\mathbf{A1})$-{$(\mathbf{A6})$}. Then $(\phi,\eta)$ is a local (global, respectively) solution to \eqref{characteristic-Eq-1}-\eqref{characteristic-Eq-3} if and only if  $(\phi,\eta)$  is a local (global, respectively) solution to
\begin{alignat}{2}\nonumber
\phi(t,x)=&\ \Bigl(\phi_0(x)+\int_{0}^t\left(\omega(s,x)+\e^{-\varepsilon s}\int_X\int_{\T}g(\psi-\phi(s,x))\rd\nu_s^y(\psi)\rd\eta_0^x(y)\right.\\
\nonumber&\left.-\varepsilon\int_{0}^s\e^{-\varepsilon(t-\tau)}\int_X\left(\int_{\T}g(\psi-\phi(s,x))
\rd\nu_s^y(\psi)\right.\right.\\
\label{Eq-variation-1}&\left.\left.\cdot\int_{\T}h(\psi-\phi(\tau,x))\rd\nu_{\tau}^y(\psi)\right)\rd y\rd\tau\right)\rd s\Bigr) \mod1,\quad x\in X,\quad t\in\R,\\
\label{Eq-variation-2}  
  \eta_t^x(\sbullet)=&\ \e^{-\varepsilon t}\eta_0^x(\sbullet)-\Bigl(\varepsilon\int_{0}^t\e^{-\varepsilon(t-s)}\int_{\T}h(\psi-\phi(s,x))
  \rd\nu_{s}^{\sbullet}(\psi)\rd s\Bigr)\lambda(\sbullet),\quad  x\in X,\quad t\in\R.
\end{alignat}
\end{proposition}
The proof is given in Appendix~\ref{appendix-prop-equivalent}.

From \eqref{Eq-variation-2}, the graph measure is composed of two parts, the dilated initial measure with time-dependent dilation $e^{-\varepsilon t}$, and an absolutely continuous measure with time-dependent density $-\varepsilon\int_{0}^te^{-\varepsilon(t-\tau)}\int_{\T}h(\psi-\phi(\tau,x))\rd\nu_{\tau}^y(\psi)\rd\tau$, $y\in X$.

By Theorem~\ref{thm-characteristic-1}, let $\mathcal{T}_{t}[\nu_{\cdot},\omega]
=(\mathcal{T}^1_{t}[\nu_{\cdot},\omega],\mathcal{T}^2_{t}[\nu_{\cdot},\omega])$ denote 
{the solution map of the integral {equations \eqref{Eq-variation-1} and \eqref{Eq-variation-2}}, emphasizing the dependence on $\nu_{\cdot}$ and $\omega$}. From Proposition~\ref{prop-equivalent} we immediately get the properties of $\mathcal{T}$.

\begin{corollary}\label{cor-flow}
 Assume $(\mathbf{A1})$-{$(\mathbf{A6})$}. Then the {solution map} of \eqref{characteristic-Eq-1}-\eqref{characteristic-Eq-3} is given, for $x\in X$ and $t\in\R$, by
\begin{align*}
\mathcal{T}_{t}^{1,x}[\nu_{\cdot},\omega](\phi_0,\eta_0)=&\ \Bigl(\phi_0(x)+\int_{0}^t\omega(s,x) \rd s\\
&\left.+\int_{0}^t\e^{-\varepsilon s}\int_X\int_{\mathbb{T}}g(\psi-\mathcal{T}_{s}^{1,x}[\nu_{\cdot},\omega](\phi_0,\eta_0))
\rd\nu_s^y(\psi)\rd\eta_0^x(y) \rd s\right.\\
&-\varepsilon\int_{0}^t\int_{0}^s\e^{-\varepsilon(t-\tau)}\int_X\Bigl(\int_{\bT}g(\psi-
\mathcal{T}_{s}^{1,x}[\nu_{\cdot},\omega](\phi_0,\eta_0))\rd\nu_s^y(\psi)\\
&{\cdot}\int_{\mathbb{T}}h\bigl(\psi-\mathcal{T}_{\tau}^{1,x}
[\nu_{\cdot},\omega](\phi_0,\eta_0)\bigr)\rd\nu_{\tau}^y(\psi)\Bigr)\rd y\rd\tau\rd s\Bigr) \mod1,\\
\mathcal{T}_{t}^{2,x}[\nu_{\cdot},\omega](\phi_0,\eta_0)(\sbullet)
=&\ \e^{-\varepsilon t}\eta_0^x(\sbullet)\\
&-\Bigl(\varepsilon\int_{0}^t\e^{-\varepsilon(t-\tau)}\int_{\mathbb{T}}
 h(\psi-\mathcal{T}_{\tau}^{1,x}[\nu_{\cdot},\omega]
 (\phi_0,\eta_0))\rd\nu_{\tau}^{\sbullet}(\psi)\rd\tau\Bigr)\lambda(\sbullet).
\end{align*}
\end{corollary}

To investigate the mean field behavior of the {co-evolutionary} Kuramoto model on heterogeneous networks, one typically needs to construct a Vlasov-type equation via some fixed point equation \cite{N84}. A simple look into the generalized co-evolutionary  {Kuramoto} network reveals that such MFLs may have support on an infinite dimensional space (some measure space, for the sake of the second component\---the graph measure). 
In order to get {around} this difficulty/complexity, in the following we decouple the characteristic equation using Corollary~\ref{cor-flow} so that we embed the dynamic nature of the underlying graph measure into the dynamics of the oscillators. In this way, we come up with a one-dimensional integral equation on the circle, and can turn to study the MFL for this integral model coupled on static initial graph measures.

Furthermore, for every fixed initial {SDGM} $\eta_0\in\mathcal{C}(X,\cM(X))$, {for any given $T>0$ and $\cI=[0,T]$, we define a family of parameterized operators: For every $t\in\cI$,}
\[{\begin{cases} \mathcal{S}_t[\eta_0,\nu_{\cdot},\omega]\colon \mathcal{B}(X,\mathbb{T})\to \mathcal{B}(X,\mathbb{T})\\
\mathcal{S}^x_{t}[\eta_0,\nu_{\cdot},\omega](\phi_0)
=\mathcal{T}_{t}^{1,x}[\nu_{\cdot},\omega](\phi_0,\eta_0),\quad x\in X,
\end{cases}}
\]
that is, the solution map of the following integral equation:
\begin{equation}\label{integralequation}
\begin{split}
\phi(t,x)=&\ \Bigl(\phi_0(x)+\int_{0}^t\Bigl(\omega(s,x)+\e^{-\varepsilon s}\int_X\int_{\mathbb{T}}g(\psi-\phi(s,x))
\rd\nu_s^y(\psi)\rd\eta_0^x(y)\\
&-\varepsilon\int_{0}^s\e^{-\varepsilon(t-\tau)}\int_X\int_{\mathbb{T}}g(\psi-
\phi(s,x))\rd\nu_s^y(\psi)\\
&\cdot\int_{\mathbb{T}}h(\psi-\phi(\tau,x))\rd\nu_{\tau}^y(\psi))
\rd y\rd\tau\Bigr)\rd s\Bigr)\!\!\!\mod 1
\end{split}\end{equation}

We remark that \eqref{integralequation} is one-dimensional, which makes it possible to generate a fixed point equation induced by the semiflow of \eqref{integralequation}, and we can use this to study the mean field dynamics of the original coupled hybrid characteristic equation \eqref{characteristic-Eq-1}-\eqref{characteristic-Eq-3}. Note that \eqref{integralequation} is generally referred to as \emph{equation of mean field characteristics} \cite{G13}, and  
{in case it generates a \emph{flow}, such flow is named as ``mean field characteristic flow'' \cite{G13}. Nevertheless, the semiflow of \eqref{integralequation} may not necessarily be a \emph{flow}, unless in special cases, 
e.g., when $h$ is a constant, i.e., when coevolution disappears despite evolution of the underlying graph.}
 Indeed, it is possible that multiple different initial graph measures $\eta_0$ contribute to the same phase at future time.

In order to fully investigate the well-posedness of solutions to a fixed point equation, we need to rely on some continuity properties of the operator $\mathcal{S}$.
\begin{proposition}\label{prop-semiflow}
 Assume $(\mathbf{A1})$-$(\mathbf{A6})$.  Let $T>0$ {and $\cI=[0,T]$.}
  \begin{enumerate}
    \item[(i)] $\mathcal{S}_{t}^x[\eta_0,\nu_{\cdot},\omega]$ is continuous in $x$: For $\phi_0\in{\cC(X,\T)}$,
        \[\lim_{{|x-y|}\to0}|\mathcal{S}_{t}^x[\eta_0,\nu_{\cdot},\omega]{(\phi_0)}
        -\mathcal{S}_{t}^y[\eta_0,\nu_{\cdot},\omega]{(\phi_0)}|=0,\] provided $(\mathbf{A4})'$ holds.
    \item[(ii)] $\mathcal{S}_{t}^x[\eta_0,\nu_{\cdot},\omega]$ is Lipschitz continuous in $t$: For $\phi_0\in{\cB(X,\T)}$, for $t_1,\ t_2\in\cI$, $$|\mathcal{S}_{t_1}^{x}[\eta_0,\nu_{\cdot},\omega](\phi_0)
-\mathcal{S}_{t_2}^{x}[\eta_0,\nu_{\cdot},\omega](\phi_0)|\le L_1(\nu_{\cdot})|t_1-t_2|,$$ where $L_1(\nu_{\cdot})=\|\omega\|_{\infty}+\|g\|_{\infty}{\|\nu_{\cdot}\|_{\cI}}\|\eta_0\|
+(\tfrac{1}{2}T^2+1)\varepsilon\|g\|_{\infty}\|h\|_{\infty}({\|\nu_{\cdot}\|_{\cI}})^2$.
    \item[(iii)] $\mathcal{S}_{t}^x[\eta_0,\nu_{\cdot},\omega](\phi_0)$ is Lipschitz continuous in $\phi$:
    For $\phi_0,\ \varphi_0\in\mathcal{C}(X,\mathbb{T})$, $${\sup_{x\in X}|\mathcal{S}_{t}^{x}[\eta_0,\nu_{\cdot},\omega](\phi_0)
-\mathcal{S}_{t}^{x}[\eta_0,\nu_{\cdot},\omega](\varphi_0)|\le \e^{L_2(\nu_{\cdot})t}\|\phi_0-\varphi_0\|_{\infty},}$$where $L_2(\nu_{\cdot})=C_1(\nu_{\cdot})+
\frac{\varepsilon\|g\|_{\infty}\Lip(h)({\|\nu_{\cdot}\|_{\cI}})^2}{C_1(\nu_{\cdot})}$ and $C_1(\nu_{\cdot})=\Lip(g){\|\nu_{\cdot}\|_{\cI}}\left(\|\eta_0\|+\|h\|_{\infty}{\|\nu_{\cdot}\|_{\cI}}\right)$.
       \item[(iv)] $\mathcal{S}_{t}^x[\eta_0,\nu_{\cdot},\omega]$ is Lipschitz continuous in $\omega$: Assume $\widetilde{\omega}$ also satisfies $(\mathbf{A4})$ with $\omega$ replaced by $\widetilde{\omega}$, then
       \[|\mathcal{S}_{t}^{x}[\eta_0,\nu_{\cdot},\omega](\phi_0)-\mathcal{S}_{t}^{x}
       [\eta_0,\nu_{\cdot},\widetilde{\omega}](\phi_0)|\le T\e^{L_2t}\|\omega-\widetilde{\omega}\|_{\infty,\cI}.\]
        \item[(v)] $\mathcal{S}_{t}^x[\eta_0,\nu_{\cdot},\omega]$ is continuous in $\eta_0$: Let $(\eta_k)_{k\in\N_0}\subseteq\mathcal{C}(X,{\cM(X)})$\footnote{Here we slightly abuse $\eta_k$, which does not refer to $\eta_t$ at time $t=k\in\N$.} be such that $$\lim_{k\to\infty}d_{\infty,\BL}(\eta_0,\eta_k)=0.$$ {Assume additionally $(\mathbf{A6})'$.} Then
    \[\lim_{k\to\infty}\sup_{t\in\cI}\sup_{x\in X}|\mathcal{S}_{t}^{x}[\eta_0,\nu_{\cdot},\omega](\phi_0)-\mathcal{S}_{t}^{x}
    [\eta_k,\nu_{\cdot},\omega](\phi_0)|=0.\]
    \item[(vi)] $\mathcal{S}_{t}^x[\eta_0,\nu_{\cdot},\omega]$ is Lipschitz continuous in $\nu$: For $\nu_{\cdot},\ \upsilon_{\cdot}\in\cC(\cI,\cB_{\infty})$, $$|\mathcal{S}_{t}^{x}[\eta_0,\nu_{\cdot},\omega](\phi_0)
        -\mathcal{S}_{t}^{x}[\eta_0,\upsilon_{\cdot},\omega](\phi_0)|\le L_3\e^{L_4t}\int_0^t {d}_{\infty,\BL}(\nu_s,\upsilon_s)\rd s,$$
  where $L_3=L_3(\nu_{\cdot},\upsilon_{\cdot})=\BL(g)({\|\upsilon_{\cdot}\|_{\cI}}\|h\|_{\infty}+\|\eta_0\|)
  +\|g\|_{\infty}\BL(h){\|\nu_{\cdot}\|_{\cI}}\varepsilon T$ and $L_4=L_4(\nu_{\cdot},\upsilon_{\cdot})={\|\nu_{\cdot}\|_{\cI}}\Bigl(\Lip(g)(\|\eta_0\|+\|h\|_{\infty}{\|\upsilon_{\cdot}\|_{\cI}})
    +\|g\|_{\infty}\Lip(h){\|\nu_{\cdot}\|_{\cI}}\varepsilon T\Bigr)$.
  \end{enumerate}
\end{proposition}
The proof of Proposition~\ref{prop-semiflow} is provided in Appendix~\ref{appendix-prop-semiflow}. \begin{remark}
From $d_{\T}(x,y)\le|x-y|$ for any $x,y\in\mathbb{T}$ it follows that the results in Proposition~\ref{prop-semiflow} still hold when $|\cdot-\cdot|$ is replaced by the distance $d_{\T}$ on the circle. However, the (stronger) upper estimates will be used in the proof of further {results} in Section~\ref{sect-Vlasov}.
\end{remark}

{The following result on existence and Lipschitz continuity of the semiflow is a direct consequence of Proposition~\ref{prop-semiflow}.}
\begin{proposition}\label{prop-inverse-map}
 Assume $(\mathbf{A1})$-{$(\mathbf{A6})$} and $(\mathbf{A4})'$. Let $T>0$ and 
{$\cI=[0,T]$.} Then for every $t\in\cI$, $\mathcal{S}^{{\cdot}}_{t}[\eta_0,\nu_{\cdot},\omega]$ 
 {is a Lipschitz operator from $\mathcal{C}(X,\mathbb{T})$ to $\mathcal{C}(X,\mathbb{T})$}. 
\end{proposition}

\section{Generalized Vlasov Equation}\label{sect-Vlasov}

{To study convergence of \eqref{coevolution-a}-\eqref{coevolution-b} to some limit in a distributional sense, based on the standard approach of deriving MFL \cite{G13}, it is tempting to construct a fixed point equation based on the equation of characteristics \eqref{characteristic-Eq-1}-\eqref{characteristic-Eq-3}. Nevertheless, \eqref{characteristic-Eq-1}-\eqref{characteristic-Eq-3} is not a (fiber) finite dimensional ODE, so one may expect one has to to study a Vlasov type PDE in an infinite dimensional spatial domain. It seems the usual approximation schemes used in the literature e.g., \cite{KM18,KX21}, all collapse.}

{To get around this barrier caused by infinite dimensionality, our plan is to use \eqref{integralequation} in place of \eqref{characteristic-Eq-1}-\eqref{characteristic-Eq-3}, as the equation of characteristics, for the former is a finite dimensional (integral) equation. Then we instead use the solution map of \eqref{integralequation} to construct the fixed point equation, and study the convergence of \eqref{coevolution-a}-\eqref{coevolution-b}. To do this,}  we investigate the Lipschitz continuity of the solution $\nu$ to the 
\emph{generalized Vlasov equation} (VE) in the sense of Neunzert \cite{N84}
\begin{align}\label{Fixed}
  \nu_t^x=\nu_0^x\circ (\mathcal{S}^x_{t})^{-1}[\eta_0,\nu_{\cdot},\omega],\quad x\in X,
\end{align}
with respect to $\eta_0$ and $\nu_0$, where $(\mathcal{S}^x_{t})^{-1}[\eta_0,\nu_{\cdot},\omega](\phi(t,x))$ is the pre-image\footnote{Note that ${\mathcal{S}^x_{t}[\eta_0,\nu_{\cdot},\omega]}$ may not have an inverse.} of $\phi(t,x)$ under the operator $\mathcal{S}^x_{t}$, for {$t\in\cI=[0,T]$ with $T$ fixed while to be determined below}. We remark that a variant of the generalized VE seemed to be first introduced in \cite{KM18} to investigate MFLs of IPS coupled on heterogeneous static networks.

{To obtain the existence of solutions to the generalized Vlasov equation, it is standard to apply Banach fixed point theorem to a given complete metric space, e.g., $\cB(\cI,\cB_{\infty})$ or $\cC(\cI,\cB_{\infty})$.}

{To prepare for the existence result of the generalized Vlasov equation as outlined above, we will first establish some continuity properties of the map on the right hand side of \eqref{Fixed}. 
 Define $\cF[\eta_0,\omega]$ by}
$$(\cF[\eta_0,\omega]\nu_{\cdot})_t^x=\nu_0^x\circ \left(\mathcal{S}^x_{t}[\eta_0,\nu_{\cdot},\omega]\right)^{-1},\quad t\in\cI,\ x\in X.$$

{The result on  existence as well as approximation of solutions to the generalized Vlasov equation rests on the properties of {$\cF$}, which we will establish below, is as follows.}

{
\begin{proposition}\label{prop-continuousdependence}
Assume $\mathbf{(A1)}$-{$\mathbf{(A5)}$}. Let $T>0$ and $\cI=[0,T]$. Assume $\nu_{\cdot}\in \mathcal{C}(\cI,\mathcal{B}_{\infty})$. Then 
 $\cF[\eta_0,\omega]\nu_t$ is\footnote{Here we simply denote $\cF[\eta_0,\omega]\nu_t$ for $(\cF[\eta_0,\omega]\nu_{\cdot})_t$.}
\begin{enumerate}
\item[\textnormal{(i)}] continuous in $t$:  $t\mapsto\cF[\eta_0,\omega]\nu_{t}\in \mathcal{C}(\cI,\mathcal{B}_{\infty})$. {Moreover, the mass conservation law holds: $$(\cF[\eta_0,\omega]\nu_t)^x(\mathbb{T})=\nu_0^x(\bT),\quad \forall t\in\cI,\quad x\in X.$$In particular, if $\mathbf{(A4)'}$ holds and  $\nu_0\in\mathcal{C}_{\infty}$, 
then $\cF[\eta_0,\omega]\nu_{\cdot}\in \mathcal{C}(\cI,\mathcal{C}_{\infty})$.}
  \item[\textnormal{(ii)}] Lipschitz continuous in $\omega$:  Assume $\widetilde{\omega}$ satisfies $\mathbf{(A4)}$ with $\omega$ replaced by $\widetilde{\omega}$. For all $t\in\cI$, \begin{equation*}
        d_{\infty,\BL}(\mathcal{F}[\eta_0,\omega]\nu_{t},
        \mathcal{F}[\eta_0,\widetilde{\omega}]\nu_{t})\le T\|\nu_{\cdot}\|_{\cI} \e^{L_2(\nu_{\cdot})t}\|\omega-\widetilde{\omega}\|_{\infty,\cI},
\end{equation*}
where $L_2(\nu_{\cdot})$ is defined as in Proposition~\ref{prop-semiflow}(iii). 
  \item[\textnormal{(iii)}]  continuous in $\eta_0$: Assume $\mathbf{(A6)'}$. For all $t\in\cI$, and $\eta_k\in\mathcal{B}(X,\cM(X))$ for $k\in\N_0$ such that $\lim_{k\to\infty}{d_{\infty,\BL}}(\eta_0,\eta_k)=0$,
  \begin{equation*}
  \begin{split}
  &\lim_{k\to\infty}d_{\infty,\BL}(\cF[\eta_0,\omega]\nu_t,
  \cF[\eta_k,\omega]\nu_t)=0.
  \end{split}
\end{equation*}
 \item[\textnormal{(iv)}] Lipschitz continuous in $\nu_{\cdot}$: For all $t\in\cI$, and $\nu_{\cdot},\upsilon_{\cdot}\in \mathcal{C}(\cI,\cB_{\infty})$,
  \begin{multline*}
  d_{\infty,\BL}(\cF[\eta_0,\omega]\nu_t,\cF[\eta_0,\omega]\upsilon_t)\le \e^{{L_2(\upsilon_{\cdot})}t}d_{\infty,\BL}(\nu_0,\upsilon_0)\\
+ {L_3(\nu_{\cdot},\upsilon_{\cdot}){\|\nu_{\cdot}\|_{\cI}}\e^{L_4(\nu_{\cdot},\upsilon_{\cdot}) t}}\int_0^t d_{\infty,\BL}(\nu_s,\upsilon_s) \rd s,  
  \end{multline*}
{where $L_2(\upsilon_{\cdot})$, $L_3(\nu_{\cdot},\upsilon_{\cdot})$ and $L_4(\nu_{\cdot},\upsilon_{\cdot})$ are as defined in Proposition~\ref{prop-semiflow}.} 
\end{enumerate}
\end{proposition}
}
The proof of Proposition~\ref{prop-continuousdependence} is provided in Appendix~\ref{appendix-prop-continuousdependence}.

Next, we will provide well-posedness of the generalized Vlasov equation.

{\begin{theorem}
  \label{thm-existence-VE}
  Assume $\mathbf{(A1)}$-$\mathbf{(A5)}$. Let $T>0$ and $\cI=[0,T]$. Assume $\nu_0\in \mathcal{B}_{\infty}$. Then there exists a unique solution   
to 
\eqref{Fixed}. In particular, if $\nu_0\in\cC_{\infty}$, then $\nu_{\cdot}\in\mathcal{C}(\cI,\mathcal{C}_{\infty})$.
\end{theorem}}

The proof of Theorem~\ref{thm-existence-VE} is provided in Section~\ref{subsect-proof-th-existence}.

\begin{remark}
 Although we successfully decouple the dynamics of the edges from the dynamics of the nodes and thus achieve an analogue of equation of characteristics in the classical sense, the classical technique via equation of characteristics by interchanging derivative in time with derivative in state fails, since the vector field not only depends on the current state, but also on the entire trajectory (as well as the mean field limit) from initial time to present. This makes it impossible to obtain a characterization of the MFL{, when it is absolutely continuous,} via a Vlasov PDE on the \emph{finite dimensional} phase space $\mathbb{T}$. We believe such MFL may share similar properties as that for a delay IPS \cite{BPR22}.
\end{remark}

The next proposition provides continuous dependence of the solutions to the generalized VE, which is useful to obtain the approximation result later.

\begin{proposition}\label{prop-sol-fixedpoint}
{Assume $\mathbf{(A1)}$-{$\mathbf{(A5)}$}. 
Let $T>0$, $\cI=[0,T]$. 
Then solutions to  
\eqref{Fixed} have continuous dependence on
\begin{enumerate}
\item[\textnormal{(i)}] the initial conditions:
\[d_{\infty,\BL}(\nu_t^1,\nu_t^2)\le \textnormal{e}^{L_5t}d_{\infty,\BL}(\nu_0^1,\nu_0^2),\quad t\in\cI,\]
where $\nu^i_{\cdot}$ is the solution to \eqref{Fixed} with initial condition $\nu^i_0\in{\cB_{\infty}}$ for $i=1,2$, and $L_5(\nu_{\cdot}^1,\nu_{\cdot}^2) = \Bigl(L_3(\nu^{1}_{\cdot},\nu^{2}_{\cdot})\|\nu^{1}_{\cdot}\|+\max\{L_2(\nu^{1}_{\cdot}),L_4(\nu^{1}_{\cdot},\nu^{2}_{\cdot})\}\Bigr)$
\item[\textnormal{(ii)}] $\omega$: Let $\nu^i_{\cdot}$ be the solution to \eqref{Fixed} with functions $\omega_i$ for $i=1,2$ and the same initial condition $\nu_0^1=\nu_0^2{\in\cB_{\infty}}$. Then
\[d_{\infty,\BL}(\nu^1_t,\nu^2_t)\le
{T\|\nu^1_{\cdot}\|_{\cI}}\textnormal{e}^{L_5(\nu^1_{\cdot},\nu^2_{\cdot})t}\|\omega_1-\omega_2\|_{\infty,\cI},\]
{where $L_5(\nu^1_{\cdot},\nu^2_{\cdot})$ is given in (i).} 
\item[\textnormal{(iii)}] $\eta_0$: Let {$\{\eta_k\}_{k\in\N}\subseteq \mathcal{B}(X,\cM(X))$.} If $\lim_{k\to\infty}d_{\infty,\BL}(\eta_0,\eta_k)=0$, then
    $$\lim_{k\to\infty}d^{\cI}_{\infty,\BL}(\nu_{\cdot},\nu^k_{\cdot})=0,$$
where $\nu^k_{\cdot}$ is the solution to \eqref{Fixed} with initial {SDGM} $\eta_k$ for $k\in\N$ and the same initial condition {$\nu_0\in\cC_{\infty}$}.
\end{enumerate}}
\end{proposition}
The proof of Proposition~\ref{prop-sol-fixedpoint} is provided in Appendix~\ref{appendix-prop-sol-fixedpoint}.
{\begin{remark}
One can simply observe by symmetry that one can replace $L_5$ in Proposition~\ref{prop-sol-fixedpoint}(i) by 
$$L_5(\nu_{\cdot}^1,\nu_{\cdot}^2) = \min_{i=1,2} \Bigl(L_3(\nu^{i}_{\cdot},\nu^{3-i}_{\cdot})\|\nu^{i}_{\cdot}\|+\max\{L_2(\nu^{i}_{\cdot}),L_4(\nu^{i}_{\cdot},\nu^{3-i}_{\cdot})\}\Bigr)$$
Nevertheless, we here do not aim to make an effort to optimize the constant.
\end{remark}}
\section{Approximation of the mean field limit}\label{sect-th-approx}

{Based on Proposition~\ref{prop-sol-fixedpoint}, one can simply obtain the convergence of empirical distributions to the MFL using triangle inequalities, \emph{provided the respective SDGMs, frequency functions, as well as initial empirical distributions converge.}}

{In this section, we consider an inverse problem: Given the solution to the generalized Vlasov equation, how to construct a sequence of {co-evolutionary} ODEs in forms of \eqref{coevolution-a}-\eqref{coevolution-b}, so that the solution is the MFL of the the constructed ODEs? In this way, we not only address the existence of the approximations of particularly, the SDGM and the initial distributions of the generalized Vlasov equation, but do so by construction. Hence such construction can be viewed as a numerical scheme of the generalized Vlasov equation.}

{To make a long story short, the Vlasov equation can be approximated by a sequence of ODEs \eqref{lattice} indexed by $m,n$ (Theorem~\ref{th-approx}), which are well-posed (Proposition~\ref{prop-well-posed-lattice}).  The following lemmas provide the involved functions that appear in the constructed ODEs, as approximations of the frequency function $\omega$, the initial distribution $\nu_0$, as well as the initial SDGM $\eta_0$.}

{Before proceeding to the technical lemmas, we first provide an idea on how to construct these functions ($\omega$, $\nu_0$, $\eta_0$) of $x$.} 

{We first mesh the underlying space $X$ using a partition $\{A_i^m\}_{1\le i\le m}$ into grids of size $m$. Then by virtue of the continuity of the above three functions, within each small set $A_i^m$ of the partition, we will use a constant function on $A_i^m$ to approximate each of them confined to the subdomain $A_i^m$, respectively;  for $\omega$ this is rather standard, while for the other two measure valued functions, such ``constant'' may be of a measure value  (with or without an absolutely continuous part w.r.t. $\lambda$) of a support containing potentially infinitely many (or finite while unbounded number in $m$ as $m$ tends to infinity). To generate a measure in terms of a finite dimensional ODE (in most cases, e.g., \cite{KX21}) or equivalently while more conveniently in form of an integral equation (see \eqref{lattice} below) in our case, it is arguably intuitive to consider a finitely supported measure (in \cite{KX21} as well as in this paper) or a piecewise uniform measure (in \cite{KM18} as well as most relevant papers on graphon systems since graphon can be viewed as the kernel of a measure-valued function) to do this task. In the former case as what we plan to do here, we apply results on approximation of probability measures by atomic measures (e.g., \cite{C18,XB19}) to the positive and negative part of a signed measure (for $\eta_0$) with normalization (to make a positive finite measure a probability measure). To make these discretizations of measures feasible, we reasonably need more points (of size $n$, which generically is far larger than $m$) in the support of the approximations for each fixed $m$. With these discretizations of functions, we substitute into the general model \eqref{integralequation} to obtain the integral equation discretization (\eqref{lattice}) of the equation of characteristics in the light of continuous dependence of the solution map $\mathcal{S}^x_t$ (Proposition~\ref{prop-semiflow}). One needs to pay particular attention that replacing each $\phi(t,x)$ one needs $n$ ``copies'' ($\phi_{(i-1)n+j}(t)$ for $1\le j\le n$) instead of just one for each $x\in A^m_i$, which is consistent with that we need $n$ atoms in the approximation of $\eta^x_t$ as well as that there are $mn$ nodes in the underlying graph. The approximation in terms of the empirical distributions of the integral equation of the solution to the generalized Vlasov equation comes from the Dobrushin estimate we established earlier (Proposition~\ref{prop-sol-fixedpoint}). The discretization of the evolving SDGM $\eta_t$ is constructed based on \eqref{Eq-variation-2} by further discretizing the Lebesgue measure on $X$ by atomic measures supported on the mesh points $x^m_i\in A^m_i$ given in Lemma~\ref{le-ini-2} below.}

\begin{lemma}\label{le-ini-2}
Assume $(\mathbf{A1})$ and $\nu_0\in\cC_{\infty}$. {Then there exists a partition $\{A^m_i\}_{1\le i\le m}$ of $X$  for $m\in \N$} satisfying \[\lim_{m\to\infty}\max_{1\le i\le m}\Diam A^m_i=0.\] Let $x_i^m\in A_i^m$, for $i=1,\ldots,m$, $m\in\N$. 
{Moreover,} there exists a sequence $\{\varphi^{m,n}_{(i-1)n+j}\colon i=1,\ldots,m,j=1,\ldots,n\}_{n,m\in\N}\subseteq \mathbb{T}$ such that $$\lim_{m\to\infty}\lim_{n\to\infty}{d_{\infty,\BL}}({\nu^{m,n}_0},\nu_0)=0,$$
where ${\nu^{m,n}_0}\in \mathcal{B}_{\infty}$ with \begin{equation}
  \label{nu-representation}
  {(\nu^{m,n}_0)^x}\coloneqq\sum_{i=1}^m\mathbbm{1}_{A^m_i}(x)\frac{a_{m,i}}{n}\sum_{j=1}^n
\delta_{\varphi^{m,n}_{(i-1)n+j}},\quad x\in X,
\end{equation}
$$a_{m,i}=\begin{cases}
  \frac{\int_{A_i^m}\nu_0^x(\mathbb{T})\rd x}{\lambda(A_i^m)},& \textnormal{if}\quad \lambda(A_i^m)>0,\\
  \nu_0^{x^m_i}(\mathbb{T}),& \textnormal{if}\quad \lambda(A_i^m)=0.
\end{cases}$$
\end{lemma}
\begin{proof}
{Note that the existence of such a partition of $X$ is ensured by \cite[Lemma~4.4]{KX21}.} Let $Y=[0,1]$. Note that $${d_{\mathbb{T}}(x,y)=\min\{|x-y|,1-|x-y|\}\le |x-y|,\quad x,y\in\bT\equiv[0,1[}$$ {and  $\mathcal{BL}_1(\T)\subsetneq\mathcal{BL}_1(Y)$ as the former only consists of the proper subset of $1$-periodic functions (so that the function value coincide at $x=0$ and $x=1$) defiend on $Y$. It follows from the supremum representation of the bounded Lipschitz metric \cite{CM95} (see also e.g., \cite{RDG11,X18,X19}) that}
\begin{equation}\label{inequality-compare}
d_{\BL}(\mu,\nu)\le \widetilde{d}_{\BL}(\mu,\nu),\quad \mu,\nu\in\cM(\mathbb{T}),
\end{equation}where $\widetilde{d}_{\BL}$ stands for the bounded Lipschitz metric on $\cM(Y)$, and any $\mu\in\cM(\bT)$ can be regarded as a measure in $\cP(Y)$ supported in a subset of $[0,1[$. Since $\nu_0\in\cC_{\infty}$, applying \cite[{Lemma~4.5}]{KX21}, there exists ${\widetilde{\nu}^{m,n}_0}\in\cB(X,\cM_+(Y))$ such that
$$\lim_{m\to\infty}\lim_{n\to\infty}\sup_{x\in X}\widetilde{d}_{\BL}({(\widetilde{\nu}^{m,n}_0)^x},\nu_0^x)=0,$$
where $${(\widetilde{\nu}^{m,n}_0)^x}\coloneqq\sum_{i=1}^m\mathbbm{1}_{A^m_i}(x)\frac{a_{m,i}}{n}\sum_{j=1}^n
\delta_{\widetilde{\varphi}^{m,n}_{(i-1)n+j}},\quad x\in X,$$
$$a_{m,i}=\begin{cases}
  \frac{\int_{A_i^m}\nu_0^x(\mathbb{T})\rd x}{\lambda(A_i^m)},& \textnormal{if}\quad \lambda(A_i^m)>0,\\
  \nu_0^{x^m_i}(\mathbb{T}),& \textnormal{if}\quad \lambda(A_i^m)=0,
\end{cases}$$
and $\{\widetilde{\varphi}^{m,n}_{(i-1)n+j}\}_{1\le i\le m,1\le j\le n}\subseteq Y$. For every $x\in X$, let ${(\nu^{m,n}_0)^x}$ be the discrete measure by transporting the mass of the discrete measure ${(\widetilde{\nu}^{m,n}_0)^x}$ at $1$ to that at $0$:
\[{(\nu^{m,n}_0)^x}(z)=\begin{cases}
  {(\widetilde{\nu}^{m,n}_0)^x}(z),& \text{if}\quad z\neq0,\\
  {(\widetilde{\nu}^{m,n}_0)^x}(0)+{(\widetilde{\nu}^{m,n}_0)^x}(1),& \text{if}\quad z=0.
\end{cases}\]
Then ${\nu^{m,n}_0}$ can be represented by \eqref{nu-representation}, and it follows from \eqref{inequality-compare} that
\[d_{\infty,\BL}({\nu^{m,n}_0},\nu_0)=\sup_{x\in X}d_{\BL}({(\nu^{m,n}_0)^x},\nu_0^x)\le\sup_{x\in X}\widetilde{d}_{\BL}({(\widetilde{\nu}^{m,n}_0)^x},\widetilde{\nu}_0^x),\] which immediately yields the conclusion.
 \end{proof}
{
\begin{lemma}\label{le-graph}
Assume $(\mathbf{A1})$ and $(\mathbf{A4})'$. For every $m\in\N$, let $A^m_i$ and $x_i^m$ be defined in Lemma~\ref{le-ini-2} for $i=1,\ldots,m$, $m\in\N$. Then  there exist two sequences $\{y^{k,m,n}_{(i-1)n+j}\colon i=1,\ldots,m,j=1,\ldots,n\}_{m,n\in\N}\subseteq X$ for $k=1,2$ such that $$\lim_{m\to\infty}\lim_{n\to\infty}d_{\infty,\BL}(\eta_0^{m,n},\eta_0)=0,$$
where $\eta_0^{m,n}\in \mathcal{B}(X,\cM(X))$ are given by \begin{equation}\label{eq:discretization-of-eta_0}(\eta_0^{m,n})^x\coloneqq\sum_{i=1}^m\mathbbm{1}_{A^m_i}(x)
\sum_{j=1}^n\Bigl(\frac{b_{1,m,i}}{n}\delta_{y^{1,m,n}_{(i-1)n+j}}-\frac{b_{2,m,i}}{n}\delta_{y^{2,m,n}_{(i-1)n+j}}\Bigr);
\end{equation} for $k=1,2$, \begin{equation}\label{eq:b}
b_{k,m,i}=\begin{cases}
  \frac{\int_{A_i^m}\eta_{0,k}^x(X)\rd x}{\lambda(A_i^m)},& \textnormal{if}\quad \lambda(A_i^m)>0,\\
  \eta_{0,k}^{x_i^m}(X),& \textnormal{if}\quad \lambda(A_i^m)=0,
\end{cases}
\end{equation}
and  $\eta_{0,1}$ and $\eta_{0,2}$ are the positive and negative part of $\eta_0$, respectively.
\end{lemma}
\begin{proof}
Let $\eta_0=\eta_{0,1}-\eta_{0,2}$ be the Hahn decomposition of $\eta_0$, where $\eta_{0,1}$ and $\eta_{0,2}$ are the positive and negative part of $\eta_0$, respectively.  Applying \cite[Lemma~4.6]{KX21} (with $r=1$ therein) to $\eta_{0,1}$ and $\eta_{0,2}$, respectively, we obtain two sequences of points $\{y^{k,m,n}_{(i-1)n+j}\colon i=1,\ldots,m,j=1,\ldots,n\}_{m,n\in\N}\subseteq X$ for $k=1,2$ such that for $k=1,2$, $$\lim_{m\to\infty}\lim_{n\to\infty}d_{\infty,\BL}(\eta_{0,k}^{m,n},\eta_{0,k})=0,$$
where $\eta_{0,k}^{m,n}\in \mathcal{B}(X,\cM_+(X))$ with 
\[(\eta_{0,k}^{m,n})^x\coloneqq\sum_{i=1}^m\mathbbm{1}_{A^m_i}(x)\frac{b_{k,m,i}}{n}
\sum_{j=1}^n\delta_{y^{k,m,n}_{(i-1)n+j}},\quad x\in X,
\] and $b_{k,m,i}$ are given in \eqref{eq:b}. By triangle inequality, we obtain 
$$\lim_{m\to\infty}\lim_{n\to\infty}d_{\infty,\BL}(\eta_{0}^{m,n},\eta_{0})=0,$$
where $\eta_{0}^{m,n}=\eta_{0,1}^{m,n}-\eta_{0,2}^{m,n}$ is given in \eqref{eq:discretization-of-eta_0}.
\end{proof}
}

\begin{remark}
\begin{enumerate}
\item[(i)] {We remark that $\lambda(A^m_i)=0$ is possible, even in the light of $\lambda(X)>0$. Indeed, this is because it is possible that $X=X_1\cup X_2$, where the two subsets $X_1$ and $X_2$ have a positive Hausdorff distance and one of them, say $X_2$, is of Lebesgue measure zero, e.g., $X=[0,1]\cup\{2\}$. In this case, to ensure the diameter of $A_i^m$ tends to zero as $m\to\infty$, $A^m_i$ may only contain points in one of the two sets $X_1$ and $X_2$, and hence $\lambda(A^m_i)=0$ for those $A_i^m$ that partition $X_2$.}
\item[(ii)] {One generally cannot obtain the discrete approximation directly applied to a sign measure in $\cM(X)$ as it with the metric $d_{\BL}$ is \emph{not} complete.}  
\end{enumerate}
\end{remark}
{To associate the points in $\{y^{k,m,n}_{n(i-1)+j}\}_{1\le i\le m; 1\le j\le n}$ with the sets $A^m_i$ of the partition, let $q^{k,m,n}_{i,j}$ be the indices such that $y^{k,m,n}_{n(i-1)+j}\in A^m_{q^{k,m,n}_{i,j}}$, for $k=1,2$,  $i=1,\ldots,m$, $j=1,\ldots,n$.}

{For more examples of discretizations of $\nu_0$ and $\eta_0$ (where the SDGM is a DGM, while one can  construct discretizations of an SDGM by applying the discretizations of DGMs to the positive and negative parts of the SDGM), the reader is referred to \cite[{Section~4}]{KX21}.} Next we provide discretizations of the frequency function $\omega$.

\begin{lemma}\label{le-omega}
Assume $(\mathbf{A1})$, $(\mathbf{A4})$ and $(\mathbf{A4})'$.  Let {$T>0$ and $\cI=[0,T]$}. For every $m\in\N$, let $A^m_i$ and $x_i^m$ be defined in Lemma~\ref{le-ini-2} for $i=1,\ldots,m$, $m\in\N$. For every $m\in\N$, let $$\omega^m(t,z)=\sum_{i=1}^m\mathbbm{1}_{A^m_i}(z)\omega(t,x_i^m),\quad t\in\cI,\ z\in X.$$ Then
  $$\lim_{m\to\infty}\int_0^T\sup_{x\in X}\lt|\omega^m(t,x)-\omega(t,x)\rt|\rd t=0.$$
\end{lemma}
The proof of Lemma~\ref{le-omega} is analogous to \cite[{Lemma~4.9}]{KX21} and hence omitted.

Now we are ready to provide a discretization of the integral equation of characteristics \eqref{integralequation} on an initial {SDGM} $\eta_0$ by a sequence of ODEs characterizing the dynamics of the oscillators coupled on the underlying coevolving graphs. {To make it convincing that we so far obtain all necessary discretizations to construct the desired ODEs, 
{ we summarize the information below: There exist}
\begin{enumerate}
 \item[$\bullet$] a partition $\{A^m_i\}_{1\le i\le m}$ of $X$ and points $x^m_i\in A^m_i$ for $i=1,\ldots,m$, for every $m\in\N$;
  \item[$\bullet$] three sequence of non-negative  numbers $\{a_{m,i}\}_{1\le i\le m},\ \{b_{k,m,i}\}_{1\le i\le m} \subseteq\R_+$ for $k=1,2$, $m\in\N$;
   \item[$\bullet$] a sequence of double indexed points on the circle $\{\varphi^{m,n}_{(i-1)n+j}\}_{1\le i\le m;1\le j\le n}\subseteq \T$ for $n,m\in\N$;
  \item[$\bullet$] {two sequences of double indexed points on $X$ $\{y^{k,m,n}_{(i-1)n+j}\}_{1\le i\le m;1\le j\le n}\subseteq X$, for $k=1,2$ and  $n,m\in\N$;} 
  \item[$\bullet$] {a sequence of frequency functions  $\{\omega_i\}_{1\le i\le m}\subseteq\cC(\cI)$ for $m\in\N$;}
    \item[$\bullet$] a sequence of {SDGMs $\eta_0^{m,n}\in\mathcal{B}(X,\cM(X))$ for $n,m\in\N$, and}  
        \item[$\bullet$] a sequence of {finite discrete  measures 
        $\nu_0^{m,n}\in\cM_+(X)$ for $n,m\in\N$}
 \end{enumerate}
 such that
  $$\lim_{m\to\infty}\lim_{n\to\infty}d_{\infty,\BL}({\nu^{m,n}_0},\nu_0)=0,$$
 $$\lim_{m\to\infty}\lim_{n\to\infty}d_{\infty,\BL}(\eta_0^{m,n},\eta_0)=0,$$
     $$\lim_{m\to\infty}\int_0^T\max_{1\le i\le m}\sup_{x\in A_i^m}\lt|\omega^m_i(t)-\omega(t,x)\rt|\rd t=0,$$
 where
 \begin{alignat}{2}
\label{discretization-3}\omega^m_i(t)=&\  \omega(t,x_i^m),\quad t\in\bT,\ i=1,\ldots,m,\\
\label{discretization-2}{(\eta^{m,n}_0)^x}= &
 \sum_{i=1}^m\mathbbm{1}_{A^m_i}(x)\sum_{j=1}^n
\bigl(\tfrac{b_{1,m,i}}{n}\delta_{y^{1,m,n}_{(i-1)n+j}}-\tfrac{b_{2,m,i}}{n}\delta_{y^{2,m,n}_{(i-1)n+j}}\bigr),\quad x\in X\\
\label{discretization-1}{(\nu^{m,n}_0)^x}= &
 \sum_{i=1}^m\mathbbm{1}_{A^m_i}(x)\sum_{j=1}^n
\frac{a_{m,i}}{n}\delta_{\varphi^{m,n}_{(i-1)n+j}},\quad x\in X.
 \end{alignat}
 }
  
{
Consider the following coupled system of integral equations
\begin{equation}\label{lattice}
\begin{split}
\phi_{(i-1)n+j}(t) = & \varphi_{(i-1)n+j}^{m,n}+\int_0^t w^m_i(s)ds + \int_0^t\e^{-\varepsilon s}\\
&\cdot\Bigl(\frac{b_{1,m,i}}{n}\sum_{\ell=1}^n
  \frac{a_{m,q^{1,m,n}_{i,\ell}}}{n}\sum_{\ell'=1}^ng(\phi_{(q^{1,m,n}_{i,\ell}-1)n+\ell'}(s)-\phi_{(i-1)n+j}(s))\\
  &-\frac{b_{2,m,i}}{n}\sum_{\ell=1}^n
  \frac{a_{m,q^{2,m,n}_{i,\ell}}}{n}\sum_{\ell'=1}^ng(\phi_{(q^{2,m,n}_{i,\ell}-1)n+\ell'}(s)-\phi_{(i-1)n+j}(s))\Bigr)\rd s\\
  &-\Bigl(\varepsilon\int_{0}^t\int_0^s\e^{-\varepsilon(t-\tau)}\sum_{p=1}^m\lambda(A^m_p)\frac{a_{m,p}^2}{n^2}\sum_{\ell=1}^n\sum_{\ell'=1}^ng(\phi^{m,n}_{(p-1)n+\ell}(s)-\phi^{m,n}_{(i-1)n+j}(s))\\
&  \cdot h(\phi^{m,n}_{(p-1)n+\ell'}(\tau)-\phi^{m,n}_{(i-1)n+j}(\tau))\rd\tau \rd s\Bigr) \mod1,\quad  t\in\R,
\end{split}
\end{equation}
}

{The above integral equation is well-posed, as a  consequence of the standard Picard-Lindel\"{o}f iteration.}

\begin{proposition}\label{prop-well-posed-lattice}
Assume $(\mathbf{A1})$-$(\mathbf{A4})$ and $(\mathbf{A4})'$. Let $T>0$.  Then there exists a unique solution {$(\Phi^{m,n}(t)=(\phi^{m,n}_{(i-1)n+j}(t))_{1\le i\le m,\ 1\le j\le n}$ to \eqref{lattice}.}
\end{proposition}

For $t\in\cI$, let $\Phi^{m,n}(t)=(\phi^{m,n}_{(i-1)n+j}(t))_{1\le i\le m;1\le j\le n}$ be the solution to \eqref{lattice}, {and define a sequence of fiber empirical distributions} $({(\nu^{m,n}_{\cdot})})_{m,n\in\N}\subseteq\cC(\cI,\cB_{\infty})$: \begin{equation}\label{Eq-approx}
{(\nu^{m,n}_t)^x}\coloneqq\sum_{i=1}^m\mathbbm{1}_{A^m_i}(x)\frac{a_{m,i}}{n}
\sum_{j=1}^{n}\delta_{\phi^{m,n}_{(i-1)n+j}(t)},\quad x\in X,\quad t\in\cI.
\end{equation}

\begin{theorem}\label{th-approx}
{Assume ($\mathbf{A1}$)-($\mathbf{A5}$) and ($\mathbf{A4}$)'. Let $T>0$ and $\cI=[0,T]$. Assume $\nu_0\in\mathcal{C}_{\infty}$. Let
 ${\eta^{m,n}_0}\in \mathcal{B}(X,\cM(X))$, ${\nu^{m,n}_0}\in\mathcal{B}_{\infty}$, and $\omega^m_i\in \mathcal{C}(\cI)$} and ${(\nu^{m,n}_{\cdot})}$ be defined in \eqref{discretization-1}-\eqref{discretization-3} and \eqref{Eq-approx}, respectively. Let $\nu_{\cdot}$ the solution to the generalized VE \eqref{Fixed} with initial condition $\nu_0$.   Then $$\lim_{n\to\infty}{d^{\cI}_{\infty,\BL}(\nu^{m,n}_{\cdot},\nu_{\cdot})}=0$$ 
\end{theorem}
The proof of Theorem~\ref{th-approx} is provided in Section~\ref{subsect-proof-th-approx}.

\section{An example\---A model on binary tree networks}\label{sect-example}
{In this section, to demonstrate the applicability of our main results obtained in Sections~\ref{sect-Vlasov} and \ref{sect-th-approx}, we provide {one} example where the sequence of initial graphs is not dense. It is noteworthy that despite it is assumed in this paper that $X$ is a compact set of a positive Lebesgue measure, almost the same arguments yield some other interesting cases where $X$ is a circle (and hence as a subset of $\R^2$, it is a Lebesgue measure zero set) and the reference probability measure on $X$ is chosen to be the Haar measure on the circle. We refer the interested reader to other examples including the Kuramoto model on a ring network in an earlier version of this paper \cite[Section~6.1]{GKX22}.}

 \tikzset{
    net node/.style = {circle, inner sep=0pt, outer sep=0pt, ball color=white},
    net root node/.style = {net node},
    net connect/.style = {draw=black},
  treenode/.style = {align=center, inner sep=0pt, text centered,
    font=\sffamily},
  arn_n/.style = {treenode, circle, font=\sffamily, draw=black,
    fill=white, text width=1.5em,thick},
   arn_r/.style = {treenode, circle, draw=black,
   fill=red, text width=1em},
  arn_x/.style = {treenode, circle, draw=black,
   fill=white, opacity=1, text width=1.5em}
}
  \begin{figure}[h]
    \centering
\scalebox{1}{\begin{tikzpicture}[-,thick,level/.style={sibling distance = 6cm/#1,
  level distance = 1.5cm},scale=.8]
\node [arn_x] {1}
    child{ node [arn_x] {2}
            {child{ node [arn_x] {4}
            	{
                child{ node [arn_x] {8}}
				child{ node [arn_x] {9}}
            }}
            child{ node [arn_x] {5}
				{
				child{ node [arn_x] {10}}
                child{ node [arn_x] {11}}
            }}
         }
    }
    child{ node [arn_x] {3}
            {child{ node [arn_x] {6}
            	{
				child{ node [arn_x] {12}}
                child{ node [arn_x] {13}}
            }}
            child{ node [arn_x] {7}
				{
				child{ node [arn_x] {14}}
                child{ node [arn_x] {15}}
            }}
         }
    };
\end{tikzpicture}}
      \caption{Oscillators coupled on binary trees.}
\label{fig-1}
  \end{figure}

{Consider the binary Kuramoto network \begin{equation}
  \label{Kuramoto-model-1}
\begin{split}
\dot{\phi}^N_i=&\ \omega_i^N(t)+\frac{1}{N}\sum_{j=1}^NW^N_{ij}g(\phi^N_j-\phi^N_i),\quad 0<t\le T^*,\\
\dot{W}^N_{ij}=&-\varepsilon(W^N_{ij}+h(\phi_j^N-\phi^N_i)),\quad 0<t\le T^*\\
\phi^N_i(0)=&\ \varphi^N_{i},\quad W^N_{ij}(0)=W^N_{i,j,0},\quad i,j=1,\ldots,N,
\end{split}
\end{equation}
where $T^*>0$, $\omega_i^N(t)$ is the natural time-dependent frequency of the $i$-th oscillator, $h(u)=-\sin^22\pi u$ and $g(u)=\sin2\pi u$,} and the network is
 a sequence of binary trees of $N$ nodes (see Figure~\ref{fig-1}(b)) with
\[W^N_{i,j,0}=N\mathbbm{1}_{\{2i,2i+1,\lfloor i/2\rfloor\}}(j),\quad i,j=1,\ldots,N,\]
for all $N=2^{m+1}-1$ where $m$ is the number of levels of the binary tree.
Let $X=[0,1]$, $I_i^N=\bigl[\frac{i-1}{N},\frac{i}{N}\bigr[$, $i=1,\ldots,N-1$ and {$I^N_N=\bigl[\frac{N-1}{N},1\bigr]$} be a uniform partition of $X$. For every $x\in I^N_i$, $t\in[0,T^*]$, let
\begin{equation}
  \label{discretize-1}
  \phi^{N}(t,x)=\phi^N_i(t),\quad \varphi^N(x)=\varphi^N_i,\quad \omega^N(t,x)=\omega_i^N(t),
\end{equation}\begin{equation}
  \label{discretize-2}
\eta^{x}_{N,0}=\frac{1}{N}\sum_{j=1}^NW_{i,j,0}^N\delta_{\frac{2j-1}{2N}},\quad \nu_{N,t}^{x}=\delta_{\phi_i^N(t)}.\end{equation}
{Define the \emph{empirical distribution} of the network \eqref{Kuramoto-model-1}
\begin{equation}\label{eq:empirical-distribution}
\int_0^1\nu^{x}_{N,t}\rd x\coloneqq\frac{1}{N}\sum_{i=1}^N\delta_{\phi^N_i(t)}
\end{equation}}
 Let
\[\eta^x_0=\begin{cases}
2\delta_0,&\text{if}\ x=0,\\
2\delta_{2x}+\delta_{x/2},& \text{if}\ 0<x\le1/2,\\
\delta_{x/2},& \text{if}\ 1/2<x\le1.
  \end{cases}\]
Note that $x\mapsto\eta^x_0$ is continuous at all $x\in ]0,1/2[\cup]1/2,1]$. 

By Theorem~\ref{thm-characteristic-1} and  Corollary~\ref{cor-flow}, there exists a {solution map} $\mathcal{S}^x_{t}$ generated by
\begin{equation}\label{integralequation-new}
\begin{split}
\phi(t,x)=&\ \Bigl(\phi_0(x)+\int_{0}^t\Bigl(\omega(s,x)+\e^{-\varepsilon s}\int_0^1\int_{\mathbb{T}}\sin\bigl(2\pi(\psi-\phi(s,x))\bigr)
\rd\nu_s^y(\psi)\rd\eta_0^x(y)\\
&+\varepsilon\int_{0}^s\e^{-\varepsilon(t-\tau)}\int_0^1\int_{\mathbb{T}}\sin\bigl(2\pi(\psi-
\phi(s,x))\bigr)\rd\nu_s^y(\psi)\\
&\cdot\int_{\mathbb{T}}\sin^2\bigl(2\pi(\psi-\phi(\tau,x))\bigr)\rd\nu_{\tau}^y(\psi))
\rd y\rd\tau\Bigr)\rd s\Bigr)\mod1
\end{split}\end{equation}

Define the following generalized VE 
\begin{align}\label{Fixed-model2}
  \nu_t^x=\nu_0^x\circ (\mathcal{S}^x_{t})^{-1}[\eta_0,\nu_{\cdot},\omega],\quad x\in [0,1].
\end{align}

\begin{theorem}\label{th-tree}

Assume $\mathbf{(A4)}$ and $\mathbf{(A4)'}$. Let $\nu_0\in\cB_{\infty}$. Then there exists a unique solution $\nu_t$ to  \eqref{Fixed-model2}.
 Moreover, if $\nu_0\in\cC_{\infty}$ and  $\lim_{N=2^{m+1}-1\to\infty}d_{\infty,\BL}(\nu_{N,0},\nu_0)=0$, then
\begin{equation*}
  \lim_{N=2^{m+1}-1\to\infty}d_{\infty,\BL}(\nu_{N,t},\nu_t)=0,\quad \forall t\in[0,T^*].
\end{equation*}
In particular,
\begin{equation*}
\lim_{N=2^{m+1}-1\to\infty}d_{\BL}\Bigl(\frac{1}{N}\sum_{i=1}^N\delta_{\phi^N_i(t)},
\int_0^1\nu_t^x(\sbullet)\rd x\Bigr)=0.\end{equation*}
\end{theorem}
\begin{proof}
It suffices to show that
\[\lim_{N=2^{m+1}-1\to\infty}d_{\infty,\BL}(\eta_{N,0},\eta_0)=0.\]
For any $x\in[0,1[$, let $i=i(x)=\lfloor xN\rfloor+1$. Then $x\in\Bigl[\frac{i-1}{N},\frac{i}{N}\Bigr[$.
For $x=0$, $i=1$, and $W_{i,j,0}^N=N\mathbbm{1}_{\{2,3\}}(j)$. For $x\in\bigl]0,\frac{1}{2}\bigr[$, $W_{i,j,0}^N=N\mathbbm{1}_{\{2i,2i+1,\lfloor i/2\rfloor\}}(j)$. For $x\in\bigl[\frac{1}{2},1\bigr[$, $W_{i,j,0}^N=N\mathbbm{1}_{\{\lfloor i/2\rfloor\}}(j)$. Hence
for $x=0$,
\begin{align*}
  d_{\BL}(\eta^{x}_{N,0},\eta^x_0)=&\sup_{f\in\mathcal{BL}_1([0,1])}\int_0^1f\rd\bigl(\delta_{\frac{3}{2N}}
  +\delta_{\frac{5}{2N}}-2\delta_{0}\bigr)\\
  \le&\ \Bigl|\frac{3}{2N}\Bigr|+\Bigl|\frac{5}{2N}\Bigr|=\frac{4}{N};
  \end{align*}
for $x\in]0,1/2[$,
\begin{align*}
  d_{\BL}(\eta^{x}_{N,0},\eta^x_0)=&\sup_{f\in\mathcal{BL}_1([0,1])}\int_0^1f\rd\bigl(\delta_{\frac{4i-1}{2N}}
  +\delta_{\frac{4i+1}{2N}}+\delta_{\frac{2\lfloor i/2\rfloor-1}{2N}}-2\delta_{2x}-\delta_{x/2}\bigr)\\
  \le&\ \Bigl|\frac{4i-1}{2N}-2x\Bigr|+\Bigl|\frac{4i+1}{2N}-2x\Bigr|+\Bigl|\frac{2\lfloor i/2\rfloor-1}{2N}-\frac{x}{2}\Bigr|\\
  \le&\ \frac{3}{2N}+\frac{5}{2N}+\frac{2}{2N}=\frac{5}{N};
  \end{align*}
  for $x\in\Bigl]\frac{1}{2},1\Bigr[$,
\begin{align*}
  d_{\BL}(\eta^{x}_{N,0},\eta^x_0)=&\sup_{f\in\mathcal{BL}_1([0,1])}\int_0^1f\rd\bigl(\delta_{\frac{2\lfloor i/2\rfloor-1}{2N}}-\delta_{x/2}\bigr)\\
  \le&\ \Bigl|\frac{2\lfloor i/2\rfloor-1}{2N}-\frac{x}{2}\Bigr|\le\frac{1}{N};
  \end{align*}
  for $x=1$, \begin{align*}
  d_{\BL}(\eta^{x}_{N,0},\eta^x_0)=&\sup_{f\in\mathcal{BL}_1([0,1])}\int_0^1f\rd\bigl(\delta_{\frac{2\lfloor N/2\rfloor-1}{2N}}-\delta_{1/2}\bigr)\\
  \le&\ \Bigl|\frac{(N-1)-1}{2N}-\frac{1}{2}\Bigr|=\frac{1}{N}.
  \end{align*}
This implies that
\[d_{\infty,\BL}(\eta_{N,0},\eta_0)\le\frac{5}{N}\to0,\quad \text{as}\quad N\to\infty.\]

\end{proof}

We comment that the limit of such sequence of graphs are \emph{not} dense and hence \emph{cannot} be represented as a ``graphon'' \cite{LS06}, instead, it can be viewed as a symmetric digraph measure \cite{KX21} (see also \cite{KLS19}).

\section{Proofs of main results}\label{sect-proof}
\subsection{Proof of Theorem~\ref{thm-characteristic-1}}\label{subsect-proof-of-characteristic}
\begin{proof}
The proof is in similar spirit to that of \cite[Theorem~A]{GKX21}. For the reader's convenience, we provide a complete proof here. 

 Since $\nu_{\cdot}\in \mathcal{C}_{\sf b}(\R,{\mathcal{B}_{\infty}})$, we have $$\|\nu_{\cdot}\|_{\R}=\sup_{t\in\R}\|\nu_t\|<\infty.$$
For simplicity, let $\mathcal{N}=[0-t_*,0+t_*]$ for any fixed $0<t_*<\varepsilon^{-1}$, {where we recall that $\varepsilon$ is slow time-scale of the underlying {SDGM} in \eqref{characteristic-Eq-2}.}
In the following, we prove the conclusions in several steps. First we show the solution exists locally in a subset of $\mathcal{C}(\mathcal{N},\mathcal{C}(X,\T\times\cM(X)))$ {(Step 1 and Step 2)}, and then we prove the uniqueness of solutions in $\mathcal{C}(\mathcal{N},\mathcal{C}(X,\T\times\cM(X)))$ using Gronwall inequality {(Step 3)}. Then we extend the solution to an open maximal existence interval, and use a priori estimates to show global existence {(Step 4)}. 

To show that the solution uniquely exists in the bigger space $\mathcal{C}(\mathcal{N},\mathcal{B}(X,\T\times\cM(X)))$, all the arguments still remain, by simply replacing $\mathcal{C}(\mathcal{N},\mathcal{C}(X,\T\times\cM(X)))$ by $\mathcal{C}(\mathcal{N},\mathcal{B}(X,\T\times\cM(X)))$. Note that Step~1-(c) below is not needed in this case.

 {For $(\phi,\eta_{\cdot}),(\varphi,\xi_{\cdot})\in \mathcal{C}(\mathcal{N},\mathcal{B}(X,\mathbb{T}\times\cM(X)))$, $t\in\cN$, $x\in X$, define}
{\[d_{\T,\TV}\left((\phi(t,x),\eta_t^x),(\varphi(t,x),\xi_t^x)\right)\coloneqq
d_{\mathbb{T}}(\phi(t,x),\varphi(t,x))+d_{\TV}(\eta_t^x,\xi_t^x),\]
$${d}_{\infty,\T,\TV}((\phi(t),\eta_t),(\varphi(t),\xi_t))\coloneqq\sup_{x\in X}d_{\T,\TV}((\phi(t,x),\eta_t^x),(\varphi(t,x),\xi_t^x)),$$
\[d^{\cN}_{\infty,\T,\TV}((\phi,\eta_{\cdot}),(\varphi,\xi_{\cdot}))\coloneqq\sup_{t\in\cN}d_{\infty,\T,\TV}((\phi(t),\eta_t),(\varphi(t),\xi_t))\]}

{Note that $\mathcal{C}(\mathcal{N},\mathcal{B}(X,\T\times\cM(X)))$  and $\mathcal{C}(\mathcal{N},\mathcal{C}(X,\T\times\cM(X)))$ are both complete metric spaces under the uniform metric $d^{\mathcal{N}}_{\infty,\T,\TV}$, which can be readily proved as \cite[Proposition~2.6]{KX21}.} To show the local existence of solutions, we will construct a subspace $\Omega$ of $\mathcal{C}(\mathcal{N},\mathcal{C}(X,\T\times\cM(X)))$ and apply the Banach fixed point theorem on the space $\Omega$.

Let $\sigma\ge\frac{(\|\eta_0\|+\|\nu_{\cdot}\|_{\cN}\|h\|_{\infty})}{(\varepsilon t_*)^{-1}-1}$ and $$\Omega=\{(\phi,\eta)\in \mathcal{C}(\mathcal{N},\mathcal{C}(X,\T\times\cM(X)))\colon \phi(0,x)=\phi_0(x),\ \eta_{t}^x|_{t=0}=\eta_0^x,\ \forall x\in X;\ \|\eta_{\cdot}-(\eta_0)_{\cdot}\|_{\cN}\le\sigma\}$$ Here we abuse $\eta_0\in \mathcal{C}(\mathcal{N}, \cC(X,\cM(X))$ for the constant function $$(\eta_0)_t^x\equiv\eta_0^x\ \text{for}\ t\in\mathcal{N}\ \text{and}\ x\in X.$$ {By the assumption  {$\mathbf{(A5)}$}, $\eta_0\in\mathcal{C}(X,\T\times\cM(X))$.} The space $\Omega$ is complete since it is a closed subset of the complete metric space $\mathcal{C}(\mathcal{N},\mathcal{C}(X,\T\times\cM(X)))$. Define the operator $\mathcal{A}=(\mathcal{A}^{1},\mathcal{A}^{2})=\{\mathcal{A}^{x}\}_{x\in X}=\{(\mathcal{A}^{1,x},\mathcal{A}^{2,x})\}_{x\in X}$ from $\Omega$ to $\Omega$: For every $x\in X$ and $(\phi,\eta)\in\Omega$,
\begin{align}\label{characteristic-1}
\mathcal{A}^{1,x}(\phi,\eta)(t)=&\ \Bigl(\phi_0(x)+\int_{0}^t\Bigl(\omega(\tau,x)
+\int_X\int_{\sT}g(\psi-\phi(\tau,x))
\rd\nu_{\tau}^y(\psi)\rd\eta_{\tau}^x(y)
\Bigr)\rd\tau\Bigr)\mod1\\
\label{characteristic-2}(\mathcal{A}^{2,x}(\phi,\eta)(t))(\sbullet)=&\ \eta_0^x(\sbullet)
-\varepsilon\int_{0}^t\eta_{\tau}^x(\sbullet)\rd\tau -\varepsilon\left(\int_{0}^t\int_{\sT}h(\psi-\phi(\tau,x))\rd\nu_{\tau}^{\sbullet}(\psi)\rd\tau\right)
\lambda(\sbullet),\quad y\in X.
\end{align}

In Steps~1 and 2 below, we will show that the $n$-th iteration $\mathcal{A}^n$ for some large $n\in\N$ is a contraction from $\Omega$ to $\Omega$.

\begin{enumerate}
  \item[Step~1.] $\mathcal{A}$ is a mapping from $\Omega$ to $\Omega$.

  \begin{enumerate}
    \item[Step~1(a).] It is obvious that $\mathcal{A}^{1,x}(\phi,\eta)(t)\in\mathbb{T}$, since \eqref{characteristic-1} is regarded as an equation modulo 1. That $\mathcal{A}^{2,x}(\phi,\eta)(t)\in\cM(X)$ for $t\in\mathcal{N}$ follows from
\[\sup_{\tau\in\mathcal{N}}\sup_{y\in X}\left|\int_{\sT}h(\psi-\phi(\tau,x))\rd\nu_{\tau}^y(\psi)\right|\le\|h\|_{\infty}\|\nu_{\cdot}\|_{\cN}\]as well as
\[\sup_{t\in\mathcal{N}}|\eta_t^x|(X)\le\|\eta_0\|+\sigma<\infty.\]
    \item[Step~1(b).] {$t\mapsto \mathcal{A}(\phi,\eta)(t)$ is continuous. 
    Let $t,\ t'\in\mathcal{N}$ with $t<t'$. First,
       \begin{align*}
       &d_{\T}(\mathcal{A}^{1,x}(\phi,\eta)(t),\mathcal{A}^{1,x}(\phi,\eta)(t'))\\
       \le&\ |\mathcal{A}^{1,x}(\phi,\eta)(t)-\mathcal{A}^{1,x}(\phi,\eta)(t')|\\
       \le&\int_t^{t'}\left|\omega(\tau,x)+\int_X\int_{\sT}g(\psi-\phi(\tau,x))
\rd\nu_{\tau}^y(\psi)\rd\eta_{\tau}^x(y)
\right|\rd\tau\\
\le&\left(\sup_{\tau\in[t,t']}|\omega(\tau,x)|+\|g\|_{\infty}\|\nu_{\cdot}\|_{\cN}\sup_{t\in\cN}
|\eta_t^x|(X)\right)|t-t'|\\
\le&\left(\sup_{\tau\in[t,t']}|\omega(\tau,x)|+\|g\|_{\infty}\|\nu_{\cdot}\|_{\cN}(\|\eta_0\|
+\sigma)\right)|t-t'|
       \end{align*}

Moreover,  \begin{align*}
    &d_{\TV}(\mathcal{A}^{2,x}(\phi,\eta)(t),\mathcal{A}^{2,x}(\phi,\eta)(t'))\\
    =&\sup_{f\in\mathcal{B}_1(X)}\int_Xf\rd\left(\mathcal{A}^{2,x}(\phi,\eta)(t)
    -\mathcal{A}^{2,x}(\phi,\eta)(t')\right)\\
    \le&\ \varepsilon\int_t^{t'}\sup_{f\in\mathcal{B}_1(X)}\int_Xf(y)\rd\left(-\eta_{\tau}^x(y)
    -\int_{\T}h(\psi-\phi(\tau,x))
    \rd\nu_{\tau}^y(\psi)\rd y\right)\rd\tau\\
    \le&\ \varepsilon|t'-t|\left(\sup_{\tau\in\cN}\|\eta_0^x-\eta_{\tau}^x\|_{\cN}
    +\sup_{\tau\in\cN}\bigl\|\eta_0^x(\sbullet)+\int_{\T}
    h(\psi-\phi(\tau,x))\rd\nu_{\tau}^{\sbullet}(\psi)\lambda(\sbullet)\bigr\|_{\cN}\right)\\
    \le&\ \varepsilon|t'-t|(\sigma+\|\eta_0\|+\|\nu_{\cdot}\|_{\cN}\|h\|_{\infty})
  \end{align*}
Together it yields
\begin{align*}
&d_{\infty,\T,\TV}(\mathcal{A}(\phi,\eta)(t),\mathcal{A}(\phi,\eta)(t'))\\
\le& |t'-t|\Bigl(\varepsilon(\sigma+\|\eta_0\|+\|\nu_{\cdot}\|_{\cN}\|h\|_{\infty}) + \sup_{\tau\in[t,t']}|\omega(\tau,x)|+\|g\|_{\infty}\|\nu_{\cdot}\|_{\cN}(\|\eta_0\|
+\sigma)\Bigr)\\
\to&0,\quad \text{as}\quad |t-t'|\to0.
\end{align*}}
  \item[Step~1(c).] We show for each fixed $t\in\mathcal{N}$, $x\mapsto\mathcal{A}^{x}(\phi,\eta)(t)$ is continuous provided $x\mapsto\eta_0^x$ is so {by $(\mathbf{A5})$.}

   The continuity of $\mathcal{A}^{1,x}(\phi,\eta)(t)$ in $x$ follows from $(\mathbf{A2})$, $(\mathbf{A4})'$, the continuity of $x\mapsto\phi(t,x)$, as well as the fact that continuity of $x\mapsto \eta_{\tau}^{x}$ in total variation distance implies that in bounded Lipschitz distance, which further implies their  weak continuity, by {applying Proposition~\ref{prop-nu}(iii)}.  
  
  Next, we verify the continuity of $x\mapsto\mathcal{A}^{2,x}(\phi,\eta)(t)$. For $x,x'\in X$,
  \begin{align*}
    &d_{\TV}(\mathcal{A}^{2,x}(\phi,\eta)(t),\mathcal{A}^{2,x'}(\phi,\eta)(t))\\
    =&\sup_{f\in\mathcal{B}_1(X)}\int_Xf\rd\left(\mathcal{A}^{2,x}(\phi,\eta)(t)
    -\mathcal{A}^{2,x'}(\phi,\eta)(t)\right)\\
    \le&\ d_{\TV}(\eta_0^x,\eta_0^{x'})+\varepsilon\left|\int_{0}^{t}d_{\TV}(\eta_{\tau}^x,\eta_{\tau}^{x'})
    \rd\tau\right|\\
    &+\varepsilon\left|\int_{0}^{t}\int_X\int_{\T}|h(\psi-\phi(\tau,x))-h(\psi-\phi(\tau,x'))|
    \rd\nu_{\tau}^y(\psi)\rd y\rd\tau\right|\\
    \le&\ d_{\TV}(\eta_0^x,\eta_0^{x'})+\varepsilon\left|\int_{0}^{t}d_{\TV}(\eta_{\tau}^x,\eta_{\tau}^{x'})
    \rd\tau\right|\\
    &+\varepsilon\Lip(h)\|\nu_{\cdot}\|_{\cN}\left|\int_{0}^{t}|\phi(\tau,x)-\phi(\tau,x')|\rd\tau\right|\to0,\quad \text{as}\quad {|x-x'|}\to0,
  \end{align*}
by the Dominated Convergence Theorem, since for every $\tau\in\cN$, 
$$d_{\TV}(\eta_{\tau}^x,\eta_{\tau}^{x'}),\ |\phi(\tau,x)-\phi(\tau,x')|\to0,\quad \text{as}\ {|x-x'|}\to0,$$ due to $\eta_{\tau}\in \mathcal{C}(X,\cM(X))$ and $\phi(\tau,\cdot)\in\mathcal{C}(X,\T)$.
\item[Step~1(d).] {We show $$\|\mathcal{A}^{2}(\phi,\eta)-\eta_0\|_{\cN}\le\sigma.$$ Indeed, since $\mathcal{A}^{2}(\phi,\eta)(0)=\eta_0$, by Step~1(b),
 \begin{align*}
 \|\mathcal{A}^{2}(\phi,\eta)-\eta_0\|_{\cN}=&\sup_{t\in\mathcal{N}}\|\mathcal{A}^{2}(\phi,\eta)(t)-\eta_0\|\\
   \le&\ \sup_{t\in\mathcal{N}} \varepsilon|t|(\sigma+\|\eta_0\|+\|\nu_{\cdot}\|_{\cN}\|h\|_{\infty})\\
   \le&\ \varepsilon t_*(\sigma+\|\eta_0\|+\|\nu_{\cdot}\|_{\cN}\|h\|_{\infty})\le\sigma
 \end{align*}}
   \end{enumerate}
  \item[Step~2.] The aim is to prove that $\mathcal{A}^n$ is a contraction for some $n\in\N$. 
{Let $(\phi,\eta),\ (\varphi,\zeta)\in\Omega$. Then
  \begin{align*}
   &d_{\T}(\mathcal{A}^{1,x}(\phi,\eta)(t),\mathcal{A}^{1,x}(\varphi,\zeta)(t))\\ \le&\left|\mathcal{A}^{1,x}(\phi,\eta)(t)-\mathcal{A}^{1,x}(\varphi,\zeta)(t)\right|\\
    \le&\left|\int_{0}^t\int_X\int_{\T}(g(\psi-\phi(\tau,x))-g(\psi-\varphi(\tau,x)))
\rd\nu_{\tau}^y(\psi)\rd\eta_0^x(y)\rd\tau\right|\\
&+\left|\int_{0}^t\int_X\int_{\T}(g(\psi-\phi(\tau,x))-g(\psi-\varphi(\tau,x)))
\rd\nu_{\tau}^y(\psi)\rd(\eta_{\tau}^x(y)-\eta_0^x(y))\rd\tau\right|\\&+\left|\int_{0}^t\int_X
\int_{\T}g(\psi-\varphi(\tau,x))\rd\nu_{\tau}^y(\psi)
\rd\left(\eta_{\tau}^x(y)-\zeta_{\tau}^x(y)\right)\rd\tau\right|\\
\le&\ \Lip(g)\|\nu_{\cdot}\|_{\cN}\|\eta_0\|\left|\int_{0}^t|\phi(\tau,x)-\varphi(\tau,x)|\rd\tau\right|\\
&+\Lip(g)\|\nu_{\cdot}\|_{\cN}\sup_{\tau\in\cN}d_{\TV}(\eta_{\tau}^x,\eta_0^x)\left|\int_{0}^t
|\phi(\tau,x)-\varphi(\tau,x)|\rd\tau\right|\\
&+\|g\|_{\infty}\|\nu_{\cdot}\|_{\cN}\left|\int_{0}^td_{\TV}(\eta_{\tau}^x,\zeta_{\tau}^x)\rd\tau\right|\\
\le&\ |t|\left(\Lip(g)\|\nu_{\cdot}\|_{\cN}\Bigl(\|\eta_0\|+\sup_{\tau\in\cN}d_{\TV}(\eta_{\tau}^x,
\eta_0^x)\Bigr)\sup_{\tau\in\cN}
|\phi(\tau,x)-\varphi(\tau,x)|\right.\\
&\left.+\|g\|_{\infty}\|\nu_{\cdot}\|_{\cN}\sup_{\tau\in\cN}d_{\TV}(\eta_{\tau}^x,\zeta_{\tau}^x))\right)\\
\le&\ |t|\left(\Lip(g)\|\nu_{\cdot}\|_{\cN}(\|\eta_0\|+\sigma)\sup_{\tau\in\cN}
|\phi(\tau,x)-\varphi(\tau,x)|\right.\\&\left.+\|\nu_{\cdot}\|_{\cN}\|g\|_{\infty}\sup_{\tau\in\cN}
d_{\TV}(\eta_{\tau}^x,\zeta_{\tau}^x)\right)\\
\le&\ |t|M_{1}\sup_{\tau\in\cN}d_{\infty,\T,\TV}((\phi(\tau,x),\eta_{\tau}^x),(\varphi(\tau,x),\zeta_{\tau}^x)),
  \end{align*}
where $M_{1}=\|\nu_{\cdot}\|_{\cN}(\Lip(g)(\|\eta_0\|+\sigma)+\|g\|_{\infty})$.
   Similarly,
  \begin{align*}
&d_{\TV}(\mathcal{A}^{2,x}(\phi,\eta)(t),\mathcal{A}^{2,x}(\varphi,\zeta)(t))\\
=&\sup_{f\in\mathcal{B}_1(X)}\int_Xf\rd\left(\mathcal{A}^{2,x}(\phi,\eta)(t)-
\mathcal{A}^{2,x}(\varphi,\zeta)(t)\right)\\
\le&\ \varepsilon\left(\left|\int_{0}^td_{\TV}(\eta_{\tau}^x,\zeta_{\tau}^x)\rd\tau\right|\right.\\&\left.
+\left|\int_{0}^t\int_X\int_{\T}|h(\psi-\phi(\tau,x))
-h(\psi-\varphi(\tau,x))|\rd\nu_{\tau}^y(\psi)\rd y\rd\tau\right|\right)\\
\le&\ \varepsilon\left(\left|\int_{0}^td_{\TV}(\eta_{\tau}^x,\zeta_{\tau}^x)\rd\tau\right|+
\Lip(h)\|\nu_{\cdot}\|_{\cN}\left|\int_{0}^t|\phi(\tau,x)-\varphi(\tau,x)|\rd\tau\right|\right)\\
\le&\ \varepsilon|t|\Bigl(\sup_{\tau\in\cN}d_{\TV}(\eta_{\tau}^x,\zeta_{\tau}^x)
+\Lip(h)\|\nu_{\cdot}\|_{\cN}\sup_{\tau\in\cN}|\phi(\tau,x)-\varphi(\tau,x)|\Bigr)\\
\le&\ |t|M_{2}\sup_{\tau\in\cN}d_{\T,\TV}((\phi(\tau,x),\eta_{\tau}^x),(\varphi(\tau,x),\zeta_{\tau}^x)),
  \end{align*}
  where $M_{2}=\varepsilon(1+\Lip(h)\|\nu_{\cdot}\|_{\cN})$.}

{Hence \begin{align*}
   d_{\T,\TV}(\mathcal{A}^x(\phi,\eta)(t),\mathcal{A}^x(\varphi,\zeta)(t))
   \le&\ |t|M_3\sup_{\tau\in\cN}d_{\T,\TV}((\phi(\tau,x),\eta_{\tau}^x),(\varphi(\tau,x),\zeta_{\tau}^x)),
\end{align*}
where $M_3=\max\{M_{1},M_{2}\}$. This implies that
\begin{align*}
  d_{\infty,\T,\TV}(\mathcal{A}(\phi,\eta)(t),\mathcal{A}(\varphi,\zeta)(t))
  \le&\ |t|M_3\sup_{\tau\in\cN}d_{\infty,\T,\TV}((\phi(\tau,\cdot),\eta_{\tau}),(\varphi(\tau,\cdot),\zeta_{\tau})).
\end{align*}}

Moreover, from the above estimates we can further prove that
{\begin{alignat}
  {2}\label{inequality-contraction}
  d_{\T,\TV}\left(\mathcal{A}^{x}(\phi,\eta)(t),\mathcal{A}^{x}(\varphi,\zeta)(t)\right)\le M_3\left|\int_{0}^td_{\T,\TV}((\phi(\tau,x),\eta_{\tau}^x),(\varphi(\tau,x),\zeta_{\tau}^x))\rd\tau\right|.
\end{alignat}
Repeatedly applying \eqref{inequality-contraction} yields: For $n\in\N$,
\begin{alignat}
  {2}
  d_{\T,\TV}\left((\mathcal{A}^{x}(\phi,\eta))^n(t),(\mathcal{A}^{x}(\varphi,\zeta))^n(t)\right)\le \frac{(M_3|t|)^n}{n!}\sup_{\tau\in\cN}d_{\T,\TV}((\phi(\tau,x),\eta_{\tau}^x),(\varphi(\tau,x),\zeta_{\tau}^x)),
\end{alignat}
which further implies that
\[d^{\cN}_{\infty,\T,\TV}\left((\mathcal{A}(\phi,\eta))^n,(\mathcal{A}(\varphi,\zeta))^n\right)\le \frac{(M_3t_*)^n}{n!}d^{\cN}_{\infty,\T,\TV}((\phi,\eta),
(\varphi,\zeta)).\]
Hence there exists some large $n\in\N$ such that $\frac{(M_3t_*)^n}{n!}<1$ and hence $\mathcal{A}^m$ is a contraction for all $m\ge n$. By the Banach contraction principle, there exists a unique solution $\mathcal{T}_{t}(\phi_0,\eta_0)$ for $t\in\cN$ in 
{$\Omega\subseteq\mathcal{C}(\mathcal{N},\mathcal{C}(X,\T\times\cM(X)))$} 
to the equation \eqref{characteristic-Eq-1}-\eqref{characteristic-Eq-2} of characteristics.
\item[Step~3.] In Steps~1 and 2, we only obtained the uniqueness within $\Omega$. Next, we show that the solution is unique in ${\mathcal{C}(\mathcal{N},\mathcal{C}(X,\mathbb{T}\times\cM(X)))}$. Let $(\phi,\eta),\ (\varphi,\zeta)\in {\mathcal{C}(\mathcal{N},\mathcal{C}(X,\mathbb{T}\times\cM(X)))}$ be two solutions to the IVP of \eqref{characteristic-Eq-1}-\eqref{characteristic-Eq-2} with $(\phi,\eta)\in\Omega$. Similar as in \eqref{inequality-contraction}, one can show: For $t\ge 0$,
    \begin{align*}
      d_{\infty,\T,\TV}((\phi(t,\cdot),\eta_t),(\varphi(t,\cdot),\zeta_t))\le M_3\int_{0}^t d_{\infty,\T,\TV}((\phi(\tau,\cdot),\eta_{\tau}),(\varphi(\tau,\cdot),\zeta_{\tau}))\rd\tau
      \end{align*}
    which implies by Gronwall's inequality that
    \begin{equation}\label{Eq-unique}
      d_{\infty,\T,\TV}((\phi(t,\cdot),\eta_t),(\varphi(t,\cdot),\zeta_t))=0.
    \end{equation}
    Similarly, one can show that \eqref{Eq-unique} holds for $t\le 0$. Hence $(\phi(t,\cdot),\eta_t)=(\varphi(t,\cdot),\zeta_t)$ for all $t\in\cN$. This shows that the solution to the IVP of {\eqref{characteristic-Eq-1}-\eqref{characteristic-Eq-3}} is unique in the entire set ${\mathcal{C}(\mathcal{N},\mathcal{C}(X,\mathbb{T}\times\cM(X)))}$.} 
\item[Step~4.] By Zorn's lemma, one can always extend the solution by repeating Steps~1-3 indefinitely up to a maximal existence time $T^+_{0}$ with the dichotomy:
    \begin{enumerate}
    \item[(i)] $\lim_{t\uparrow T^+_{0}}|\mathcal{T}^1_{t}(\phi_0,\eta_0)|+\|\mathcal{T}^2_{t}(\phi_0,\eta_0)\|=\infty$;
    \item[(ii)] $T^+_{0}=+\infty$.
    \end{enumerate}
    Note that $|\mathcal{T}^1_{t}(\phi_0,\eta_0)|\le1$ since $\mathcal{T}^1_{t}(\phi_0,\eta_0)\in\mathbb{T}$.
    Moreover, by \eqref{Eq-variation-2} in Proposition~\ref{prop-equivalent},
    \begin{align*}
    \|\mathcal{T}^2_{t}(\phi_0,\eta_0)\|\le&\ \|\eta_0\|+\|\nu_{\cdot}\|_{\cN}\|h\|_{\infty}\varepsilon\int_{0}^t\e^{-\varepsilon(t-\tau)}\rd\tau
    \le\|\eta_0\|+\|\nu_{\cdot}\|_{\cN}\|h\|_{\infty},\quad t\ge 0,
    \end{align*}
    which implies that case (i) will never occur. Hence $T^+_{0}=+\infty$. Analogously, one can show the minimal existence time $T^-_{0}=-\infty$. This shows that the solution globally exists on $\R$.

\end{enumerate}
We have completed the proof.
\end{proof}

\subsection{Proof of Theorem~\ref{thm-existence-VE}}\label{subsect-proof-th-existence}

\begin{proof}
{Note that $\cB_{\infty}$ is a closed subset of $\mathcal{B}(X,\cM_+(\T))$. The unique existence of solutions to the generalized VE is a result of the Banach contraction principle applied in the complete metric space 
 $\mathcal{C}(\cI,\cB_{\infty})$ \cite[{Proposition~2.11}]{KX21} under a dilated  metric $d_{\infty,\BL}^{\cI,\alpha}\coloneqq \sup_{t\in\cI}\e^{-\alpha t}d_{\infty,\BL}(\eta_t,\xi_t)$,
with an appropriately chosen $\alpha>0$, due to the mass conservation law in Proposition~\ref{prop-continuousdependence}(i).} The arguments are analogous to those in the proof of \cite[{Proposition~3.5}]{KX21}, based on Proposition~\ref{prop-continuousdependence}.

{Note that $\nu_{\cdot}\in\mathcal{C}(\cI,\mathcal{C}_{\infty})$ follows from Proposition~\ref{prop-continuousdependence}(i).}
\end{proof}

\subsection{Proof of Theorem~\ref{th-approx}}\label{subsect-proof-th-approx}

\begin{proof}
{The idea of the proof is analogous to that of \cite[{Theorem~4.11}]{KX21}. We will prove the approximation of the solution to the generalized Vlasov equation in four steps.} 

{
Based on the continuous dependence of the solutions on the frequency function, the SDGM, as well as the initial distribution (Proposition~\ref{prop-sol-fixedpoint}), we aim to construct three auxilliary generalized VEs, each replacing one of the three variables by their approximation, and show convergence using the triangle inequalities.} 
\medskip

{\noindent Step I. Show that ${(\nu^{m,n}_{\cdot})}$ defined in \eqref{Eq-approx} is the solution to the generalized VE associated with ${\eta^{m,n}_0}$ and $\omega^m$: \begin{equation}\label{Eq-fixed-approx}{(\nu^{m,n}_{\cdot})}
=\cF[{\eta^{m,n}_0},{(\nu^{m,n}_{\cdot})},\omega^m]{(\nu^{m,n}_{\cdot})}.\end{equation}
To prove this, we calculate $\cF[{\eta^{m,n}_0},{(\nu^{m,n}_{\cdot})},\omega^m]{(\nu^{m,n}_{\cdot})}$ explicitly and show that ${(\nu^{m,n}_{\cdot})}$ satisfies \eqref{Eq-fixed-approx}. Then by the uniqueness of solutions from Theorem~\ref{thm-existence-VE}, we prove that ${(\nu^{m,n}_{\cdot})}$ is the unique solution to \eqref{Eq-fixed-approx}. We first need to examine the equation of characteristics \eqref{integralequation} associated with ${\eta^{m,n}_0}$ and $\omega^m$.} 
 
{We first prove the equivalence of \eqref{lattice} and \eqref{integralequation} associated with ${\eta^{m,n}_0}$ and $\omega^m$.}

{ 
Let $W_{k,i,0}^{m,n}=b_{k,m,i}$ and $W^{m,n}_{0}=(W^{m,n}_{1,0},W^{m,n}_{2,0})$ with $W^{m,n}_{k,0}=(W^{m,n}_{k,i,0})_{1\le i\le m}$ for $k=1,2$. 
By Proposition~\ref{prop-well-posed-lattice},  let $\cQ_{t}[W^{m,n}_0,\omega^m]$ be the solution map generated by  \eqref{lattice}  such that $$\Phi^{m,n}(t)=\cQ_{t}[W^{m,n}_0,\omega^m]\Phi^{m,n}_0,\quad \text{with}\quad \Phi^{m,n}_0=(\varphi^{m,n}_{(i-1)n+j})_{1\le i\le m,1\le j\le n}.$$
In the following, we verify that this 
solution map coincides with 
$\mathcal{S}^x_{t}[\eta^{m,n}_0,\nu^{m,n}_{\cdot},\omega^m]$. More precisely, by the uniqueness of solutions of \eqref{lattice} as well as those of \eqref{integralequation}, it suffices to show that solutions to the following equation solve \eqref{lattice}:  For $i=1,\ldots,m$, for $x\in A^m_i$, $j=1,\ldots,n$,
\begin{equation}\label{Eq-discrete-semiflow}\phi_{(i-1)n+j}(t)
=\cS^x_{t}[\eta^{m,n}_{0},\nu^{m,n}_{\cdot},\omega^m]\varphi^{m,n}_{(i-1)n+j},\end{equation}
where $\eta^{m,n}_{0}=\eta^{m,n}_{0,1}-\eta^{m,n}_{0,2}$  is given by: For $k=1,2$, $i=1,\ldots,m$, $x\in A_i^m$,
$$(\eta^{m,n}_{0,k})^{x}=\frac{W^{m,n}_{k,i,0}}{n}\sum_{j=1}^n\delta_{y^{k,m,n}_{(i-1)n+j}},$$
and
$$(\nu^{m,n}_{t})^x\coloneqq\frac{a_{m,i}}{n}\sum_{j=1}^n
\delta_{\phi^{m,n}_{(i-1)n+j}(t)},$$
\begin{equation}
  \label{omega}\omega^m(t,x)=\omega^m_i(t).
\end{equation}
Next, we explicitly calculate each term in \eqref{integralequation} associated with $\eta^{m,n}_{0,k}$, $\nu^{m,n}_{t}$ and $\omega^m$. Note that for $x\in A_i^m$, 
\begin{align}
 \nonumber  &\int_X\int_{\T}g(\psi-\phi(s,x))\rd(\nu^{m,n}_s)^y(\psi)\int_{\T}
h(\psi-\phi(\tau,x)))
  \rd(\nu^{m,n}_{\tau})^y(\psi)\rd y\\
  \nonumber   = &\ \sum_{p=1}^m\lambda(A^m_p)\frac{a_{m,p}}{n}\sum_{\ell=1}^n
  g(\phi^{m,n}_{(p-1)n+\ell}(s)-\phi(s,x))\cdot\frac{a_{m,p}}{n}\sum_{\ell'=1}^nh(\phi^{m,n}_{(p-1)n+\ell'}(\tau)-\phi^{m,n}_{(i-1)n+j}(\tau))\\
 \label{2-term} = &\ \sum_{p=1}^m\lambda(A^m_p)\frac{a_{m,p}^2}{n^2}\sum_{\ell=1}^n\sum_{\ell'=1}^n
  g(\phi^{m,n}_{(p-1)n+\ell}(s)-\phi(s,x))h(\phi^{m,n}_{(p-1)n+\ell'}(\tau)-\phi(\tau,x)),
\end{align}
\begin{align}
 \nonumber    &\int_X\int_{\T}g(\psi-\phi(s,x))\rd(\nu^{m,n}_{s})^y(\psi)\rd(\eta^{m,n}_{0,k})^{x}(y)\\
 \nonumber    = &\ \frac{W^{m,n}_{k,i,0}}{n}\sum_{\ell=1}^n\int_{\T}g(\psi-\phi(s,x))
  \rd(\nu^{m,n}_{s})^{y^{k,m,n}_{(i-1)n+\ell}}(\psi)\\
 \nonumber    = &\ \frac{W^{m,n}_{k,i,0}}{n}\sum_{\ell=1}^n\sum_{p=1}^n
  \mathbbm{1}_{A_p^m}(y^{k,m,n}_{(i-1)n+\ell})
  \frac{a_{m,p}}{n}\sum_{\ell'=1}^ng(\phi^{m,n}_{(p-1)n+\ell'}(s)-\phi(s,x))\\
 \label{3-term} = &\ \frac{W^{m,n}_{k,i,0}}{n}\sum_{\ell=1}^n
  \frac{a_{m,q^{k,m,n}_{i,\ell}}}{n}\sum_{\ell'=1}^ng(\phi^{m,n}_{(q^{k,m,n}_{i,\ell}-1)n+\ell'}(s)
  -\phi(s,x)),
\end{align}
where we recall $1\le q^{m,n}_{i,\ell}\le m$ such that $y^{m,n}_{(i-1)n+\ell}\in A^m_{q^{m,n}_{i,\ell}}$.}

Plugging the above four expressions \eqref{omega},  \eqref{2-term} and \eqref{3-term} into \eqref{integralequation} yields that solutions to \eqref{Eq-discrete-semiflow} are $\Phi^{m,n}=(\phi^{m,n}_{(i-1)n+j})_{1\le i\le m; 1\le j\le n}$ that solves \eqref{lattice}. 

Hence we can conclude that \begin{equation}\label{Eq-discrete-fixed-point}(\nu^{m,n}_{t})^x={(\nu^{m,n}_0)^x}\circ \Bigl(\cS^x_{t}[\eta_0^{m,n},\nu^{m,n}_{\cdot},\omega^m]\Bigr)^{-1},\quad x\in X,\end{equation}i.e., \eqref{Eq-fixed-approx} holds. To see this,
pick an arbitrary Borel measurable set $B\in\mathcal{B}(\mathbb{T})$, let $f=\mathbbm{1}_B$.
Then for $x\in A^m_i$,
\begin{align*}
  &\int_{\sT} f\rd\nu_{m,n,0}^x\circ \Bigl(\cS^x_{t}[{\eta^{m,n}_0},{(\nu^{m,n}_{\cdot})},\omega^m]\Bigr)^{-1}\\
  = &\ \int_{\sT} f\circ \cS^x_{t}[{\eta^{m,n}_0},{(\nu^{m,n}_{\cdot})},\omega^m]\rd {(\nu^{m,n}_0)}^x\\
  = &\ \frac{a_{m,i}}{n}
  \sum_{j=1}^nf\lt(\cS^x_{t}[{\eta^{m,n}_0},{(\nu^{m,n}_{\cdot})},\omega^m]\varphi^{m,n}_{(i-1)n+j}\rt)\\
  = &\ \frac{a_{m,i}}{n}
  \sum_{j=1}^nf\lt(\phi^{m,n}_{(i-1)n+j}(t)\rt)=\int_{\T}f\rd {(\nu^{m,n}_t)^x},
\end{align*}
which shows that \eqref{Eq-discrete-fixed-point} holds, since $B$ was arbitrary and $X=\cup_{i=1}^mA^m_i$.

\medskip

\noindent Step II. Construct an auxiliary approximation based on continuous dependence on the graph measures. Since $\nu_0\in \mathcal{C}_{{{\infty}}}$, {by  Theorem~\ref{thm-existence-VE}}, let ${\widehat{\nu}^{m,n}_{\cdot}}$ be the solution to the generalized VE in $\mathcal{C}(\cI,\cC_{{{\infty}}})$
\[{\widehat{\nu}^{m,n}_{\cdot}}=\cF[{\eta^{m,n}_0},{\widehat{\nu}^{m,n}_{\cdot}},\omega]{\widehat{\nu}^{m,n}_{\cdot}}\] with ${\widehat{\nu}^{m,n}_0}=\nu_0$.
By Proposition~\ref{prop-sol-fixedpoint}(iii), we have\begin{equation}\label{Eq-18}
\lim_{m\to\infty}\lim_{n\to\infty}d_{\infty,\BL}(\nu_t,{\widehat{\nu}^{m,n}_{t}})=0,
\end{equation}
since $$\lim_{m\to\infty}\lim_{n\to\infty}d_{\infty,\BL}(\eta_0,{\eta^{m,n}_0})=0.$$

\medskip

\noindent Step III. Construct another auxiliary approximation based on continuous dependence on $\omega$. Let ${\bar{\nu}^{m,n}_{\cdot}}$ be the solution to the fixed point equation
\[{\bar{\nu}^{m,n}_{\cdot}}=\cF[{\eta^{m,n}_0},{\bar{\nu}^{m,n}_{\cdot}},\omega^m]{\bar{\nu}^{m,n}_{\cdot}}\] with $\bar{\nu}_{m,n,0}=\nu_0$.
By Lemma~\ref{le-graph} and Lemma~\ref{le-omega}, \begin{equation}
  \label{uniform-eta}\sup_{m,n\in\N}\|{\eta^{m,n}_0}\|+\sup_{m\in\N}|\omega^m|<\infty
\end{equation} are uniformly bounded,
which implies that \[C=\sup_{m,n\in\N}T{\|\eta^{m,n}_0\|}\textnormal{e}^{{L_5}({\eta^{m,n}_0})T}<\infty.\]

By Proposition~\ref{prop-sol-fixedpoint}(ii), \begin{alignat*}{2}
&d_{\infty,\BL}({\bar{\nu}^{m,n}_{t}},{\widehat{\nu}^{m,n}_{t}})
\le C\|\omega-\omega^m\|_{\infty,\cI},
\end{alignat*}
which implies that
\begin{equation}
  \label{Eq-19}
\lim_{m\to\infty}\lim_{n\to\infty}d_{\infty,\BL}({\bar{\nu}^{m,n}_{t}},{\widehat{\nu}^{m,n}_{t}})=0.
\end{equation}\medskip

\noindent Step IV. Since ${(\nu^{m,n}_{\cdot})}$ is the solution to the generalized VE
\[{(\nu^{m,n}_{\cdot})}=\cF[{\eta^{m,n}_0},{(\nu^{m,n}_{\cdot})},\omega^m]{(\nu^{m,n}_{\cdot})}\] with initial condition $\nu^{m,n}_0$, by Lemma~\ref{le-ini-2} and Proposition~\ref{prop-sol-fixedpoint}(i),
\[
d_{\infty,\BL}({\nu^{m,n}_t},{\bar{\nu}^{m,n}_{t}})
\le \textnormal{e}^{{L_5}({\eta^{m,n}_0})t}
d_{\infty,\BL}({\nu^{m,n}_0},\nu_0).
\]
Similarly, $\sup_{m,n\in\N}{L_5}({\eta^{m,n}_0})<\infty$ by \eqref{uniform-eta}, and thus
\begin{equation}\label{Eq-20}
  \lim_{m\to\infty}\lim_{n\to\infty}d^{\cI}_{\infty,\BL}({\nu^{m,n}_{\cdot}},{\bar{\nu}^{m,n}_{\cdot}})=0.
\end{equation} 

In sum, from \eqref{Eq-18}, \eqref{Eq-19}, and \eqref{Eq-20}, by triangle inequality it yields that
\[\lim_{m\to\infty}\lim_{n\to\infty}d^{\cI}_{\infty,\BL}(\nu_{\cdot},{\nu^{m,n}_{\cdot}})=0.\]
\smallskip
\end{proof}

\section{Discussions and outlook}\label{sect-discussion}
\subsection{Sign of the underlying graphs}
{
For instance, when $h$ is a signed function like many harmonic functions, or even the initial graph is the empty graph where all edge weights are zero, it hardly is possible to separate the positive part and  negative part of the evolving graph in a generic way directly.}

{The sign change in a physical or biological sense  implies the feedback the underlying graph receives from the dynamics on the graph can be \emph{alternatively}  inhibitory and excitatory, which is actually a scenario observed in many neural networks. Hence it could arguably reflects the \emph{nature} of many such real-world models.}

{In what follows, we propose a new framework where  the positive part and the negative part of the evolving  SDGM is trackable, by separating the positive adaptation rule from the negative one} 
{\begin{alignat}{2}\label{coevolution-positive-negative-a}
\dot{\phi}_i =&\ \omega_i+\frac{1}{N}\sum_{j=1}^NW^1_{ij}(t)g_1(\phi_j-\phi_i)-\frac{1}{N}\sum_{j=1}^NW^2_{ij}g_2(\phi_j-\phi_i),\\
\label{coevolution-positive-negative-b}
\dot{W}^1_{ij} =& -\varepsilon_1(W_{ij}-h_1(\phi_j-\phi_i)),
\\
\label{coevolution-positive-negative-c}
\dot{W}^2_{ij} =& -\varepsilon_2(W_{ij}-h_2(\phi_j-\phi_i)),
\end{alignat}
where $h_1$ and $h_2$ are positive functions, and $\varepsilon_1$ and $\varepsilon_2$ are two small positive numbers accounting for different time scales of the evolution of the graphs  relative to the evolution of the oscillators. Here one may take $h_1$ and $h_2$ as the positive and negative part of $h$ in \eqref{coevolution-b}, respectively. It is easily observed in terms of the variation of parameters formula that $W^1$ and $W^2$ will remain positive for all times.} 
{While in contrast to \eqref{coevolution-a}-\eqref{coevolution-b}, such a model might be \emph{less realistic} as the underlying graphs may have no chance to alter the sign of their edge weights, it does become mathematically more tractable. So it would not be surprising if the approach used in this paper may extend to the model \eqref{coevolution-positive-negative-a}-\eqref{coevolution-positive-negative-c}.}

\subsection{A model with non-local adaptivity}

{In this paper, the edge weights of the underlying digraph depend \emph{locally} on the dynamics of the oscillators, i.e., the adaptation of an edge is made only based on its  two vertices. It will be interesting to study where long-range nonlocal adaptations are incorporated, e.g., consider the following system
\begin{align*}
\dot{\phi}_i=& \ \omega_i(t)+\frac{1}{N}\sum_{j=1}^NW_{ij}(t)g(\phi_j-\phi_i),\\
\dot{W}_{ij}=&-\varepsilon\Bigl(W_{ij}+\sum_{k=1}^Na_{jk}h_1(\phi_j-\phi_k)+\sum_{k=1}^Nb_{ik}h_2(\phi_j-\phi_k)\Bigr),
\end{align*}
where $a_{jk}$ and $b_{ik}$ stand for the connectivity between oscillators $j$, $k$ and that between oscillators $i$, $k$. Similarly, the evolution of edge weights may depend on the edge weights also nonlocally:  the evolution of $W_{ij}$ may depend on $W_{ik}$ and $W_{jk}$ for all $k=1,\ldots,N$. Such more general models seem desirable for the sake of their potential applications \cite{BGKKY23}. Nevertheless, new appropriate perspectives and methods are called for to  study the MFL.} 

\subsection{Other co-evolutionary models}
{To get around the difficulty due to the infinite dimensionality caused by the generic mechanism that the dynamics on the network cannot be decoupled from that of the network, we propose the following {co-evolutionary} network:
\begin{equation}
\begin{split}
\dot{\phi}_i(t) =&\ \omega_i(t)+\frac{1}{N}\sum_{j=1}^NW_{ij}u_i(t)v_j(t)g(\phi_j-\phi_i),\\
\dot{u}_{i}(t) =& -\varepsilon_1 (u_i(t)-\frac{1}{N}\sum_{j=1}^Na_{ij}h_1(\phi_j(t)-\phi_i(t))),\\
\dot{v}_{i}(t) =&  -\varepsilon_2 (v_i(t)-\frac{1}{N}\sum_{j=1}^Nb_{ij}h_2(\phi_j(t)-\phi_i(t))),
\end{split}
\end{equation}
where $\phi_i$ is the phase of the $i$-th oscillator, $h_1$, $h_2$ are the adaptivity functions, $W=(W_{ij})_{1\le i,j\le N}$ is a static matrix as a module to generate the time-dependent underlying coupling graph with the adjacency matrix $W_{ij}(t)=\bar{W}_{ij}u_i(t)v_j(t)$ that coevolves with the dynamics of the oscillators. For this network,  one can treat the adaptive network for $\phi$ as a coupled network for the triple $(\phi,u,v)$, and hence the problem reduces to a finite dimensional problem so that the classical approach in \cite{KX21} for studying MFL of static networks applies.} 

{To incorporate certain \emph{nonlinear} self-adaptation ($G'$), we propose the following two models:
\begin{equation}\label{eq:modification-1}
\begin{split}
\dot{\phi}_i=&\ \omega_i(t)+\frac{1}{N}\sum_{j=1}^NW^{(1)}_{ij}(t)g(t,\phi_j,\phi_i),\\
\dot{W}^{(1)}_{ij}=&\ -\varepsilon (G'(W^{(1)}_{ij}))^{-1}h(t,\phi_j,\phi_i);
\end{split}
\end{equation}
\begin{equation}\label{eq:modification-2}
\begin{split}
\dot{\phi}_i=&\ \omega_i(t)+\frac{1}{N}\sum_{j=1}^NW^{(2)}_{ij}(t)g(t,\phi_j,\phi_i)\\
\dot{W}^{(2)}_{ij}=&\ -\varepsilon \bigl((G'(W^{(2)}_{ij}))^{-1}G(W_{ij}^{(2)})+h(t,\phi_i,\phi_j)\bigr),
\end{split}
\end{equation}
where $G$ in either case is assumed to be continuously differentiable with an invertible derivative so that we can represent the edge weights by applying the chain rule and product rule to the second equations of \eqref{eq:modification-1} and  \eqref{eq:modification-2} respectively:
\[W^{(1)}_{ij}(t) = G^{-1}\bigl(-\varepsilon G(W_{ij}^{(1)}(0))\int_0^t h(s,\phi_i(s),\phi_j(s))\rd s\bigr)\]
\[W^{(2)}_{ij}(t) =G^{-1}\bigl(\e^{-\varepsilon s}G(W^{(2)}_{ij}(0))-\varepsilon\int_0^t\e^{-\varepsilon (t-s)}h(s,\phi_i(s),\phi_j(s))\rd s\bigr)\]
Note that in both cases, the way that the \emph{connection} (in term of $W_{ij}(0)$ rather than the coupling in terms of $g$) among nodes on the graph become \emph{nonlinear}. We would like to clarify that these two models are proposed more out of the mathematical curiosity, considering the potential new technical challenges that might occur owing to this nonlinearity. 
For the influence of the nonlinearity of the graph limit on the MFL of the network, we refer the interested reader to, e.g.,  \cite{CDP24}. 
}

\subsection{Technical extensions}

{The key approach of this paper may still be valid with other necessary  technical ingredients added, under the following relaxed assumptions:
\begin{enumerate}
\item[$\bullet$] the initial time $0$ changed to be an arbitrary finite time $t_0$.
\item[$\bullet$] $g(\psi-\phi)$ and $h(\psi-\phi)$ replaced by $g(t,\psi,\phi)$ and $h(t,\psi,\phi)$.
\item[$\bullet$] $\lambda$ changed to be an arbitrary probability measure on $X$.
\item[$\bullet$] the linear vector field of the edge weights replaced by certain nonlinear one (e.g.,  \eqref{eq:modification-1} and \eqref{eq:modification-2}) so that the dynamics of the oscillators can still be decoupled from that of the edge weights.
\item[$\bullet$] the underlying graph changed to be random.
\item[$\bullet$] one type of interaction (in terms of the underlying graph or coupling function) extended to finitely many interactions (so long as the variation of parameters trick still applies).  
\item[$\bullet$] the pairwise interaction changed to higher order interactions (with the techniques from \cite{KX21} adapted to those from \cite{KX22}).
\item[$\bullet$] the deterministic node dynamics (in terms of ODEs) changed  to stochastic dynamics (e.g., in terms of SDEs).
\end{enumerate}
}
\subsection{Challenge regarding absolute continuity} {Generically, even if one is able to obtain absolute continuity of the solution to the generalized Vlasov equation, the PDE for its density is reasonably expected to be one on an infinite dimensional state space (i.e., a functional PDE). Classical conditions for absolute continuity of the MFL based on Rademacher's change of variables formula \cite{EG15} assume the equation of characteristics to generate a Lipschitz flow. It would be interesting to construct simple networks that might serve as examples or counter-examples, in order to gain a better understanding of what is the essential mechanism responsible for absolute continuity of the MFL or Lipschitz flow of the equation of characteristics, and further to explore the physical or biological explanation behind these phenomena.}

\subsection*{Acknowledgments}
The authors appreciate the high-value comments from both referees which help improve the presentation of this manuscript in an indispensable way. MAG and CK gratefully thank the TUM International Graduate School
of Science and Engineering (IGSSE) for support via the project ``Synchronization in Co-Evolutionary
Network Dynamics (SEND)''. CK also acknowledges partial support by a Lichtenberg Professorship
funded by the VolkswagenStiftung. CX acknowledges support by TUM Foundation Fellowship funded by TUM University Foundation, the Alexander von Humboldt Fellowship funded by the Alexander-von-Humboldt Foundation, the Simons Travel Grant by Simons Foundation, and an internal start-up funding from the University of Hawai'i at M\={a}noa.

\section*{Competing Interests Declaration}
The authors have no competing interests.

\bibliographystyle{plain}
\bibliography{references}

\appendix

\section{Gronwall inequalities}\label{appendix-Gronwall}
The following is a second order Gronwall-Bellman inequality.
\begin{proposition}\cite{P73}\label{prop-Gronwall}
Let $T>0$ and $u,\ f,\ g\in\cC(\cI,\R^+)$\footnote{Recall that $\cI=[0,T]$.}. If \[u(t)\le u_0+\int_0^tf(s)u(s)\rd s+\int_0^tf(s)\int_0^sg(\tau)u(\tau)\rd\tau{\rd s},\quad t\in\cI,\]
then
\[u(t)\le u_0\left(1+\int_0^tf(s)\exp\left(\int_0^s(f(\tau)+g(\tau))\rd\tau\right)\rd s\right),\quad t\in\cI.\]
\end{proposition}

\section{Proof of Proposition~\ref{prop-discrete-characteristic}}\label{appendix-prop-discrete-characteristic}
\begin{proof}
The uniqueness and existence of a global solution follows from Picard-Lindel\"{o}f theorem \cite{T12}. Indeed, it is a corollary of Theorem~\ref{thm-characteristic-1} below. Let $X=[0,1]$ and $\{A_i^N\}_{1\le i\le N}$ be the uniform partition of $X$: $$A_i^N=\Bigl[\frac{i-1}{N},\frac{i}{N}\Bigr[,\quad \text{for}\ i=1,\ldots,N-1,\ \text{and}\ A_N^N=\Bigl[1-\frac{1}{N},1\Bigr].$$ 
Let
$$\omega(t,x)=\omega_i,\quad \phi(t,x)=\phi_i(t),\quad  \nu_t^x=\delta_{\phi_i(t)},\quad x\in A_i^N,\quad i=1,\ldots,N,\ t\ge0,$$ and $$\rd\eta_t^x(y)=W_{ij}(t)\rd y,\quad \rd\eta_0^x(y)=W_{ij}(0)\rd y,\quad  (x,y)\in A_i^N\times A_j^N,\quad i,j=1,\ldots,N.$$
Plugging the specific expressions above into the generalized {co-evolutionary} Kuramoto model \eqref{characteristic-Eq-1}-\eqref{characteristic-Eq-3}, we have

\begin{align*}
\dot{\phi}_i(t)=&\ \omega_i+\frac{1}{N}\sum_{j=1}^NW_{ij}(t)g(\phi_j(t)-\phi_i(t)),\\
\dot{W}_{ij}(t)=&-\varepsilon(W_{ij}+h(\phi_j-\phi_i)),
\end{align*}
i.e., \eqref{coevolution-a}-\eqref{coevolution-b}.
\end{proof}
\section{Proof of Proposition~\ref{prop-equivalent}}\label{appendix-prop-equivalent}
\begin{proof}
We show the conclusion in three steps.
\begin{enumerate}
  \item[Step~1.] Every solution to \eqref{characteristic-Eq-1}-\eqref{characteristic-Eq-3} is a solution to \eqref{Eq-variation-1}-\eqref{Eq-variation-2}. By Theorem~\ref{thm-characteristic-1}, let $(\phi,\eta)$ be the unique solution to \eqref{characteristic-Eq-1}-\eqref{characteristic-Eq-3}. Let $$\xi_t=-\varepsilon\int_{0}^t\e^{\varepsilon s}\int_{\T}h(\psi-\phi(s,x))\rd\nu_{s}^y(\psi)\rd y\rd s\in \mathcal{C}(X;\cM(X)),\quad \forall t\in\cI.$$ By Proposition~\ref{prop-derivative}, $\xi$ is differentiable and \begin{align*}
        \frac{\rd\xi_t}{\rd t}=&-\e^{\varepsilon t}\varepsilon\int_{\T}h(\psi-\phi(t,x))\rd\nu_t^y(\psi)\rd y,
      \end{align*} by the chain rule (using the definition of derivative). On the other hand, since $(\phi,\eta)$ is the solution to \eqref{characteristic-Eq-1}-\eqref{characteristic-Eq-3}, we have $\frac{\rd\xi_t}{\rd t}=\e^{\varepsilon t}\frac{\rd\eta_t}{\rd t}+\varepsilon \e^{\varepsilon t}\eta_t=\frac{\rd}{\rd t}(\e^{\varepsilon t}\eta_t)$. This shows
      \[\frac{\rd}{\rd t}(\e^{\varepsilon t}\eta_t)=-\e^{\varepsilon t}\varepsilon\int_{\T}h(\psi-\phi(t,x))\rd\nu_{t}^y(\psi)\rd y;\]
      integrating on both sides in $t$ yields \eqref{Eq-variation-2}. Then substituting \eqref{Eq-variation-2} into \eqref{characteristic-Eq-1} yields \eqref{Eq-variation-1}.
  \item[Step~2.] We use Gronwall inequality to show solutions of  \eqref{Eq-variation-1}-\eqref{Eq-variation-2} are unique. Let $(\varphi,\zeta)$ be another solution to \eqref{Eq-variation-1}-\eqref{Eq-variation-2}. It suffices to show $\varphi=\phi$, and then $\zeta=\eta$ follows from \eqref{Eq-variation-2}. Write \eqref{Eq-variation-1} in its integral form. For $0\le t$,
   \begin{align*}
        &|\phi(t,x)-\varphi(t,x)|\\
        \le& \int_{0}^t\e^{-\varepsilon s}\int_X\int_{\T}|g(\psi-\phi(s,x))-g(\psi-\varphi(s,x))|\rd\nu_s^y(\psi)\rd\eta_0^x(y)\rd s\\
&+\varepsilon\int_{0}^t\int_{0}^s\e^{-\varepsilon(t-\tau)}\int_X\left|\int_{\T}g(\psi-\phi(s,x))
\rd\nu_s^y(\psi)
\int_{\T}h(\psi-\phi(\tau,x))\rd\nu_{\tau}^y(\psi)\right.\\ &\left.-\int_{\T}g(\psi-\varphi(s,x))\rd\nu_s^y(\psi)
\int_{\T}h(\psi-\varphi(\tau,x))\rd\nu_{\tau}^y(\psi)\right|\rd y\rd\tau\rd s\\
\le&\ \Lip(g){\|\nu_{\cdot}\|_{\cI}}\|\eta_0\|\int_{0}^t\e^{-\varepsilon s}|\phi(s,x)-\varphi(s,x)|\rd s\\
&+\varepsilon\int_{0}^t\int_{0}^s\e^{-\varepsilon(t-\tau)}\int_X
\left(\left|\int_{\T}g(\psi-\phi(s,x))\rd\nu_s^y(\psi)
\int_{\T}h(\psi-\phi(\tau,x))\rd\nu_{\tau}^y(\psi)\right.\right.\\ &\left.\left.-\int_{\T}g(\psi-\varphi(s,x))\rd\nu_s^y(\psi)
\int_{\T}h(\psi-\phi(\tau,x))\rd\nu_{\tau}^y(\psi)\right|\right.\\ &\left.+\left|\int_{\T}g(\psi-\varphi(s,x))\rd\nu_s^y(\psi)
\int_{\T}h(\psi-\phi(\tau,x))\rd\nu_{\tau}^y(\psi)\right.\right.\\ &\left.\left.-\int_{\T}g(\psi-\varphi(s,x))\rd\nu_s^y(\psi)
\int_{\T}h(\psi-\varphi(\tau,x))\rd\nu_{\tau}^y(\psi)\right|\right)\rd y\rd\tau\rd s\\
\le&\ \Lip(g){\|\nu_{\cdot}\|_{\cI}}\|\eta_0\|\int_{0}^t\e^{-\varepsilon s}|\phi(s,x)-\varphi(s,x)|\rd s\\
&+\varepsilon\int_{0}^t\int_{0}^s\e^{-\varepsilon(t-\tau)}
\int_X\left(\int_{\T}\left|g(\psi-\phi(s,x))-g(\psi-\varphi(s,x))\right|\rd\nu_s^y(\psi)\right.\\ &\left.\cdot\int_{\T}|h(\psi-\phi(\tau,x))|\rd\nu_{\tau}^y(\psi)+\int_{\T}
|g(\psi-\varphi(s,x))|\rd\nu_s^y(\psi)\right.\\ &\left.
\cdot\int_{\T}\left|h(\psi-\phi(\tau,x))
-h(\psi-\varphi(\tau,x))\right|\rd\nu_{\tau}^y(\psi)\right)\rd y\rd\tau\rd s\\
\le&\ \Lip(g){\|\nu_{\cdot}\|_{\cI}}\|\eta_0\|\int_{0}^t\e^{-\varepsilon s}|\phi(s,x)-\varphi(s,x)|\rd s\\
&+\varepsilon({\|\nu_{\cdot}\|_{\cI}})^2\int_{0}^t\int_{0}^s\e^{-\varepsilon(t-\tau)}(\Lip(g)
\|h\|_{\infty}|\phi(s,x)-\varphi(s,x)|\\ &+\|g\|_{\infty}\Lip(h)|\phi(\tau,x)
-\varphi(\tau,x)|)\rd\tau\rd s\\
\le&\ \Lip(g){\|\nu_{\cdot}\|_{\cI}}\left(\|\eta_0\|+\|h\|_{\infty}{\|\nu_{\cdot}\|_{\cI}}\right)
\int_{0}^t|\phi(s,x)-\varphi(s,x)|\rd s\\ &+\varepsilon\|g\|_{\infty}\Lip(h)({\|\nu_{\cdot}\|_{\cI}})^2\int_{0}^t\int_{0}^s|\phi(\tau,x)
-\varphi(\tau,x)|\rd\tau\rd s.
      \end{align*}
      By Proposition~\ref{prop-Gronwall},
$$|\phi(t,x)-\varphi(t,x)|\le0\cdot\frac{C_1\e^{(C_1+C_2)t}+C_2}{C_1+C_2},$$
where \begin{equation}\label{constants}
    C_1=\Lip(g){\|\nu_{\cdot}\|_{\cI}}\left(\|h\|_{\infty}{\|\nu_{\cdot}\|_{\cI}}+\|\eta_0\|\right)\quad \text{and}\ C_2=\frac{\varepsilon\|g\|_{\infty}\Lip(h)({\|\nu_{\cdot}\|_{\cI}})^2}{C_1}.
\end{equation}This shows the uniqueness of solutions of \eqref{Eq-variation-1}-\eqref{Eq-variation-2}.
  \item[Step~3.] By Steps~1 and 2, and in the light of the uniqueness of solutions to  \eqref{characteristic-Eq-1}-\eqref{characteristic-Eq-3}, we show the solution to \eqref{characteristic-Eq-1}-\eqref{characteristic-Eq-3} and that to \eqref{Eq-variation-1}-\eqref{Eq-variation-2} coincide.
\end{enumerate}
\end{proof}

\section{Proof of Proposition~\ref{prop-semiflow}}\label{appendix-prop-semiflow}
\begin{proof}
 (i) is obvious since solutions to \eqref{characteristic-Eq-1}-\eqref{characteristic-Eq-3} are in ${\cC(\cI,\cC(X,\mathbb{T}\times\cM(X)))}$, by Theorem~\ref{thm-characteristic-1}.

From Corollary~\ref{cor-flow},
\begin{equation*}\label{Flow}
\begin{split}
\mathcal{S}_{t}^{x}[\eta_0,\nu_{\cdot},\omega](\phi_0)=&\ \Bigl(\phi_0(x)+\int_0^t\left(\omega(s,x)+\e^{-\varepsilon s}\int_X\int_{\T}g(\psi-\mathcal{S}_{s}^{x}[\eta_0,\nu_{\cdot},\omega](\phi_0))\rd\nu_s^y(\psi)
\rd\eta_0^x(y)\right.\\
&-\varepsilon\int_{0}^s\e^{-\varepsilon(t-\tau)}\int_X\left(\int_{\T}g(\psi-
\mathcal{S}_{s}^{x}[\eta_0,\nu_{\cdot},\omega](\phi_0))\rd\nu_s^y(\psi)\right.\\
&\left.\left.\int_{\T}h(\psi-\mathcal{S}_{\tau}^{x}[\eta_0,\nu_{\cdot},\omega](\phi_0))
\rd\nu_{\tau}^y(\psi)\right)\rd y\rd\tau\right)\rd s\Bigr)\mod1
\end{split}
\end{equation*}
\begin{enumerate}
\item[(ii)] Lipschitz continuity in $t$.
Let $t_1<t_2$.
\begin{align*}
&|\mathcal{S}_{t_1}^{x}[\eta_0,\nu_{\cdot},\omega](\phi_0)
-\mathcal{S}_{t_2}^{x}[\eta_0,\nu_{\cdot},\omega](\phi_0)|\\
\le&\left|\int_{t_1}^{t_2}\left(\omega(s,x)+\e^{-\varepsilon s}\int_X\int_{\T}g(\psi-\mathcal{S}_{s}^{x}[\eta_0,\nu_{\cdot},\omega](\phi_0))
\rd\nu_s^y(\psi)\rd\eta_0^x(y)\right)\rd s\right|\\
&+\varepsilon\int_{0}^{t_2}\int_0^s\left|\e^{-\varepsilon(t_2-\tau)}
-\e^{-\varepsilon(t_1-\tau)}\right|\int_X\left(\int_{\T}g(\psi-
\mathcal{S}_{s}^{x}[\eta_0,\nu_{\cdot},\omega](\phi_0))\rd\nu_s^y(\psi)\right.\\
&\ \left.\left.\int_{\T}h(\psi-\mathcal{S}_{\tau}^{x}[\eta_0,\nu_{\cdot},\omega](\phi_0))
\rd\nu_{\tau}^y(\psi)\right)\rd y\rd\tau\right)\rd s\\
&+\varepsilon\int_{t_1}^{t_2}\int_0^s\e^{-\varepsilon(t_1-\tau)}\int_X\left(\int_{\T}|g(\psi-
\mathcal{S}_{s}^{x}[\eta_0,\nu_{\cdot},\omega](\phi_0))|\rd\nu_s^y(\psi)\right.\\
&\ \left.\left.\int_{\T}|h(\psi-\mathcal{S}_{\tau}^{x}[\eta_0,\nu_{\cdot},\omega](\phi_0))|
\rd\nu_{\tau}^y(\psi)\right)\rd y\rd\tau\right)\rd s\\
\le&\ |t_2-t_1|\left(\max_{t_1\le s\le t_2}{|\omega(s,x)|}+\|g\|_{\infty}{\|\nu_{\cdot}\|_{\cI}}\|\eta_0\|\e^{-\varepsilon t_1}\right)\\
&+\varepsilon\|g\|_{\infty}\|h\|_{\infty}({\|\nu_{\cdot}\|_{\cI}})^2\left(\int_0^{t_2}\int_0^s|\e^{-\varepsilon(t_1-\tau)}
-\e^{-\varepsilon(t_2-\tau)}|\rd\tau\rd s
+\int_{t_1}^{t_2}\int_0^s\e^{-\varepsilon(t_1-\tau)}\rd\tau\rd s\right)\\
\le&\ |t_2-t_1|\left(\max_{t_1\le s\le t_2}{|\omega(s,x)|}+\|g\|_{\infty}{\|\nu_{\cdot}\|_{\cI}}\|\eta_0\|\e^{-\varepsilon t_1}+(\tfrac{1}{2}t_2^2\varepsilon^2+1)\|g\|_{\infty}\|h\|_{\infty}({\|\nu_{\cdot}\|_{\cI}})^2\right)\\
\le&\ L_1|t_2-t_1|,
\end{align*}
where we rest on the fact that $\e^{-x}+x$ is increasing in $x\in\R^+$ and $$L_1=L_1(\nu_{\cdot},\eta_0,\omega)=\max\limits_{s\in\cI}\|\omega(s,\cdot)\|_{\infty}+\|g\|_{\infty}{\|\nu_{\cdot}\|_{\cI}}\|\eta_0\|
+{(\tfrac{1}{2}T^2+\varepsilon^2)}\|g\|_{\infty}\|h\|_{\infty}({\|\nu_{\cdot}\|_{\cI}})^2$$
  \item[(iii)] Lipschitz continuous dependence on the initial conditions. \begin{align*}
    &|\mathcal{S}_{t}^{x}[\eta_0,\nu_{\cdot},\omega](\phi_0)
    -\mathcal{S}_{t}^{x}[\eta_0,\nu_{\cdot},\omega](\varphi_0)|\le|\phi_0(x)-\varphi_0(x)|
    \\
    &+\int_0^t\e^{-\varepsilon s}\int_X\int_{\T}|g(\psi-\mathcal{S}_{s}^{x}[\eta_0,\nu_{\cdot},\omega](\phi_0))
    -g(\psi-\mathcal{S}_{s}^{x}[\eta_0,\nu_{\cdot},\omega](\varphi_0))|
    \rd\nu_s^y(\psi)\rd\eta_0^x(y)\rd s\\
&+\varepsilon\int_0^t\int_{0}^s\e^{-\varepsilon(t-\tau)}\left.\int_X\left|\int_{\T}g(\psi-
\mathcal{S}_{s}^{x}[\eta_0,\nu_{\cdot},\omega](\phi_0))\rd\nu_s^y(\psi)
\int_{\T}h(\psi-\mathcal{S}_{\tau}^{x}[\eta_0,\nu_{\cdot},\omega](\phi_0))
\rd\nu_{\tau}^y(\psi)\right.\right.\\
&\left.-\int_{\T}g(\psi-
\mathcal{S}_{s}^{x}[\eta_0,\nu_{\cdot},\omega](\varphi_0))\rd\nu_s^y(\psi)
\int_{\T}h(\psi-\mathcal{S}_{\tau}^{x}[\eta_0,\nu_{\cdot},\omega](\varphi_0))
\rd\nu_{\tau}^y(\psi)\right|\rd y\rd\tau\rd s\\
\le&|\phi_0(x)-\varphi_0(x)|+\Lip(g){\|\nu_{\cdot}\|_{\cI}}\|\eta_0\|\int_0^t\e^{-\varepsilon s}|\mathcal{S}_{s}^{x}[\eta_0,\nu_{\cdot},\omega](\phi_0)-\mathcal{S}_{s}^{x}
[\eta_0,\nu_{\cdot},\omega](\varphi_0))|\rd s\\
    &+\varepsilon\int_0^t\int_{0}^s\e^{-\varepsilon(t-\tau)}\int_X\left(\int_{\T}|g(\psi-
\mathcal{S}_{s}^{x}[\eta_0,\nu_{\cdot},\omega](\phi_0))-g(\psi-
\mathcal{S}_{s}^{x}[\eta_0,\nu_{\cdot},\omega](\varphi_0))|\rd\nu_s^y(\psi)\right.\\
&\cdot\int_{\T}|h(\psi-\mathcal{S}_{\tau}^{x}[\eta_0,\nu_{\cdot},\omega](\varphi_0))|
\rd\nu_{\tau}^y(\psi)+\int_{\T}|g(\psi-
\mathcal{S}_{s}^{x}[\eta_0,\nu_{\cdot},\omega](\phi_0))|\rd\nu_s^y(\psi)\\
&\left.\cdot\int_{\T}|h(\psi-\mathcal{S}_{\tau}^{x}[\eta_0,\nu_{\cdot},\omega](\phi_0))
-h(\psi-\mathcal{S}_{\tau}^{x}[\eta_0,\nu_{\cdot},\omega](\varphi_0))|\rd\nu_{\tau}^y(\psi)
\right)\rd y\rd\tau\rd s\\
\le&\ |\phi_0(x)-\varphi_0(x)|+\Lip(g){\|\nu_{\cdot}\|_{\cI}}\|\eta_0\|\int_0^t\e^{-\varepsilon s}|\mathcal{S}_{s}^{x}[\eta_0,\nu_{\cdot},\omega](\phi_0)
-\mathcal{S}_{s}^{x}[\eta_0,\nu_{\cdot},\omega](\varphi_0))|\rd s\\
    &+\Lip(g)\|h\|_{\infty}({\|\nu_{\cdot}\|_{\cI}})^2\varepsilon\int_0^t\int_{0}^s\e^{-\varepsilon(t-\tau)}
    |\mathcal{S}_{s}^{x}[\eta_0,\nu_{\cdot},\omega](\phi_0)
    -\mathcal{S}_{s}^{x}[\eta_0,\nu_{\cdot},\omega](\varphi_0)|\rd\tau\rd s\\
    &+\|g\|_{\infty}\Lip(h)({\|\nu_{\cdot}\|_{\cI}})^2\varepsilon\int_0^t\int_{0}^s\e^{-\varepsilon(t-\tau)}
    |\mathcal{S}_{\tau}^{x}[\eta_0,\nu_{\cdot},\omega](\phi_0)-\mathcal{S}_{\tau}^{x}
    [\eta_0,\nu_{\cdot},\omega](\varphi_0)|\rd\tau\rd s\\
    \le&\ |\phi_0(x)-\varphi_0(x)|+\|g\|_{\infty}\Lip(h)({\|\nu_{\cdot}\|_{\cI}})^2\varepsilon\int_0^t\int_{0}^s
    |\mathcal{S}_{\tau}^{x}[\eta_0,\nu_{\cdot},\omega](\phi_0)
    -\mathcal{S}_{\tau}^{x}[\eta_0,\nu_{\cdot},\omega](\varphi_0)|\rd\tau\rd s\\
    &+\Lip(g){\|\nu_{\cdot}\|_{\cI}}(\|h\|_{\infty}\|\nu_{\cdot}\|+\|\eta_0\|)
    \int_0^t|\mathcal{S}_{s}^{x}[\eta_0,\nu_{\cdot},\omega](\phi_0)
    -\mathcal{S}_{s}^{x}[\eta_0,\nu_{\cdot},\omega](\varphi_0))|\rd s\\
  \le&\ |\phi_0(x)-\varphi_0(x)|\cdot\frac{C_1\e^{(C_1+C_2)t}+C_2}{C_1+C_2}\le \e^{L_2t}\|\phi_0-\varphi_0\|_{\infty},
  \end{align*}
  where {in the last inequality we applied Proposition~\ref{prop-Gronwall}, the Gronwall-Bellman inequality, with $C_1$ and $C_2$ given in \eqref{constants} and $L_2=C_1+C_2$.}
   \item[(iv)] Lipschitz continuous dependence in $\omega$. The proof is similar to that of (iii). 
  {For the reader's convenience, we provide a complete proof with detailed estimates.
   \begin{align*}
    &|\mathcal{S}_{t}^{x}[\eta_0,\nu_{\cdot},\omega](\phi_0)-\mathcal{S}_{t}^{x}
    [\eta_0,\nu_{\cdot},\widetilde{\omega}](\phi_0)|\\
    \le&\int_0^t|\omega(s,x)-\widetilde{\omega}(s,x)|\rd s\\
    &+\int_0^t\e^{-\varepsilon s}\int_X\int_{\T}
    |g(\psi-\mathcal{S}_{s}^{x}[\eta_0,\nu_{\cdot},\omega](\phi_0))
    -g(\psi-\mathcal{S}_{s}^{x}[\eta_0,\nu_{\cdot},\widetilde{\omega}](\phi_0))|
    \rd\nu_s^y(\psi)\rd\eta_0^x(y)\rd s\\
&\left.+\varepsilon\int_0^t\int_0^s\e^{-\varepsilon(t-\tau)}\int_X\left|\int_{\T}g(\psi-
\mathcal{S}_{s}^{x}[\eta_0,\nu_{\cdot},\omega](\phi_0))\rd\nu_s^y(\psi)
\int_{\T}h(\psi-\mathcal{S}_{\tau}^{x}[\eta_0,\nu_{\cdot},\omega](\phi_0))
\rd\nu_{\tau}^y(\psi)\right.\right.\\
&\left.-\int_{\T}g(\psi-
\mathcal{S}_{s}^{x}[\eta_0,\nu_{\cdot},\widetilde{\omega}](\phi_0))\rd\nu_s^y(\psi)
\int_{\T}h(\psi-\mathcal{S}_{\tau}^{x}[\eta_0,\nu_{\cdot},\widetilde{\omega}](\phi_0))
\rd\nu_{\tau}^y(\psi)\right|\rd y\rd\tau\rd s\\
\le&\ t\max\limits_{0\le s\le t}|\omega(s,x)-\widetilde{\omega}(s,x)|+\Lip(g){\|\nu_{\cdot}\|_{\cI}}\|\eta_0\|\int_0^t\e^{-\varepsilon s}|\mathcal{S}_{s}^{x}[\eta_0,\nu_{\cdot},\omega](\phi_0)-\mathcal{S}_{s}^{x}
[\eta_0,\nu_{\cdot},\widetilde{\omega}](\phi_0))|\rd s\\
    &+\varepsilon\int_0^t\int_{0}^s\e^{-\varepsilon(t-\tau)}\int_X\left(\int_{\T}|g(\psi-
\mathcal{S}_{s}^{x}[\eta_0,\nu_{\cdot},\omega](\phi_0))-g(\psi-
\mathcal{S}_{s}^{x}[\eta_0,\nu_{\cdot},\widetilde{\omega}](\phi_0))|\rd\nu_s^y(\psi)\right.\\
&\cdot\int_{\T}|h(\psi-\mathcal{S}_{\tau}^{x}[\eta_0,\nu_{\cdot},\widetilde{\omega}](\phi_0))|
\rd\nu_{\tau}^y(\psi)|+\int_{\T}|g(\psi-
\mathcal{S}_{s}^{x}[\eta_0,\nu_{\cdot},\omega](\phi_0))|\rd\nu_s^y(\psi)\\
&\left.\cdot\int_{\T}|h(\psi-\mathcal{S}_{\tau}^{x}[\eta_0,\nu_{\cdot},\omega](\phi_0))
-h(\psi-\mathcal{S}_{\tau}^{x}[\eta_0,\nu_{\cdot},\widetilde{\omega}](\phi_0))|\rd\nu_{\tau}^y(\psi)\right)\rd y\rd\tau\rd s\\
\le&\ t\max\limits_{0\le s\le t}|\omega(s,x)-\widetilde{\omega}(s,x)|+\Lip(g){\|\nu_{\cdot}\|_{\cI}}\|\eta_0\|\int_0^t\e^{-\varepsilon s}|\mathcal{S}_{s}^{x}[\eta_0,\nu_{\cdot},\omega](\phi_0)
-\mathcal{S}_{s}^{x}[\eta_0,\nu_{\cdot},\widetilde{\omega}](\phi_0))|\rd s\\
    &+\Lip(g)\|h\|_{\infty}({\|\nu_{\cdot}\|_{\cI}})^2\varepsilon\int_0^t\int_{0}^s\e^{-\varepsilon(t-\tau)}
    |\mathcal{S}_{s}^{x}[\eta_0,\nu_{\cdot},\omega](\phi_0)
    -\mathcal{S}_{s}^{x}[\eta_0,\nu_{\cdot},\widetilde{\omega}](\phi_0)|\rd\tau\rd s\\
     &+\|g\|_{\infty}\Lip(h)({\|\nu_{\cdot}\|_{\cI}})^2\varepsilon\int_0^t\int_{0}^s\e^{-\varepsilon(t-\tau)}
    |\mathcal{S}_{\tau}^{x}[\eta_0,\nu_{\cdot},\omega](\phi_0)-\mathcal{S}_{\tau}^{x}
    [\eta_0,\nu_{\cdot},\widetilde{\omega}](\phi_0)|\rd\tau\rd s\\
    \le&\ t\max\limits_{0\le s\le t}|\omega(s,x)-\widetilde{\omega}(s,x)|\\
    &+\|g\|_{\infty}\Lip(h)({\|\nu_{\cdot}\|_{\cI}})^2\varepsilon\int_0^t\int_{0}^s
    |\mathcal{S}_{\tau}^{x}[\eta_0,\nu_{\cdot},\omega](\phi_0)
    -\mathcal{S}_{\tau}^{x}[\eta_0,\nu_{\cdot},\widetilde{\omega}](\phi_0)|\rd\tau\rd s\\
    &+\Lip(g){\|\nu_{\cdot}\|_{\cI}}(\|h\|_{\infty}{\|\nu_{\cdot}\|_{\cI}}+\|\eta_0\|)
    \int_0^t|\mathcal{S}_{s}^{x}[\eta_0,\nu_{\cdot},\omega](\phi_0)
    -\mathcal{S}_{s}^{x}[\eta_0,\nu_{\cdot},\widetilde{\omega}](\phi_0))|\rd s\\
  \le&\max\limits_{0\le s\le t}|\omega(s,x)-\widetilde{\omega}(s,x)|t\frac{C_1\e^{(C_1+C_2)t}+C_2}{C_1+C_2}\le T\e^{L_2t}\max\limits_{s\in\cI}\|\omega-\widetilde{\omega}\|_{\infty},
\end{align*}
   where again in the  inequality before last we applied Proposition~\ref{prop-Gronwall}.}
   \item[(v)] Continuous dependence on the initial {SDGM} $\eta_0$. Let $$g_k(t,x,y)=\int_{\sT}
    g(\psi-\mathcal{S}_{t}^{x}[\eta_k,\nu_{\cdot},\omega](\phi_0))
    \rd\nu_t^y(\psi),\quad t\in\cI,\ x,y\in X.$$ Then
   \begin{align*}
    &|\mathcal{S}_{t}^{x}[\eta_0,\nu_{\cdot},\omega](\phi_0)
    -\mathcal{S}_{t}^{x}[\eta_k,\nu_{\cdot},\omega](\phi_0)|\\
    \le&\int_0^t\e^{-\varepsilon s}\left|\int_X\int_{\T}
    g(\psi-\mathcal{S}_{s}^{x}[\eta_k,\nu_{\cdot},\omega](\phi_0))
    \rd\nu_s^y(\psi)\rd(\eta_0^x(y)-\eta_k^x(y))\right|\rd s\\
    &+\int_0^t\e^{-\varepsilon s}\int_X\int_{\T}
    |g(\psi-\mathcal{S}_{s}^{x}[\eta_0,\nu_{\cdot},\omega](\phi_0))
    -g(\psi-\mathcal{S}_{s}^{x}[\eta_k,\nu_{\cdot},\omega](\phi_0))|
    \rd\nu_s^y(\psi)\rd\eta_0^x(y)\rd s\\
&\left.+\varepsilon\int_0^t\int_{0}^s\e^{-\varepsilon(t-\tau)}\int_X\left|\int_{\T}g(\psi-
\mathcal{S}_{s}^{x}[\eta_0,\nu_{\cdot},\omega](\phi_0))\rd\nu_s^y(\psi)
\int_{\T}h(\psi-\mathcal{S}_{\tau}^{x}[\eta_0,\nu_{\cdot},\omega](\phi_0))
\rd\nu_{\tau}^y(\psi)\right.\right.\\
&\left.-\int_{\T}g(\psi-
\mathcal{S}_{s}^{x}[\eta_k,\nu_{\cdot},\omega](\phi_0))\rd\nu_s^y(\psi)
\int_{\T}h(\psi-\mathcal{S}_{\tau}^{x}[\eta_k,\nu_{\cdot},\omega](\phi_0))
\rd\nu_{\tau}^y(\psi)\right|\rd y\rd\tau\rd s\\
\le&\int_0^t\e^{-\varepsilon s}\left|\int_Xg_k(s,x,y)\rd(\eta_0^x(y)-\eta_k^x(y))\right|\rd s\\
&+\Lip(g){\|\nu_{\cdot}\|_{\cI}}\|\eta_0\|\int_0^t\e^{-\varepsilon s}|\mathcal{S}_{s}^{x}[\eta_0,\nu_{\cdot},\omega](\phi_0)
-\mathcal{S}_{s}^{x}[\eta_k,\nu_{\cdot},\omega](\phi_0))|\rd s\\
    &+\varepsilon\int_0^t\int_{0}^s\e^{-\varepsilon(t-\tau)}
    \int_X\left(\int_{\T}|g(\psi-
\mathcal{S}_{s}^{x}[\eta_0,\nu_{\cdot},\omega](\phi_0))-g(\psi-
\mathcal{S}_{s}^{x}[\eta_k,\nu_{\cdot},\omega](\phi_0))|\rd\nu_s^y(\psi)\right.\\
&\left.\cdot\int_{\T}|h(\psi-\mathcal{S}_{\tau}^{x}[\eta_k,\nu_{\cdot},\omega](\phi_0))|
\rd\nu_{\tau}^y(\psi)|+\int_{\T}|g(\psi-\mathcal{S}_{s}^{x}[\eta_0,\nu_{\cdot},\omega](\phi_0))|
    \rd\nu_s^y(\psi)\right.\\
&\left.\cdot\int_{\T}|h(\psi-\mathcal{S}_{\tau}^{x}[\eta_0,\nu_{\cdot},\omega](\phi_0))
-h(\psi-\mathcal{S}_{\tau}^{x}[\eta_k,\nu_{\cdot},\omega](\phi_0))|\rd\nu_{\tau}^y(\psi)\right)\rd y\rd\tau\rd s\\
 \le&\int_0^t\e^{-\varepsilon s}\left|\int_X{g_k(s,x,y)}\rd(\eta_0^x(y)-\eta_k^x(y))\right|\rd s\\
 &+\Lip(g){\|\nu_{\cdot}\|_{\cI}}\|\eta_0\|\int_0^t\e^{-\varepsilon s}|\mathcal{S}_{s}^{x}[\eta_0,\nu_{\cdot},\omega](\phi_0)
 -\mathcal{S}_{s}^{x}[\eta_k,\nu_{\cdot},\omega](\phi_0))|\rd s\\
        &+\Lip(g)\|h\|_{\infty}(\|\nu\|^*)^2\varepsilon\int_0^t\int_{0}^s\e^{-\varepsilon(t-\tau)}
    |\mathcal{S}_{s}^{x}[\eta_0,\nu_{\cdot},\omega](\phi_0)
    -\mathcal{S}_{s}^{x}[\eta_k,\nu_{\cdot},\omega](\phi_0)|\rd\tau\rd s\\
    &+\|g\|_{\infty}\Lip(h)({\|\nu_{\cdot}\|_{\cI}})^2\varepsilon\int_0^t\int_{0}^s\e^{-\varepsilon(t-\tau)}
    |\mathcal{S}_{\tau}^{x}[\eta_0,\nu_{\cdot},\omega](\phi_0)
    -\mathcal{S}_{\tau}^{x}[\eta_k,\nu_{\cdot},\omega](\phi_0)|\rd\tau\rd s\\
    \le&\int_0^T\e^{-\varepsilon s}\left|\int_Xg_k(s,x,y)\rd(\eta_0^x(y)-\eta_k^x(y))\right|\rd s\\
    &+\Lip(g){\|\nu_{\cdot}\|_{\cI}}(\|h\|_{\infty}{\|\nu_{\cdot}\|_{\cI}}+\|\eta_0\|)
    \int_0^t|\mathcal{S}_{s}^{x}[\eta_0,\nu_{\cdot},\omega](\phi_0)
    -\mathcal{S}_{s}^{x}[\eta_k,\nu_{\cdot},\omega](\phi_0))|\rd s\\
    &+\|g\|_{\infty}\Lip(h)({\|\nu_{\cdot}\|_{\cI}})^2\varepsilon\int_0^t\int_{0}^s
    |\mathcal{S}_{\tau}^{x}[\eta_0,\nu_{\cdot},\omega](\phi_0)-\mathcal{S}_{\tau}^{x}
    [\eta_k,\nu_{\cdot},\omega](\phi_0)|\rd\tau\rd s\\
  \le&\frac{C_1\e^{(C_1+C_2)t}+C_2}{C_1+C_2}\int_0^T\e^{-\varepsilon s}\left|\int_Xg_k(s,x,y)\rd(\eta_0^x(y)-\eta_k^x(y))\right|\rd s\\
  \le&\ \e^{L_2t}\int_0^T\e^{-\varepsilon s}\left|\int_Xg_k(s,x,y)\rd(\eta_0^x(y)-\eta_k^x(y))\right|\rd s,
  \end{align*} where the second last inequality follows from Proposition~\ref{prop-Gronwall} again. Since $\nu_{\cdot}\in \mathcal{C}(\cI,\cC_{\infty})$ {by $\mathbf{(A6)'}$}, by Proposition~\ref{prop-nu}, $\mathbf{(A2)}$, and (ii), we have {$g_k(\cdot,x,\cdot)\in\mathcal{C}_{\sf b}(\cI\times X)$ for every $x\in X$}. Hence $\lim_{k\to\infty}{d_{\infty,\BL}}(\eta_0,\eta_k)=0$ implies \[\lim_{k\to\infty}\sup_{x\in X}\left|\int_X{g_k(s,x,y)}\rd(\eta_0^x(y)-\eta_k^x(y))\right|=0,\quad \text{for all}\ s\in\cI,\] by Proposition~\ref{prop-nu}(i). 
  Moreover, \[\sup_{s\in\cI}\sup_{x\in X}\left|\int_X{g_k(s,x,y)}\rd(\eta_0^x(y)-\eta_k^x(y))\right|\le\|g\|_{\infty}{\|\nu_{\cdot}\|_{\cI}}
    (\|\eta_0\|+\|\eta_k\|),\] by Dominated Convergence Theorem, \begin{equation*}\label{eq:strong-convergence-graph}
    \lim_{k\to\infty}\int_0^T\e^{-\varepsilon s}\sup_{x\in X}\left|\int_X{g_k(s,x,y)}\rd(\eta_0^x(y)-\eta_k^x(y))\right|\rd s=0,
    \end{equation*} and the conclusion follows from Fatou's lemma.
      \item[(vi)] Lipschitz continuous dependence on $\nu_{\cdot}$. Then 
  \begin{align*}
    &|\mathcal{S}_{t}^{x}[\eta_0,\nu_{\cdot},\omega](\phi_0)
    -\mathcal{S}_{t}^{x}[\eta_0,\upsilon_{\cdot},\omega](\phi_0)|\\
    \le&\int_0^t\e^{-\varepsilon s}\left|\int_X\int_{\T}
    g(\psi-\mathcal{S}_{s}^{x}[\eta_0,\upsilon_{\cdot},\omega](\phi_0))
    \rd(\nu_s^y(\psi)-\upsilon_s^y(\psi))\rd\eta_0^x(y)\right|\rd s\\
    &+\int_0^t\e^{-\varepsilon s}\int_X\int_{\T}
    |g(\psi-\mathcal{S}_{s}^{x}[\eta_0,\nu_{\cdot},\omega](\phi_0))
    -g(\psi-\mathcal{S}_{s}^{x}[\eta_0,\upsilon_{\cdot},\omega](\phi_0))|
    \rd\nu_s^y(\psi)\rd\eta_0^x(y)\rd s\\
&\left.+\varepsilon\int_0^t\int_{0}^s\e^{-\varepsilon(t-\tau)}\int_X\left|\int_{\T}g(\psi-
\mathcal{S}_{s}^{x}[\eta_0,\nu_{\cdot},\omega](\phi_0))\rd\nu_s^y(\psi)
\int_{\T}h(\psi-\mathcal{S}_{\tau}^{x}[\eta_0,\nu_{\cdot},\omega](\phi_0))
\rd\nu_{\tau}^y(\psi)\right.\right.\\
&\left.-\int_{\T}g(\psi-
\mathcal{S}_{s}^{x}[\eta_0,\upsilon_{\cdot},\omega](\phi_0))\rd\upsilon_s^y(\psi)
\int_{\T}h(\psi-\mathcal{S}_{\tau}^{x}[\eta_0,\upsilon_{\cdot},\omega](\phi_0))
\rd\upsilon_{\tau}^y(\psi)\right|\rd y\rd\tau\rd s\\
\le&\ \|\eta_0\|\BL(g)\int_0^t\e^{-\varepsilon s}d_{\infty,\BL}(\nu_s,\upsilon_s)\rd s\\
&+\Lip(g){\|\nu_{\cdot}\|_{\cI}}\|\eta_0\|\int_0^t\e^{-\varepsilon s}|\mathcal{S}_{s}^{x}[\eta_0,\nu_{\cdot},\omega](\phi_0)
-\mathcal{S}_{s}^{x}[\eta_0,\upsilon_{\cdot},\omega](\phi_0))|\rd s\\
    &+\varepsilon\int_0^t\int_{0}^s\e^{-\varepsilon(t-\tau)}
    \int_X\left[\int_{\T}|g(\psi-\mathcal{S}_{s}^{x}[\eta_0,\nu_{\cdot},\omega](\phi_0))|
    \rd\nu_s^y(\psi)\right.\\
&\cdot\left(\int_{\T}|h(\psi-\mathcal{S}_{\tau}^{x}[\eta_0,\nu_{\cdot},\omega](\phi_0))
-h(\psi-\mathcal{S}_{\tau}^{x}[\eta_0,\upsilon_{\cdot},\omega](\phi_0))|\rd\nu_{\tau}^y(\psi)\right.\\
&\left.+\left|\int_{\T}h(\psi-\mathcal{S}_{\tau}^{x}[\eta_0,\upsilon_{\cdot},\omega](\phi_0))
\rd(\nu_{\tau}^y(\psi)-\upsilon_{\tau}^y(\psi))\right|\right)\\
&+\left(\int_{\T}|g(\psi-
\mathcal{S}_{s}^{x}[\eta_0,\nu_{\cdot},\omega](\phi_0))\rd\nu_s^y(\psi)-g(\psi-
\mathcal{S}_{s}^{x}[\eta_0,\upsilon_{\cdot},\omega](\phi_0))|\rd\nu_s^y(\psi)\right.\\
&\left.+\left|\int_{\T}g(\psi-\mathcal{S}_{\tau}^{x}[\eta_0,\upsilon_{\cdot},\omega](\phi_0))
\rd(\nu_{s}^y(\psi)-\upsilon_{s}^y(\psi))\right|
\right)\\
&\left.\cdot\int_{\T}|h(\psi-\mathcal{S}_{\tau}^{x}[\eta_0,\upsilon_{\cdot},\omega](\phi_0))|
\rd\upsilon_{\tau}^y(\psi)|\right]\rd y\rd\tau\rd s\\
\le&\ \|\eta_0\|\BL(g)\int_0^t\e^{-\varepsilon s}d_{\infty,\BL}(\nu_{s},
    \upsilon_{s})\rd s\\
&+\Lip(g){\|\nu_{\cdot}\|_{\cI}}\|\eta_0\|\int_0^t\e^{-\varepsilon s}
|\mathcal{S}_{s}^{x}[\eta_0,\nu_{\cdot},\omega](\phi_0)
-\mathcal{S}_{s}^{x}[\eta_0,\upsilon_{\cdot},\omega](\phi_0))|\rd s\\
    &+\|g\|_{\infty}\Lip(h)({\|\nu_{\cdot}\|_{\cI}})^2\varepsilon\int_0^t\int_{0}^s\e^{-\varepsilon(t-\tau)}
    |\mathcal{S}_{\tau}^{x}[\eta_0,\nu_{\cdot},\omega](\phi_0)
    -\mathcal{S}_{\tau}^{x}[\eta_0,\upsilon_{\cdot},\omega](\phi_0)|\rd\tau\rd s\\
    &+\|g\|_{\infty}\BL(h){\|\nu_{\cdot}\|_{\cI}}\varepsilon\int_0^t\int_{0}^s
    \e^{-\varepsilon(t-\tau)}d_{\infty,\BL}(\nu_{\tau},
    \upsilon_{\tau})\rd\tau\rd s\\
    &+\Lip(g)\|h\|_{\infty}{\|\nu_{\cdot}\|_{\cI}}\|\upsilon_{\cdot}\|_{\cI}\varepsilon\int_0^t\int_{0}^s
    \e^{-\varepsilon(t-\tau)}|\mathcal{S}_{s}^{x}[\eta_0,\nu_{\cdot},\omega](\phi_0)
    -\mathcal{S}_{s}^{x}[\eta_0,\upsilon_{\cdot},\omega](\phi_0)|\rd\tau\rd s\\
 &+\BL(g)\|h\|_{\infty}\|\upsilon_{\cdot}\|_{\cI}\varepsilon\int_0^t\int_{0}^s\e^{-\varepsilon(t-\tau)}
 d_{\infty,\BL}(\nu_{s},
    \upsilon_{s})\rd\tau\rd s\\
\le&\ \BL(g)(\|h\|_{\infty}\|\upsilon_{\cdot}\|_{\cI}+\|\eta_0\|)\int_0^t d_{\infty,\BL}(\nu_{s},
    \upsilon_{s})\rd s\\
    &+\|g\|_{\infty}\BL(h){\|\nu_{\cdot}\|_{\cI}}\varepsilon\int_0^t\int_{0}^s d_{\infty,\BL}(\nu_{\tau},
    \upsilon_{\tau})\rd\tau\rd s\\
&+\Lip(g){\|\nu_{\cdot}\|_{\cI}}(\|h\|_{\infty}\|\upsilon_{\cdot}\|_{\cI}+\|\eta_0\|)\int_0^t
|\mathcal{S}_{s}^{x}[\eta_0,\nu_{\cdot},\omega](\phi_0)
-\mathcal{S}_{s}^{x}[\eta_0,\upsilon_{\cdot},\omega](\phi_0))|\rd s
\\
 &+\|g\|_{\infty}\Lip(h)({\|\nu_{\cdot}\|_{\cI}})^2\varepsilon\int_0^t\int_{0}^s
    |\mathcal{S}_{\tau}^{x}[\eta_0,\nu_{\cdot},\omega](\phi_0)
    -\mathcal{S}_{\tau}^{x}[\eta_0,\upsilon_{\cdot},\omega](\phi_0)|\rd\tau\rd s\\
\le&\ \Bigl(\BL(g)(\|h\|_{\infty}\|\upsilon_{\cdot}\|_{\cI}+\|\eta_0\|)+\|g\|_{\infty}
\BL(h){\|\nu_{\cdot}\|_{\cI}}\varepsilon T\Bigr)\int_0^t  d_{\infty,\BL}(\nu_s,\upsilon_s)\rd s\\
    &+{\|\nu_{\cdot}\|_{\cI}}\Bigl(\Lip(g)(\|h\|_{\infty}\|\upsilon_{\cdot}\|_{\cI}+\|\eta_0\|)
    +\|g\|_{\infty}\Lip(h){\|\nu_{\cdot}\|_{\cI}}\varepsilon T\Bigr)\\
    &\cdot\int_0^t
|\mathcal{S}_{s}^{x}[\eta_0,\nu_{\cdot},\omega](\phi_0)
-\mathcal{S}_{s}^{x}[\eta_0,\upsilon_{\cdot},\omega](\phi_0))|\rd s\\
    \le&\ L_3\e^{L_4t}\int_0^t d_{\infty,\BL}(\nu_{s},\upsilon_{s})\rd s,
  \end{align*}
  where again the last inequality is a consequence of the Gronwall inequality with $L_3=L_3(\nu_{\cdot},\upsilon_{\cdot})=\BL(g)(\|h\|_{\infty}\|\upsilon_{\cdot}\|_{\cI}+\|\eta_0\|)
  +\|g\|_{\infty}\BL(h){\|\nu_{\cdot}\|_{\cI}}\varepsilon T$ and $L_4=L_4(\nu_{\cdot},\upsilon_{\cdot})={\|\nu_{\cdot}\|_{\cI}}\Bigl(\Lip(g)(\|h\|_{\infty}\|\upsilon_{\cdot}\|_{\cI}
  +\|\eta_0\|)+\|g\|_{\infty}\Lip(h){\|\nu_{\cdot}\|_{\cI}}\varepsilon T\Bigr)$.
  \end{enumerate}
\end{proof}

\section{Proof of Proposition~\ref{prop-continuousdependence}}\label{appendix-prop-continuousdependence}
\begin{proof}
We will suppress some of the variables of $\cS_{t}[\eta_0,\nu_{\cdot},\omega]$ in the bracket whenever it is clear and deemphasized from the context. Recall the integro-differential equation for $\mathcal{S}_{t}^{x}[\eta_0,\nu_{\cdot},\omega](\phi)$:
\begin{equation*}\label{Recall-IDE}
\begin{split}
\mathcal{S}_{t}^{x}[\eta_0,\nu_{\cdot},\omega](\phi_0)
=&\ \Bigl(\phi(x)+\int_0^t\left(\omega(x)+\e^{-\varepsilon s}\int_X\int_Yg(\psi-\mathcal{S}_{s}^{x}[\eta_0,\nu_{\cdot},\omega](\phi_0))
\rd\nu_s^y(\psi)\rd\eta_0^x(y)\right.\\
&-\varepsilon\int_{0}^s\e^{-\varepsilon(t-\tau)}\int_X\left(\int_Yg(\psi-
\mathcal{S}_{s}^{x}[\eta_0,\nu_{\cdot},\omega](\phi_0))\rd\nu_s^y(\psi)\right.\\
&\left.\left.\int_Yh(\psi-\mathcal{S}_{\tau}^{x}[\eta_0,\nu_{\cdot},\omega](\phi_0))
\rd\nu_{\tau}^y(\psi)\right)\rd y\rd\tau\right)\rd s\Bigr)\mod1
\end{split}
\end{equation*}

\begin{enumerate}
 \item[(i)]
{
Note that for every $\nu_{\cdot}\in\cC(\cI,\cB_{\infty})$, we have $\|\nu_0\|<\infty$, and $\bT\subseteq(\mathcal{S}^x_{t}[\eta_0,\nu_{\cdot},\omega])^{-1}\bT$, 
 which implies the mass conservation law:
 \begin{equation*}\label{eq-conservation}
 (\cF[\eta_0,h]\nu_{t})^x(\bT)=\nu_0^x((\mathcal{S}^x_{t}[\eta_0,\nu_{\cdot},\omega])^{-1}\bT)=\nu_0^x(\bT),\quad t\in\cI,\quad x\in X.
 \end{equation*}
 This further implies that 
 $\cF[\eta_0,\omega]\nu_t\in\cB_{\infty}$.
}

Assume $\nu_{\cdot}\in\cC(\cI,\cB_{\infty})$. Next, we will show the Lipschitz continuity of $\cF[\eta_0,\nu_{\cdot},\omega]\nu_{t}$ in $t$. 
  Indeed, from Proposition~\ref{prop-semiflow}(ii),
\begin{align*}
  &d_{\infty,\BL}(\cF[\eta_0,\nu_{\cdot},\omega]\nu_{t},\cF[\eta_0,\nu_{\cdot},\omega]\nu_{t'})\\
  =&\sup_{x\in X}d_{\sf BL}(\nu_0^x\circ (\cS^x_{t}[\eta_0,\nu_{\cdot},\omega])^{-1},\nu_0^x\circ (\cS^x_{t'}[\eta_0,\nu_{\cdot},\omega])^{-1})\\
  =&\sup_{x\in X}\sup_{f\in\mathcal{BL}_1(\bT)}\Bigl|\int_{\bT}f(\phi)\rd(\nu_0^x\circ (\cS^x_{t}[\eta_0,\nu_{\cdot},\omega])^{-1}(\phi)-\nu_0^x\circ (\cS^x_{t'}[\eta_0,\nu_{\cdot},\omega])^{-1}(\phi))\Bigr|\\
  =&\sup_{x\in X}\sup_{f\in\mathcal{BL}_1(\bT)}\Bigl|\int_{\psi\in\cup_{\phi\in\sT}
  \cS^x_{t}[\eta_0,\nu_{\cdot},\omega]^{-1}(\phi)}f\circ(\cS^x_{t}[\eta_0,\nu_{\cdot},\omega])\psi
  \rd\nu_0^x(\psi)\\&\hspace{2cm}
  -\int_{\psi\in\cup_{\phi\in\sT}\cS^x_{t'}[\eta_0,\nu_{\cdot},\omega]^{-1}(\phi)}
  f\circ(\cS^x_{t'}[\eta_0,\nu_{\cdot},\omega])\psi\rd\nu_0^x(\psi)\Bigr|\\
  =&\sup_{x\in X}\sup_{f\in\mathcal{BL}_1(\bT)}\lt|\int_{\bT}\lt(f\circ (\cS^x_{t}[\eta_0,\nu_{\cdot},\omega])\psi-f\circ (\cS^x_{t'}[\eta_0,\nu_{\cdot},\omega])\psi\rt)\rd\nu_0^x(\psi)\rt|\\
  \le&\sup_{x\in X}\int_{\mathbb{T}}\lt|\cS^x_{t}[\eta_0,\nu_{\cdot},\omega]\psi
  -\mathcal{S}^x_{t'}[\eta_0,\nu_{\cdot},\omega]\psi\rt|
  \rd\nu_0^x(\psi)\\
  \le&\ L_1(\nu_{\cdot}){\|\nu_0\|}|t-t'|\to0,\end{align*}
as $|t-t'|\to0$. This shows that $t\mapsto\cF[\eta_0,\omega]\nu_{t}\in \mathcal{C}(\cI,{\mathcal{B}_{\infty}})$ is Lipschitz continuous. 
{Now additionally assume $\nu_0\in\cC_{\infty}$, 
which immediately implies
\begin{equation}
  \label{continuity_x}
  \lim_{{|x-y|}\to0}d_{\BL}(\nu_0^x,\nu_0^y)=0.
\end{equation}
We will show $\cF[\eta_0,\omega]\nu_{t}\in\cC_{\infty}$ for all $t$, which immediately yields $\cF[\eta_0,\omega]\nu_{\cdot}\in\cC(\cI,\cC_{{\infty}})$.} 

{Recall from Proposition~\ref{prop-semiflow}(iii) that \begin{equation*}\label{Lip-phi}|\cS_{t}^x[\eta_0,\nu_{\cdot},\omega]\phi_1(x)
    -\cS_{t}^x[\eta_0,\nu_{\cdot},\omega]\phi_2(x)|\le \e^{L_2(\nu_{\cdot})t}\|\phi_1-\phi_2\|_{\infty},\end{equation*} where the finite constant $L_2(\nu_{\cdot})$ is given in Proposition~\ref{prop-semiflow}(iii).  Hence, for every $f\in\mathcal{BL}_1(\mathbb{T})$,
\[\Lip(f\circ \cS^x_{t}[\eta_0,\nu_{\cdot},\omega])\le \Lip(f)\Lip(\cS^x_{t}[\eta_0,\nu_{\cdot},\omega])
\le\Lip(f)\e^{L_2(\nu_{\cdot})t},\ \|f\circ \cS^x_{t}[\eta_0,\nu_{\cdot},\omega]\|_{\infty}\le\|f\|_{\infty}.\]
This shows $\BL(f\circ \cS^x_{t}[\eta_0,\nu_{\cdot},\omega])\le \e^{L_2(\nu_{\cdot})t}$.}

{Similarly,  \begin{align*}
  &d_{\BL}((\cF[\eta_0,\omega]\nu_{t})^x,(\cF[\eta_0,\omega]\nu_{t})^y)\\
  =&\ d_{\sf BL}(\nu_0^x\circ (\cS^x_{t}[\eta_0,\nu_{\cdot},\omega])^{-1},\nu_0^y\circ (\cS^y_{t}[\eta_0,\nu_{\cdot},\omega])^{-1})\\
  \le&\ d_{\sf BL}(\nu_0^x\circ (\cS^x_{t}[\eta_0,\nu_{\cdot},\omega])^{-1},\nu_0^y\circ (\cS^x_{t}[\eta_0,\nu_{\cdot},\omega])^{-1})\\
  &+d_{\sf BL}(\nu_0^y\circ (\cS^x_{t}[\eta_0,\nu_{\cdot},\omega])^{-1},\nu_0^y\circ (\cS^y_{t}[\eta_0,\nu_{\cdot},\omega])^{-1})\\
  \le&\sup_{f\in\mathcal{BL}_1(\bT)}\int_{\bT} f\circ (\cS^x_{t}[\eta_0,\nu_{\cdot},\omega])\psi\rd(\nu_0^x(\psi)-\nu_0^y(\psi))\\
  &+\sup_{f\in\mathcal{BL}_1(\bT)}\lt|\int_{\bT}\lt(f\circ (\cS^x_{t}[\eta_0,\nu_{\cdot},\omega])\psi-f\circ (\cS^y_{t}[\eta_0,\nu_{\cdot},\omega])\psi\rt)\rd\nu_0^y(\psi)\rt|\\
  \le&\ \e^{L_2(\nu_{\cdot}) t}d_{\BL}(\nu_0^x,\nu_0^y)+\|\nu_0\|\sup_{\psi\in\sT}
  |\cS^x_{t}[\eta_0,\nu_{\cdot},\omega](\psi)-\cS^y_{t}[\eta_0,\nu_{\cdot},\omega](\psi)|.\end{align*}
Hence it follows from \eqref{continuity_x}, $\nu_0\in\cC_{\infty}$, and Proposition~\ref{prop-semiflow}(i) that
\[\lim_{{|x-y|}\to0}d_{\infty,\BL}((\cF[\eta_0,\omega]\nu_{t})^x,(\cF[\eta_0,\omega]\nu_{t})^y)=0.\]
This shows $\cF[\eta_0,\omega]\nu_{t}\in\cC_{\infty}$.}
\item[(ii)] Lipschitz continuity in $\omega$. First recall from Proposition~\ref{prop-semiflow}(iv) that \begin{equation*}\label{Lip-omega}|\cS_{t}^x[\omega]\phi(x)-\cS_{t}^x[\widetilde{\omega}]\phi(x)|\le T\e^{L_2t}\|\omega-\widetilde{\omega}\|_{\infty,\cI}.\end{equation*}
 Now we show $\cS^x_{t}[\omega]$ is Lipschitz continuous in $\omega$. Note that
\begin{alignat}{2}
\nonumber &d_{\sf BL}(\nu_0^x\circ (\cS^x_{t}[\omega])^{-1},\nu_0^x\circ (\cS^x_{t}[\widetilde{\omega}])^{-1})\\
\nonumber =&\sup_{f\in\mathcal{BL}_1(\mathbb{T})}\int_{\mathbb{T}} f(\phi)\mathrm{d}((\nu_0^x\circ (\cS^x_{t}[\omega])^{-1})(\psi)-(\nu_0^x\circ (\cS^x_{t}[\widetilde{\omega}])^{-1})(\psi))\\
\nonumber =&\sup_{f\in\mathcal{BL}_1(\mathbb{T})}\int_{\mathbb{T}} ((f\circ \cS^x_{t}[\omega])(\psi)-(f\circ \cS^x_{t}[\widetilde{\omega}])(\psi))\mathrm{d}\nu_0^x(\psi)\\
\nonumber \le&\int_{\mathbb{T}}\lt|\cS^x_{t}[\omega](\psi)- \cS^x_{t}[\widetilde{\omega}](\psi)\rt|\mathrm{d}\nu_0^x(\psi)\\
\label{term-1}\le&\ T\e^{L_2(\nu_{\cdot})t}{\|\nu_{\cdot}\|_{\cI}}\max\limits_{s\in\cI}\|\omega
-\widetilde{\omega}\|_{\infty,\cI}
\end{alignat}
{where the last inequality follows from Proposition~\ref{prop-semiflow}(iv) with $L_2(\nu_{\cdot})$ given in Proposition~\ref{prop-semiflow}(iii).}  
This implies \begin{equation*}
        d_{\infty,\BL}(\mathcal{F}[\eta_0,\omega]\nu_{t},
        \mathcal{F}[\eta_0,\widetilde{\omega}]\nu_{t})\le {T\|\nu_{\cdot}\|_{\cI}\e^{L_2(\nu_{\cdot})t}}\|\omega-\widetilde{\omega}\|_{\infty,\cI}.
\end{equation*}
\item[(iii)] Lipschitz continuity in $\eta_0$. 
{Note that}
\begin{alignat}{2}
\nonumber &d_{\sf BL}(\nu_0^x\circ (\cS^x_{t}[\eta_0])^{-1},\nu_0^x\circ (\cS^x_{t}[\eta_k])^{-1})\\
\nonumber =&\sup_{f\in\mathcal{BL}_1(\mathbb{T})}\int_{\mathbb{T}} f(\psi)\mathrm{d}((\nu_0^x\circ (\cS^x_{t}[\eta_0])^{-1})(\psi)-(\nu_0^x\circ (\cS^x_{t}[\eta_k])^{-1})(\psi))\\
\nonumber =&\sup_{f\in\mathcal{BL}_1(\mathbb{T})}\int_{\mathbb{T}} ((f\circ \cS^x_{t}[\eta_0])(\psi)-(f\circ \cS^x_{t}[\eta_k])(\psi))\mathrm{d}\nu_0^x(\psi)\\
\nonumber \le&\int_{\mathbb{T}}\lt|\cS^x_{t}[\eta_0](\psi)- \cS^x_{t}[\eta_k](\psi)\rt|\mathrm{d}\nu_0^x(\psi).
\end{alignat}
Define $\widehat{\nu}_0\in\cM_+(\sT)$:
\[\widehat{\nu}_0(B)\coloneqq\sup_{x\in X}\nu_0^x(B),\quad \forall B\in\mathfrak{B}(\sT).\]Obviously, {$\|\widehat{\nu}_0\|_{\TV}\le\|\nu_0\|$.} 
Then
\begin{align*}
  d_{\infty,\BL}(\cF[\eta_0,\omega]\nu_t,\cF[\eta_k,\omega]\nu_t)
  \le&\sup_{x\in X}\int_{\mathbb{T}}\lt|\cS^x_{t}[\eta_0](\psi)- \cS^x_{t}[\eta_k](\psi)\rt|\mathrm{d}\nu_0^x(\psi)\\
  \le&\int_{\mathbb{T}}\sup_{x\in X}\lt|\cS^x_{t}[\eta_0](\psi)- \cS^x_{t}[\eta_k](\psi)\rt|\mathrm{d}\widehat{\nu}_0(\psi).
\end{align*}
Note that \[\lt|\cS^x_{t}[\eta_0](\psi)- \cS^x_{t}[\eta_k](\psi)\rt|\le1,\quad \forall x\in X,\quad \psi\in\sT.\]By Dominated Convergence Theorem, it follows from {$\mathbf{(A6)'}$ and} Proposition~\ref{prop-semiflow}(v) that
\begin{align*}
&d_{\infty,\BL}(\cF[\eta_0,\omega]\nu_t,\cF[\eta_k,\omega]\nu_t)\to0,\quad \text{as}\quad k\to\infty.
\end{align*}

\item[(iv)] Lipschitz continuity in $\nu_{\cdot}$.
Next, we show $\cS^x_{t}[\nu_{\cdot}]$ is Lipschitz continuous in $\nu_{\cdot}$. Observe that
\begin{align}
\nonumber&d_{\sf BL}(\nu_0^x\circ (\cS_{t}^x[\nu_{\cdot}])^{-1},\upsilon_0^{x}\circ (\cS^x_{t}[\upsilon_{\cdot}])^{-1})\\
\label{sum}\le&\ d_{\sf BL}(\nu_0^{x}\circ (\cS_{t}^x[\nu_{\cdot}])^{-1},\nu_0^{x}\circ (\cS^x_{t}[\upsilon_{\cdot}])^{-1})+d_{\sf BL}(\nu_0^{x}\circ (\cS^x_{t}[\upsilon_{\cdot}])^{-1},\upsilon_0^{x}\circ (\cS^x_{t}[\upsilon_{\cdot}])^{-1}).
\end{align}
We estimate the two terms separately. {Note that}
\begin{alignat}{2}
\nonumber &d_{\sf BL}(\nu_0^{x}\circ (\cS^x_{t}[\nu_{\cdot}])^{-1},\nu_0^{x}\circ (\cS^x_{t}[\upsilon_{\cdot}])^{-1})\\
\nonumber =&\sup_{f\in\mathcal{BL}_1(\mathbb{T})}\int_{\mathbb{T}} f(\phi)\mathrm{d}((\nu_0^{x}\circ (\cS^x_{t}[\nu_{\cdot}])^{-1})(\phi)-(\nu_0^{x}\circ (\cS^x_{t}[\upsilon_{\cdot}])^{-1})(\phi))\\
\nonumber =&\sup_{f\in\mathcal{BL}_1(\mathbb{T})}\int_{\mathbb{T}} ((f\circ \cS^x_{t}[\nu_{\cdot}])(\phi)-(f\circ \cS^x_{t}[\upsilon_{\cdot}])(\phi))\mathrm{d}\nu_0^{x}(\phi)\\
\nonumber \le&\int_{\mathbb{T}}\lt|\cS^x_{t}[\nu_{\cdot}](\phi)- \cS^x_{t}[\upsilon_{\cdot}](\phi)\rt|\mathrm{d}\nu_0^{x}(\phi)\\
\label{term-2}\le&\ L_3(\nu_{\cdot},\upsilon_{\cdot}){\|\nu_{\cdot}\|_{\cI}}\e^{{L_4(\nu_{\cdot},\upsilon_{\cdot})} t}\int_0^t {d_{\infty,\BL}}(\nu_s,\upsilon_s)\rd s,
\end{alignat}
{where the last inequality follows from Proposition~\ref{prop-semiflow}(vi) with 
$L_3,\ L_4$ given in Proposition~\ref{prop-semiflow}(vi).} 
 From the proof of (i), for every $f\in\mathcal{BL}_1(\sT)$, we have {$\BL(f\circ \cS^x_{t}[\upsilon_{\cdot}])\le \e^{L_2(\upsilon_{\cdot})t}$.} For every $x\in X$,
\begin{alignat}{2}
 \nonumber &d_{\sf BL}(\nu_0^{x}\circ (\cS^x_{t}[\upsilon_{\cdot}])^{-1},\upsilon_0^{x}\circ (\cS^x_{t}[\upsilon_{\cdot}])^{-1})\\
  \nonumber=&\sup_{f\in\mathcal{BL}_1(\mathbb{T})}\int_{\mathbb{T}}(f\circ \cS^x_{t}[\upsilon_{\cdot}])(\phi)\rd(\nu_0^{x}(\phi)-\upsilon_0^{x}(\phi))\\
 \label{term-3} \le&\ e^{{L_2(\upsilon_{\cdot})}t}d_{\sf BL}(\nu_0^{x},\upsilon_0^{x})\le \e^{{L_2(\upsilon_{\cdot})}t}d_{\infty,\BL}(\nu_0,\upsilon_0).
\end{alignat}
Combining \eqref{term-2} and \eqref{term-3}, it follows from \eqref{sum} that
\begin{align}
\nonumber&d_{\infty,\BL}(\cF[\eta_0,\omega]\nu_{t},\cF[\eta_0,\omega]\upsilon_{t})\\
\nonumber =& \sup_{x\in X}d_{\sf BL}(\nu_0^{x}\circ{(\cS^x_{t}[\nu_{\cdot}])^{-1}},\upsilon_0^{x}\circ {(\cS^x_{t}[\upsilon_{\cdot}])^{-1}})\\
 \nonumber\le&\ \e^{{L_2(\upsilon_{\cdot})}t}d_{\infty,\BL}(\nu_0,\upsilon_0)+ {L_3(\nu_{\cdot},\upsilon_{\cdot}){\|\nu_{\cdot}\|_{\cI}}\e^{L_4(\nu_{\cdot},\upsilon_{\cdot}) t}}\int_0^t d_{\infty,\BL}(\nu_s,\upsilon_s) \rd s.
\end{align}
\end{enumerate}
\end{proof}

\section{Proof of Proposition~\ref{prop-sol-fixedpoint}}\label{appendix-prop-sol-fixedpoint}
\begin{proof}
\begin{enumerate}
  \item[(i)] 
      The proof is analogous to that of \cite[{Proposition~3.5}]{KX21}, by applying Gronwall inequality to the inequality in Proposition~\ref{prop-continuousdependence}(iv).
  \item[(ii)] Continuous dependence of solutions of \eqref{Fixed} on $\omega$.

Let $\nu^i_{\cdot}\in \mathcal{C}(\mathcal{I},\cB_{\infty})$ for $i=1,2$ be the solutions to the generalized VE \eqref{Fixed} with $\omega$ replaced by $\omega_i$, with the same initial condition $\nu^1_0=\nu^2_0$. We denote $\nu_0^{i,x}$ for $(\nu^i_0)^x$.
    \begin{align*}
      &d_{\sf BL}(\nu_0^{1,x}\circ {(\cS^x_{t}[\nu^1_{\cdot},\omega_1])^{-1}},
\nu_0^{1,x}\circ {(\cS^x_{t}[\nu^2_{\cdot},\omega_2])^{-1}})\\
      \le&\ d_{\sf BL}({\nu_0^{1,x}\circ(\cS^x_{t}[\nu^1_{\cdot},\omega_1])^{-1}},
{\nu_0^{1,x}\circ(\cS^x_{t}[\nu^1_{\cdot},\omega_2])^{-1}})\\
&+d_{\sf BL}({\nu_0^{1,x}\circ (\cS^x_{t}[\nu^1_{\cdot},\omega_2])^{-1}}, {\nu_0^{1,x}\circ(\cS^x_{t}[\nu^2_{\cdot},\omega_2])^{-1}}).
    \end{align*}
  It follows from Proposition~\ref{prop-continuousdependence}(ii) that
  \begin{align*}
  &d_{\sf BL}(\nu_0^{1,x}\circ {(\cS^x_{t}[\nu^1_{\cdot},\omega_1])^{-1}},\nu_0^{1,x}\circ {(\cS^x_{t}[\nu^1_{\cdot},\omega_2])^{-1}})\le {T\|\nu^1_{\cdot}\|_{\cI}\textnormal{e}^{L_2(\nu_{\cdot}^1)t}\|\omega_1-\omega_2\|_{\cI,\infty}}.\end{align*}

It suffices to estimate $d_{\sf BL}({\nu_0^{1,x}\circ (\cS^x_{t}[\nu^1_{\cdot},\omega_2])^{-1}},{\nu_0^{1,x}\circ (\cS^x_{t}[\nu^2_{\cdot},\omega_2])^{-1}})$, which follows from \eqref{term-1} that:
\[
  {d_{\sf BL}(\nu_0^{1,x}\circ {(\cS^x_{t}[\nu^1_{\cdot},\omega_2])^{-1}},\nu_0^{1,x}\circ  {(\cS^x_{t}[\nu^2_{\cdot},\omega_2])^{-1}})\le L_3(\nu^1_{\cdot},\nu^2_{\cdot})\|\nu^{1}\|_{\cI}\e^{L_4(\nu^1_{\cdot},\nu^2_{\cdot})t}
  \int_0^t d_{\infty,\BL} (\nu^1_{\tau},\nu^2_{\tau})\rd\tau.}\]
Hence
\begin{align*}
  d_{\infty,\BL}(\nu^1_{t},\nu^2_{t})=&\sup_{x\in X}d_{\sf BL}({\nu^{1,x}_t,\nu^{2,x}_t})\\
  \le&\ {T\|\nu^1_{\cdot}\|_{\cI}\textnormal{e}^{L_2(\nu_{\cdot}^1)t}\|\omega_1-\omega_2\|_{\cI,\infty}
  +L_3(\nu^1_{\cdot},\nu^2_{\cdot})\|\nu^{1}\|_{\cI}\e^{L_4(\nu^1_{\cdot},\nu^2_{\cdot})t}
  \int_0^t d_{\infty,\BL} (\nu^1_{\tau},\nu^2_{\tau})\rd\tau.}
\end{align*}
By Gronwall's inequality,
\[d_{\infty,\BL}(\nu^1_t,\nu^2_t)\le
 {T\|\nu^1_{\cdot}\|_{\cI}\textnormal{e}^{L_5t}\|\omega_1-\omega_2\|_{\infty,\cI}.}\]
  \item[(iii)] Continuous dependence on $\eta_0$. Assume 
{$\{\eta_k\}_{k\in\N_0}\subseteq\cB(X,\cM(X))$}  satisfy $$\lim_{k\to\infty}d_{\infty,\BL}(\eta_0,\eta_k)=0.$$ {Since $\lim_{k\to\infty}d_{\infty}(\eta_0,\eta_k)=0$, it follows from the triangle inequality that $a=\sup\limits_{k\in\N_0}\|\eta_k\|<\infty$.} Let $\nu_{\cdot}$ be the solution to the generalized VE \eqref{Fixed} with initial condition {$\nu_0\in\cC_{\infty}$.} Let ${\nu^k_{\cdot}}\in \mathcal{C}(\mathcal{I},{\mathcal{B}_{\infty}})$ be the solutions to the generalized VE \eqref{Fixed} with $\eta_0$ replaced by $\eta_k$ with the same initial conditions $\nu_0=\nu_{k,0}$. {Note that $\|\nu_{\cdot}\|_{\cI},\|{\nu^k_{\cdot}}\|_{\cI}\le\|\nu_0\|$ by Proposition~\ref{prop-continuousdependence}(i), the mass conservation law.}
   It then suffices to show
    \[\lim_{k\to\infty}d_{\infty,\BL}(\cF[\eta_0]\nu_t,\cF[\eta_0]{\nu^k_{t}})=0,\quad t\in\cI.\]
By triangle inequality,    \begin{alignat}{2}
\nonumber   d_{\sf BL}(\nu_t^x,\nu_{k,t}^x)=&\ d_{\sf BL}(\nu_0^x\circ {(\cS^x_{t}[\eta_0,\nu_{\cdot}])^{-1}},\nu_0^x\circ {(\cS^x_{t}[\eta_k,{\nu^k_{\cdot}}])^{-1}})\\
\label{Eq-10}      \le&\ d_{\sf BL}(\nu_0^x\circ {(\cS^x_{t}[\eta_k,\nu_{\cdot}])^{-1}},\nu_0^x\circ {(\cS^x_{t}[\eta_k,{\nu^k_{\cdot}}])^{-1}})\\
\nonumber      &+d_{\sf BL}(\nu_0^x\circ {(\cS_{t}^x[\eta_0,\nu_{\cdot}])^{-1}},\nu_0^x\circ {(\cS^x_{t}[\eta_k,\nu_{\cdot}])^{-1}}),\quad x\in X.
    \end{alignat}
    From \eqref{term-2} it follows that
   \begin{align}
   \nonumber&d_{\sf BL}(\nu_0^x\circ {(\cS^x_{t}[\eta_k,\nu_{\cdot}])^{-1}},\nu_0^x\circ {(\cS^x_{t}[\eta_k,{\nu^k_{\cdot}}])^{-1}})\\
   \nonumber\le&\int_{\sT}|\cS_{t}^x[\eta_k,\nu_{\cdot}]\psi
   -\cS^x_{t}[\eta_k,{\nu^k_{\cdot}}]\psi|\rd{\nu_0^x}(\psi)\eqqcolon{\beta_k(t,x)},\\
   \label{Eq-11}\le&\ {L_{3,k}(\nu_{\cdot},\nu^{k}_{\cdot})\|\nu_{\cdot}\|_{\cI}}\e^{{L_{4,k}(\nu_{\cdot},{\nu^k_{\cdot}})} t}
   \int_0^td_{\infty,\BL}(\nu_{\tau},{\nu^k_{\tau}})\rd\tau,
   \end{align}
    where {$L_{3,k}$ and $L_{4,k}$ defined in Proposition~\ref{prop-semiflow} that depend on $\|\eta_k\|$ linearly.}
 {For  $i=3,4$, let $\bar{L}_{i}$ be $L_{i}$ defined in Proposition~\ref{prop-semiflow} replacing $\|\eta_0\|$ by $a$, and $\|\nu_{\cdot}\|_{\cI}$ and $\|\upsilon_{\cdot}\|_{\cI}$ both by $\|\nu_0\|$.}

  From \eqref{Eq-11} it follows that
  \begin{equation}
    \label{Eq-12}
  d_{\sf BL}(\nu_0^x\circ {(\cS^x_{t}[\eta_k,\nu_{\cdot}])^{-1}},\nu_0^x\circ {(\cS^x_{t}[\eta_k,{\nu^k_{\cdot}}])^{-1}})\le {\beta_k(t,x)}\le {\bar{L}_3\e^{\bar{L}_4t}}
   \int_0^td_{\infty,\BL}(\nu_{\tau},{\nu^k_{\tau}})\rd\tau.\end{equation}
   We now estimate the second term in \eqref{Eq-10}.
\begin{align}
\nonumber  &d_{\sf BL}(\nu_0^x\circ (\cS^x_{t}[\eta_0,\nu_{\cdot}])^{-1},\nu_0^x\circ (\cS^x_{t}[\eta_k,\nu_{\cdot}])^{-1})\\
\nonumber    =&\sup_{f\in\mathcal{BL}_1(\bT)}\int_{\bT}\lt((f\circ \cS^x_{t}[\eta_0,\nu_{\cdot}])(\psi)-(f\circ \cS^x_{t}[\eta_k,\nu_{\cdot}])(\psi)\rt)\rd\nu_0^x(\psi)\\
\label{Eq-13}  \le&\int_{\bT}|\cS^x_{t}[\eta_0,\nu_{\cdot}]\psi-\cS^x_{t}[\eta_k,\nu_{\cdot}]\psi)|\rd\nu_0^x(\psi)
  \eqqcolon{\gamma_k(t,x)}
  \end{align}
Let $C_k={\|\gamma_k\|_{\cI,\infty}}$. It follows from \eqref{Eq-11}, \eqref{Eq-12} and \eqref{Eq-13} that, for $t\in\cI$,
\[
d_{\sf BL}(\nu_t^x,{(\nu^k)^x_t})\le C_k+ {\bar{L}_3} \e^{{\bar{L}_4} T}
   \int_0^td_{\infty,\BL}(\nu_{\tau},{\nu^k_{\tau}})\rd\tau,
\] which further implies by Gronwall inequality that
\[d_{\infty,\BL}(\nu_{t},{\nu^k_{t}})\le C_k\textnormal{e}^{{\bar{L}_3} e^{{\bar{L}_4} T}T},\quad t\in\cI.\]Note that $\lim_{k\to\infty}C_k=0$ by Proposition~\ref{prop-semiflow}(v) and the Dominated Convergence Theorem. This shows that
$$\lim_{k\to\infty}{d_{,\cI,\infty,\BL}(\nu_{\cdot},{\nu^k_{\cdot}})}=0.$$
\end{enumerate}
\end{proof}

\section{Notation}\label{appendix-Notation}

\setlength{\tabcolsep}{9pt}
\renewcommand{\arraystretch}{1.3}
\begin{longtable}{  p{2.2cm} p{10.3cm} }
\caption{Notation. 
}\label{table-summary}\\  \hline
$\R$ ($\R^+$)& the set of real numbers (nonnegative real numbers)\\\hline
$\mathbb{T}$ & $\coloneqq[0,1[$ via the natural projection $x\mapsto \e^{{\rm i}2\pi x}$.
\\\hline
$\cN$ & a compact interval.\\\hline
$\cI$ & an interval, mainly taken as $\coloneqq[0,T]$ for a finite $T>0$.\\\hline
$X$ & a compact subset of $\R^r$ for some $r\in\N$\\\hline
$(Y,d_Y)$ & a complete metric space\\\hline
$d_{\mathbb{T}}(x,y)$ & $\coloneqq\min\{|x-y|,1-|x-y|\}$, $x,\ y\in\mathbb{T}$.\\\hline
$\Diam A$ & $\coloneqq\sup_{x,y\in A}{|x-y|}$, the diameter of a set $A\subseteq X$\\\hline
$\mathfrak{B}(Y)$ & the Borel sigma algebra of $Y$\\\hline
$\cM(Y)$ & the set of all finite Borel signed measures on $Y$\\\hline
$\cM_+(Y)$ & the set of all finite Borel positive measures on $Y$\\\hline
$\cP(Y)$ & the space of Borel probabilities on $Y$\\\hline
$\mathcal{B}(X,Y)$ & the space of bounded measurable functions from $X$ to $Y$\\\hline
$\cC(X,Y)$ & the space of continuous functions from $X$ to $Y$\\\hline
$\mathcal{C}(X)$ & the space of continuous functions from $X$ to $\R$\\\hline
$\langle x\rangle$ & the fractional part of $x\in\mathbb{R}$\\\hline
$\lambda|_{\T}$ & the uniform measure on $\T$\\\hline
$\mu_Y$ & $\coloneqq\begin{cases}
  \lambda,& \text{if}\quad Y=X,\\
  \lambda|_{\T},& \text{if}\quad Y=\sT.
\end{cases}$\\\hline
$\delta_y$ & the Dirac measure at $y\in Y$\\\hline
 $\|\upsilon\|_{\TV}$ & $\sup_{f\in \mathcal{B}_1(Y)}\int f\rd\upsilon = \upsilon^+(Y)+\upsilon^-(Y),$\\\hline
$\|\upsilon\|_{\BL}$ & $\sup_{f\in\mathcal{BL}_1(Y)}\int_Y f\rd\upsilon$\\\hline
$\|\eta\|$& $\sup_{x\in X}\|\eta^x\|_{\TV}$\\\hline
{$\|\eta_{\cdot}\|_{\cI}$}&{$\sup_{t\in\mathcal{I}}\|\eta_t\|$}\\\hline
{$d_{\BL}(\upsilon_1,\upsilon_2)$} & {$\sup_{f\in \mathcal{BL}_1(Y)}\int fd(\upsilon_1-\upsilon_2)$}\\\hline
{$d_{\infty,\TV}(\eta,\xi)$} & $\sup_{x\in X}d_{\TV}(\eta^x,\xi^x)$\\\hline
{$d_{\infty,\BL}(\eta,\xi)$}& $\sup_{x\in X}d_{\BL}(\eta^x,\xi^x)$\\\hline
$d^{\cN}_{\infty,\BL}(\eta,\xi)$ & $\sup_{t\in\cN} d_{\BL,\infty}(\eta_t,\zeta_t)$\\\hline
$\|f\|_{\infty}$ & $\coloneqq\sup_{x\in Y}f(x)$ for $f\in\cB(Y)$\\\hline
$\Lip(f)$ & $\coloneqq\sup_{x,y\in Y,\ x\neq y}\frac{|f(x)-f(y)|}{d_Y(x,y)}$, the Lipschitz constant of $f\in\cC(Y)$\\\hline
$\BL(f)$ & $\coloneqq\Lip(f)+\|f\|_{\infty}$, the bounded Lipschitz constant of $f\in\cC(Y)$\\\hline
$\mathcal{L}^1_+(X)$ & $\coloneqq\{f\colon X\to\R\cup\{\pm\infty\}\colon \int_X|f|\rd\lambda<\infty,\ f(x)\ge0,\ \lambda\text{-a.e.}\ x\in X\}$\\\hline
$\mathcal{B}(X)$ & $\coloneqq\cB(X;\R)$\\\hline
$\mathcal{B}_1(X)$ & $\coloneqq\{f\in\mathcal{B}(X)\colon\|f\|_{\infty}\le1\}$\\\hline
$\mathcal{BL}_1(X)$ & $\coloneqq\{f\in\mathcal{C}(X)\colon\BL(f)\le1\}$\\\hline
$\mathcal{C}_{\sf b}(X,Y)$ & $\coloneqq\mathcal{C}(X,Y)\cap\mathcal{B}(X,Y)$ the space of bounded continuous functions
\\\hline
$\cB_{\infty}$ & $\coloneqq\{\xi\in\mathcal{B}(X,\cM_+(\T))\colon \int_{\T}\xi^x(\T)\rd x=1\}$\\\hline
$\cC_{\infty}$ & $\coloneqq\{\xi\in\mathcal{C}(X,\cM_+(\T))\colon \int_{\T}\xi^x(\T)\rd x=1\}$\\\hline
\end{longtable}

\end{document}